%% file: main.tex
\documentclass[11pt]{article} 

\input{header.tex}

\title{Nondeterministic behaviors in double categorical systems theory}

\author{Paul $\mathrm{Wang}$}

\date{}

\begin{document}

	\setcounter{tocdepth}{2}
	\maketitle

	

	\begin{abstract} 
		In this paper, we build double theories capturing the idea of nondeterministic behaviors and trajectories. Following Libkind and Myers' Double Operadic Theory of Systems  \cite{Libkind_Myers_2025}, we construct monoidal semi double categories of interfaces, along with what we call semimodules of systems, in the case of Moore machines, working with Markov categories with conditionals to handle nondeterminism. We use conditional products in these Markov categories to define trajectories in a compositional way, and represent nondeterministic systems using Markov maps; channels between systems are assumed to be deterministic. 
	\end{abstract}

\section{Introduction}
\emph{Compositional modelling} is based on the principle that one can, to some extent, represent and understand systems by first decomposing them into subsystems, then aggregating knowledge on the subsystems and their interactions. It is expected to be a useful foundation for \emph{collaborative modelling}; this principle underpins the projects ModelCollab \cite{ModelCollab} and CatColab \cite{CatColab}. Another use of compositionality lies in the \emph{modular design} of systems, building them by combining simpler components. A potential (still hypothetical at this stage) application of these practices is the training and deployment of safer AI systems \cite{SafeguardedAIThesis}.

One crucial observation is that there are many distinct notions for "systems": automata (deterministic or otherwise), (Partially Observable) Markov decision processes, Petri nets, systems driven by ODEs, PDEs, or SDEs, etc. 
Each such notion gives rise to a \textbf{systems theory}, which can be thought of as an ontology for systems.

The branch of applied category theory called \emph{categorical systems theory} aims at giving unified frameworks, based on categorical algebra, to describe \emph{composition operations} on systems, as well as \emph{comparisons maps}, or \emph{generalized simulations} between systems, and \emph{analyse the behaviors} of systems in terms of their components. Viewing systems theories as mathematical objects in their own right, one can then consider comparisons morphisms between them. For instance, as far as nondeterminism is concerned, probabilistic settings can often be "coarse-grained" into possibilistic ones, by forgetting the precise probabilities and only keeping track of the supports of the probability distributions involved.

As a matter of fact, there is a substantial body of work in categorical systems theory, for a more in-depth overview of these, we refer the reader to \cite[Section 1]{Libkind_Myers_2025}.
Many of these frameworks focus on one aspect of systems theory:

\begin{itemize}
	\item For instance, the \emph{symmetric monoidal} point of view, where systems are represented as morphisms $s: I \rightarrow O$ in symmetric monoidal categories, possibly with extra structure, where $I$ denotes the input object, and $O$ the output one; these systems can compose sequentially and in parallel, thanks to the symmetric monoidal structure. This common theme has been developed with many variations; see e.g. \cite{Broadbent_Karvonen_2023}, \cite{Lavore_Felice_Román_2025}, \cite{FRITZ-MarkovCats}, \cite{Coecke_Fritz_Spekkens_2016}. 
	
	\item There is also the \emph{operadic} point of view, which allows more involved composition patterns than just sequential and parallel ones, and also does not require interfaces of systems to be cleanly divided between inputs and outputs; see e.g. \cite{Libkind_Baas_Patterson_Fairbanks_2022}.
	
	\item On the other hand, various paradigms focus more on the notion of \emph{maps between systems}; a very prominent one is based on the idea of viewing systems as \emph{coalgebras} for endofunctors (see for instance \cite{Rutten_2000}, \cite{Kurz}). Then, morphisms of coalgebras represent \emph{simulations}, i.e. comparisons that preserve external/observable behavior.
\end{itemize}

Note that these are not disjoint: there are indeed works, such as \cite{Katis_Sabadini_Walters_1997} and \cite{Baez_Foley_Moeller_Pollard_2020}, that define structures combining composition of systems with maps between systems. However, many of these rely on $2$-categorical, or bicategorical, notions, which means that only systems maps with identical interfaces are considered.

One of the latest developments in this area is based on \emph{double category theory}. A key idea, as demonstrated by Myers \cite{DCST-book}, and then Libkind and Myers \cite{Libkind_Myers_2025}, is that \emph{wiring systems together} and \emph{comparing systems} are distinct operations, each of which ought to be encoded in a specific category. However, these operations are not independent: some compatibility conditions naturally appear. Also, one can be interested in systems maps that are not identities on the interfaces; this motivates the use of \emph{double} categories, instead of $2$-categories\footnote{Recall that, in $2$-categories, the $2$-cells can only exist between $1$-cells of the same signature; the notion of double category does not include this restriction.}. The general framework developed in \cite{Libkind_Myers_2025} is thus called the \emph{Double Operadic Theory of Systems}.

An open question is adapting the Double Operadic Theory of Systems to nondeterministic settings: although \cite{DCST-book} and \cite{Libkind_Myers_2025} contain examples of nondeterministic systems theories, there are issues with the notions of trajectories/system maps in those. See the discussion in the last three paragraphs of \cite[Section 3.5]{DCST-book}. Our goal is to give an answer to this question. There is a substantial body of work on probabilistic or nondeterministic  systems, for instance in the probabilistic programming semantics literature (see e.g. \cite{Liell-Cock_Staton_2025}, \cite{HigherOrderQBS}, \cite{Lavore_Felice_Román_2025}), but these results do not use the double categorical point of view mentioned above, and as such do not provide a complete answer to our question.

In this paper, we construct theories of nondeterministic systems (see Definitions \ref{defi_modules_systems} and \ref{defi_semi_modules_systems}) that fit into a double categorical framework, for a notion of systems based on Moore machines (see Definition \ref{defi_moore}). We describe the goals in \ref{subsection_goal}, and give a high-level description of our contributions in \ref{subsection_contrib}.

\subsection{Goal: double theories of nondeterministic systems and behaviors}\label{subsection_goal}
We wish to construct double theories of nondeterministic systems (see Definition \ref{defi_modules_systems}), where one direction (called the $y$ direction in this paper) is used to wire systems together and/or with lenses, and the other (the $x$ direction) is used to represent morphisms of systems, i.e. generalized simulations, such as trajectories, or coarse-graining maps, between state spaces or interfaces. Our key requirements are the following:

\begin{enumerate}
	\item Trajectories of systems should generalize the notion of stochastic process. In particular, we want them to encode \emph{joint distributions on states at various instants}, not merely distributions on states for each individual instant. See Subsection \ref{motiv_trajectories}.
	\item Some form of nondeterminism in the update maps of systems should be allowed.
	\item We want compositional trajectories for composite systems, that can be computed from (compatible) trajectories of the subsystems. In particular, \emph{we would rather keep track of all the sources of randomness and manage joint distributions explicitly}, than use a global probabilistic universe as a black box\footnote{This design choice is subjective, and against mainstream conventions in probability theory; we believe it is justified by the goal: to develop fully compositional modelling, one should not store the information on joint distributions inside a "black box sample space"}.
\end{enumerate}

\input{subs_contrib}

%
%
%

\subsection{Motivating example}\label{motiv_trajectories}

	In the deterministic case, the sequences of states, inputs, and outputs, is all one needs to know. In contrast, with nondeterminism, the \emph{joint distributions} contain more information than the sequences of distributions: for instance, the event ``the state of the system is never the same twice in a row'' is a property of pairs. In probabilistic settings, one cannot assign a probability to said event based only on the distributions on individual states. This phenomenon is not specific to probability theory: it appears in many nondeterministic settings, such as possibilism, imprecise probability, etc. Let's assume that nondeterminism is represented by a Markov category $\CC$ (see Definition \ref{defi_Markov_cat}). Then, distributions on pairs of states can be viewed as Markov morphisms $1 \rightarrow S \otimes S$, if $S$ is the states object, and $1$ the monoidal unit, usually a singleton set in the examples. In this context, our observation can be expressed more abstractly as follows: the function $\CC(1, S \otimes S) \rightarrow \CC(1, S) \times \CC(1, S)$ is usually not an injection.

	One way of addressing this issue is by working with \emph{distributions on trajectories}, such as elements of $\CC(1, S^{ \mathbb{N}})$, assuming an object $S^{ \mathbb{N}}$ exists. In technical terms, such an object is the (infinite) Kolmogorov product \cite[Definitions 3.1 and 4.1]{Fritz2020infiniteproducts} of the constant sequence $(S_i)_{i \in \mathbb{N}} = (S)_{i \in \mathbb{N}}$. There are technical issues to address with this approach, yet it can be made to work (see for instance \cite{HigherOrderQBS} and \cite{LazyPPL}). One approximation of that idea, which we shall follow here, is working with elements in $\CC(1, S^{\otimes n}) $ for all $n$, i.e. considering the projective limit of the $\CC(1, S^{\otimes n}) $. In favourable contexts (namely, when the Kolmogorov product does exist), this projective limit is actually isomorphic to $\CC(1, S^{\mathbb{N}})$. Even if it does not hold, one can hope for a canonical section $\lim\limits_n \, \CC(1, S^{\otimes n})  \rightarrow \CC(1, S^{\mathbb{N}})$ of the projection $\CC(1, S^{\mathbb{N}}) \rightarrow \lim\limits_n \, \CC(1, S^{\otimes n}) $. For a more general notion of time, one can consider a Directed Acyclic Graph $\GG$. This is what we shall do in this paper.

	Let us now describe what the definitions should yield, for Moore machines, in the case where $\GG$ is the graph of integers $(\mathbb{N}, (n \rightarrow n+1, \, n \in \mathbb{N})^{})$, and $\CC = BorelStoch$ is the Markov category of standard Borel spaces and measurable Markov kernels. The main point is that trajectories, i.e. behaviors whose source is the clock system, should take into account nondeterminism of states, inputs and outputs, for sequences of instants $(0, \ldots, n)$ for all $n$, not just for each instant individually.

	\begin{itemize}
		\item A system ${\begin{pmatrix} S \\ S  \\ \end{pmatrix}} \leftrightarrows {\begin{pmatrix} I \\ O  \\ \end{pmatrix}}$ should be a compatible family of lenses given by measurable functions $expose^n: S(n) \rightarrow O(n)$ and measurable probability kernels $update^n: S(n)  \times I(n+1) \rightarrow S(n+1)$ between standard Borel spaces. There should be (measurable) restriction functions $S(n+1) \rightarrow S(n)$, $I(n+1) \rightarrow I(n)$, and $O(n+1) \rightarrow O(n)$, compatible with the functions $expose^n$. Here, the objects $S(n)$, resp. $I(n)$, resp. $O(n)$, are standard Borel spaces meant to contain the information on states, resp. inputs, resp. outputs, from time $0$ to time $n$, and the restriction maps correspond to forgetting information on the last instant(s).

		For instance, we could have, $S(n) = S^{ n+1}$, $I(n+1) = I^{ n+1}$ and $O(n) = O^{ n+1}$, for all $n \geq 0$, with $S$, $I$, $O$ standard Borel spaces, and the system could represent an open Markov process: the map $expose^n$ is the ${n+1}^{th}$ cartesian power of a map $expose: S \rightarrow O$, and the map $update^n: S^{ n+1}  \times I^{ n+1} \rightarrow S^{ n+2}$ is given by $update^n(s_0, \ldots, s_{n}, i_0, \ldots, i_n) = (s_0, \ldots, s_{n}, update(s_{n}, i_n)) \in S^{n+2}$, where $update: S \times I \rightarrow S$ is the $1$-step update map of the open Markov process. 

		\item The system $Time: (*, {\begin{pmatrix} * \\ *  \\ \end{pmatrix}} \leftrightarrows {\begin{pmatrix} * \\ *  \\ \end{pmatrix}})$, which is just a clock, is the constant trivial lens: the $expose^n$ and $update^n$ functions are the unique map from the singleton set $*$ to itself. Note that it is different from the usual clock in discrete-time deterministic systems, since time is hardcoded in our theory.

		\item Then, a trajectory of a system ${\begin{pmatrix} S \\ S  \\ \end{pmatrix}} \leftrightarrows {\begin{pmatrix} I \\ O  \\ \end{pmatrix}}$, i.e. a morphism of systems whose source is the clock, can be represented as a square: 
		
		\begin{center}
			\begin{tikzcd}
				{\begin{pmatrix} * \\ *  \\ \end{pmatrix}}  \arrow[d, "", swap, shift right] \arrow[r, "", shift left]\arrow[r, shift right] & {\begin{pmatrix} S \\ S  \\ \end{pmatrix}}   \arrow[d, shift right]
				\\
				{\begin{pmatrix} * \\ *  \\ \end{pmatrix}} \arrow[u, shift right]\arrow[r, "", shift left] \arrow[r, shift right] & {\begin{pmatrix} I \\ O  \\ \end{pmatrix}}\arrow[u, "", shift right, swap]
			\end{tikzcd}
		\end{center}
		It should consist in Markov maps/probability distributions $p^n: * \rightarrow O(n) \times I(n+1)$, $\phi^n : * \rightarrow S(n) $, and $s^n: * \rightarrow S( n) \times  I(n+1)$, for all $n \in \GG$, that are compatible with $expose^n$ and $update^n$.
		
				We also want naturality in $n$, i.e. compatibility with the restriction morphisms.

				Informally, the map $p^n$ yields a joint distributions on outputs at times $(0, 1, \ldots, n)$ and inputs at times $(0, 1, \ldots, n, n+1)$, for all $n$. Similarly, the map $\phi^n : * \rightarrow S(n)$ encodes the joint distribution of the states at times $(0, 1, \ldots, n)$. Finally, the map $s^n$ \emph{should} yield \emph{the} joint distribution of the states at times $(0, \ldots, n)$ and of the inputs at times $(0, 1, \ldots, n+1)$. See point \ref{discussion_too_lax} in Section \ref{section_discussion} below.

			\end{itemize}

\input{sect_background.tex}

\section{Preliminaries}

\input{subs_markov_cats}

\subsection{Conventions}

When denoting objects in Symmetric Monoidal Categories, we may write expressions such as $AB$, or $A \,\, B$ instead of $A \otimes B$, in order to save space. We may also write expressions such as $AB \otimes CD$, or $AB \, \, CD$, for emphasis. The monoidal unit may be denoted $1$ or $*$. Regarding morphisms compositions, we use the following rules:

\begin{enumerate}
	\item The expression $f; g$ denotes $g \circ f$.
	\item The expression $f \otimes g; h \otimes k$ denotes $(f \otimes g) ; (h \otimes k)$. 
	\item (For Markov categories) We may write $\sigma: A B C D \rightarrow A C B D$, instead of $A \otimes \sigma_{B, C} \otimes D$. 
	\item (For Markov categories) Similarly, we may use $copy_B: A B C \rightarrow A B B C$ as shorthand for $A \otimes copy_B \otimes C$. 
\end{enumerate}
			
\input{subs_triple_cats}

\section{Theories of nondeterministic systems}

\input{subs_systems_theories}

%
%

\input{subs_arena_moore}

\input{subs_t_moore_c_g.tex}

\input{sect_discuss}


\input{sect_summary}

\textbf{Acknowledgements} 

Thanks to David Jaz Myers for much appreciated feedback on previous versions, useful comments and suggestions. Thanks to Owen Lynch for suggesting using graphs to deal with time. Thanks to Matteo Capucci for insightful remarks.

\printbibliography[
heading=bibintoc,
title={References}
]

\appendix
\appendices

\input{appendix_markov.tex}
\input{appendix_triple_cats.tex}

\section{Details of the main constructions}

\input{appendix_arena_moore.tex}

\input{appendix_arenasys_moore.tex}

\input{appendix_examples.tex}

\input{appendix_continuous_time.tex}

\input{appendix_conditional_independence.tex}

\input{appendix_behaviors_of_tensor_products.tex}

\end{document}

%% file: header.tex
\usepackage[T1]{fontenc}
\usepackage[english]{babel}

\usepackage[utf8]{inputenc} 

\usepackage{geometry} 
\usepackage{graphicx} 

\usepackage{booktabs} 
\usepackage{array} 
\usepackage{paralist} 
\usepackage{verbatim} 
\usepackage{subfig} 

\usepackage{csquotes}
\usepackage{graphicx} 
\usepackage{amssymb}
\usepackage{amsmath, amsfonts, bbm, dsfont}
\usepackage[all]{xy}
\usepackage{MnSymbol}
\usepackage{soul}

\usepackage{tikzit}
\input{tikz_style.tikzstyles}

\usepackage{fancyhdr} 
\usepackage{sectsty}
\allsectionsfont{\sffamily\mdseries\upshape} 

\usepackage[nottoc,notlof,notlot]{tocbibind} 
\usepackage[titles,subfigure]{tocloft} 

\usepackage[
backend=biber,
style=alphabetic,
giveninits=false,
doi=false,
url=true,
isbn=false
]{biblatex}
\renewbibmacro{in:}{}
\AtEveryBibitem{\clearfield{month}\clearlist{language}}

\usepackage{amsthm}
\usepackage{tikz-cd}
\usepackage{appendix}

\tikzcdset{row sep/normal= 0.8 cm,
column sep/normal= 0.8 cm}

\newtheorem{theo}{Theorem}[section]
\newtheorem{theorem}[theo]{Theorem}
\newtheorem{prop}[theo]{Proposition}
\newtheorem{lemma}[theo]{Lemma}

\newtheorem{claim}[theo]{Claim}

\theoremstyle{definition}
\newtheorem{fact}[theo]{Fact}

\newtheorem{definition}[theo]{Definition}
\newtheorem{notation}[theo]{Notation}

\newtheorem{remark}[theo]{Remark}

\newtheorem{ex}[theo]{Example}
\newtheorem{example}[theo]{Example}

\newtheorem{contre-ex}[theo]{Counter-example}

\newtheorem{construction}[theo]{Construction}

\usepackage{hyperref}
\hypersetup{
    colorlinks=true,
    linkcolor=black,
    filecolor=black,      
    urlcolor=black,
    citecolor=black,
}

\usepackage{tikz}
\usetikzlibrary{intersections}
\usepackage{quiver}

\newcommand{\CC}{\mathcal{C}}

\DeclareMathOperator{\GG}{\mathcal{G}}

\DeclareMathOperator{\Dd}{\mathbb{D}}

\DeclareMathOperator{\Cc}{\mathbb{C}}

\newcommand{\lens}[4]{\left(\begin{array}{c}#1 \\#2  \\\end{array}\right) \leftrightarrows \left(\begin{array}{c}#3 \\#4  \\\end{array}\right)}

\renewcommand{\phi}{\varphi}

%% file: tikz_style.tikzstyles
\tikzstyle{box}=[fill=white, draw=black, shape=rectangle, tikzit shape=rectangle, tikzit fill=white, tikzit draw=black]
\tikzstyle{round box}=[fill=white, draw=black, shape=circle, tikzit fill=white, tikzit draw=black, tikzit shape=circle]


%% file: subs_contrib.tex
\subsection{Contributions}\label{subsection_contrib}

We use the framework of \emph{Markov category with conditionals} (see Definition \ref{defi_Markov_cat}) to handle nondeterminism. If $\CC$ is a Markov category with conditionals, and $\GG$ a directed (acyclic) graph, we define a \emph{semi}module of systems (see Definition \ref{defi_semi_modules_systems}), in the spirit of modules of systems in \cite{Libkind_Myers_2025}, denoted $T^{\mathrm{Moore}}(\CC, \GG)$. Here, nondeterminism is handled by $\CC$, and time is represented via $\GG$. To that end, we first find \emph{semi} triple categories $Arena^{\mathrm{Moore}}_{\CC} \subseteq ArenaSys^{\mathrm{Moore}}_{\CC}$ (Theorems \ref{theo_Arena_C} and \ref{theo_ArenaSys_C}), where the extra dimension is meant for time-restriction, then define a \emph{semi}double category\footnote{See Definitions \ref{defi_semi_double_cat} and \ref{defi_semi_triple_cat} for semi triple and semi double categories.} $ArenaSys^{\mathrm{Moore}}_{\CC}(\GG)$ of interfaces and systems, for each $\GG$. Our main result can be stated as Theorem \ref{theo_main}, and its proof amounts to Construction \ref{constr_Arena_C^G}. Let us now explain how we meet each of the requirements:

\begin{enumerate}
	\item We think of trajectories as families of Markov morphisms that are compatible with deterministic time-restriction maps, i.e. compatible families of finite marginals. In our semi triple categories $Arena^{\mathrm{Moore}}_{\CC} \subseteq ArenaSys^{\mathrm{Moore}}_{\CC}$, the extra dimension/direction (called $z$) is used to accomodate the time-restriction, i.e. marginalization, maps. Then, our semi double category of systems and interfaces (Theorem \ref{theo_main}) is essentially a functor category $[\GG; ArenaSys^{\mathrm{Moore}}_{\CC}]$.
	
	\item Considering Moore machines, we allow arbitrary Markov maps to represent nondeterministic updates \emph{for systems, but not for lenses/channels between interfaces}: the latter are assumed to be deterministic. This is a technical restriction: it ensures associativity of vertical composition of squares in the resulting double category.

	\item Using conditional products, i.e. conditional independence assumptions, we \emph{create joint distributions} for trajectories of composite systems or lenses. See the definitions of $y$-composition of $xy$-squares in $Arena^{\mathrm{Moore}}_{\CC}$ and $ArenaSys^{\mathrm{Moore}}_{\CC}$ (points \ref{enum_xy_squares_ArenaC} and \ref{enum_xy_squares_ArenaSysC} in Constructions \ref{constr_Arena_C} and \ref{constr_ArenaSys_C} respectively), and point \ref{discussion_cond_indep} in the discussion in Section \ref{section_discussion}.
	
	\item To ensure that $xy$-interchange holds, we use copy-composition for $x$-composition of $x$-morphisms and $xy$-squares. The drawback is the lack of identities for $x$-composition, which is why we get \emph{semi} triple categories and \emph{semi} double categories.
	
\end{enumerate}

%% file: sect_background.tex
\section{Some background on categorical notions of systems}\label{section_classical}

In this section, we describe some of the main established categorical frameworks used to reason about systems, deterministic or otherwise. Double categorical systems theory can be viewed as trying to find a common generalization of these. 

\subsection{Symmetric Monoidal Categories}

The "symmetric monoidal" point of view on systems stems from the assumption that systems can be composed sequentially and in parallel, with natural compatibility conditions. Over the years, many applications of this idea have been found, in abstract cryptography \cite{Broadbent_Karvonen_2023}, quantum computing \cite{Danos_Kashefi_Panangaden_2007} \cite{Coecke_Kissinger_Gogioso_Dündar-Coecke_Puca_Yeh_Waseem_Pothos_Pfaendler_Wang-Mascianica_et_al._2025}, and of course probabilistic programming semantics \cite{HigherOrderQBS} \cite{LazyPPL}, etc.

For the remainder of this section, we fix a symmetric monoidal category $\CC$, whose symmetries are denoted by $\sigma_{A, B}: A \otimes B \rightarrow B \otimes A$.

\subsubsection{Moore machines}

\begin{definition}\label{defi_moore}
    \begin{enumerate}
	\item A Moore machine in $\CC$ is given by a triple of objects $(S, I, O)$ in $\CC$, and morphisms $S \rightarrow O$ and $S \otimes I \rightarrow S$. We say that $S$, $I$, $O$ are the states, inputs, and outputs objects, respectively.
	\item A lens in $\CC$ is given by a quadruple of objects $(I, O, I', O')$, and morphisms $O \rightarrow O'$ and $O \otimes I' \rightarrow I$.
\end{enumerate}
\end{definition}

Note that, with the definitions given above, Moore machines are special cases of lenses. 
We wish to view lenses $(I, O, I', O',  f: O \rightarrow O', f^{\sharp}: O \otimes I' \rightarrow I)$ as morphisms $(I, O) \rightarrow (I', O')$, which we denote $\lens{I}{O}{I'}{O'}$. To compose them, we need some extra structure on $\CC$, namely \emph{copying maps} $copy_X: X \rightarrow X \otimes X$: 

	Assume given, for each object $X \in \CC$, a map $copy_X: X \rightarrow X \otimes X$. Assume that copying is (co)associative and (co)commutative, i.e. $copy_X; (X \otimes copy_X) = copy_X \otimes (copy_X \otimes X)$ and $copy_X; \sigma_{X, X} = copy_X$ for all $X$.

\begin{definition}\label{def_lens_composition}
	Considering only lenses $(I, O, I', O', f, f^{\sharp})$ such that\footnote{The condition acn be thought of as determinism of the map $f$.} we have $f; copy_{O'} = copy_O ; (f \otimes f)$, we define the following composition:

	Given $(f, f^{\sharp}): \lens{I}{O}{I'}{O'}$ and $(g, g^{\sharp}): \lens{I'}{O'}{I''}{O''}$, we define the composite $(h, h^{\sharp}): \lens{I}{O}{I''}{O''}$, where $h= f; g$, and $h^{\sharp}$ is the following composite: $O \otimes I'' \xrightarrow{(copy_O; O \otimes f)\otimes I''} O \otimes O' \otimes  I'' \xrightarrow{O \otimes g^{\sharp}} O \otimes  I' \xrightarrow{f^{\sharp}} I$.
\end{definition}

We leave the proofs of the following Facts to the reader; they amount to computing in $\CC$.

\begin{fact}
	The composition operation of Definition \ref{def_lens_composition} is associative.
\end{fact}

\begin{fact}
	Assume given, for each object $X$ in $\CC$, a morphism $del_X: X \rightarrow 1$, which makes $(X, del_X, copy_X)$ a commutative comonoid in $\CC$. Assume that $del$ is \emph{natural}, i.e. $f; del_Y = del_X$ for all $X$, $Y$, $f: X \rightarrow Y$.
	
	Then, the composition operation of Definition \ref{def_lens_composition} defines a category, i.e. admits identities. Moreover, this category has a natural symmetric monoidal structure.
\end{fact}

\subsubsection{Mealy machines}\label{subsubsection_Mealy}
Here, we describe a more general notion of systems, namely \emph{Mealy machines}. These can be defined without needing copy or discard maps. Recall that we have fixed a Symmetric Monoidal Category $\CC$.

\begin{definition}
	\begin{itemize}
		\item 	Given objects $X$, $Y$, a Mealy machine $X \xrightarrow{} Y$ is given by an object $S$, along with a morphism $f: S \otimes X \rightarrow S \otimes Y$. We may denote it $f: X \xrightarrow{S} Y$. 
		\item  Mealy machines compose sequentially: let $f: X \xrightarrow{S} Y$ and $g: Y \xrightarrow{T} Z$. Then, the composite $(f; g): X \xrightarrow{S \otimes T} Z$ is given by the following composite map: $S \otimes T \otimes X \xrightarrow{\sigma \otimes X; T \otimes f} T \otimes S \otimes Y \xrightarrow{\sigma \otimes Y; S \otimes g} S \otimes T \otimes Z$.
	\end{itemize}
\end{definition}

Given a Mealy machine $f: X \xrightarrow{S} Y$, we think of the objects $S$, $X$, $Y$, as representing the spaces of states, inputs, and outputs, respectively.

\begin{fact}
	There is a symmetric monoidal category $\mathrm{Mealy}(\CC)$ of Mealy machines in $\CC$, with a symmetric monoidal faithful functor $\CC \rightarrow  \mathrm{Mealy}(\CC)$, acting as identity on objects, and mapping $f: X \rightarrow Y$ in $\CC$ to $f: X \xrightarrow{1_{\CC}} Y$ in $\mathrm{Mealy}(\CC)$.
\end{fact}

In fact, one can enrich that, and get the following:

\begin{fact}
    There exists a symmetric monoidal $2$-category of Mealy machines in $\CC$, whose underlying symmetric monoidal $1$-category is $\mathrm{Mealy}(\CC)$, and such that a $2$-morphism from $f: X \xrightarrow{S} Y$ to $g: X \xrightarrow{T} Y$ is given by a morphism $\phi: S \rightarrow T$ in $\CC$, such that the following diagram commutes:
\begin{center}
\[\begin{tikzcd}
	{S \otimes X} && {T \otimes X} \\
	{S \otimes Y} && {T \otimes Y}
	\arrow["{\phi \otimes X}", from=1-1, to=1-3]
	\arrow["f"{description}, from=1-1, to=2-1]
	\arrow["g"{description}, from=1-3, to=2-3]
	\arrow["{\phi \otimes Y}", from=2-1, to=2-3]
\end{tikzcd}\]
\end{center}

\end{fact}

\begin{remark}
    As explained in \cite[Subsection 1.2]{Libkind_Myers_2025}, this symmetric monoidal $2$-categorical approach can handle maps between systems as well as systems compositions (sequential and parallel), but only those with identical interfaces.
\end{remark}

Several variations on this idea have been studied; the most relevant to our work is that of \emph{coinductive streams} \cite{Lavore_Felice_Román_2025}, which are, up to suitable equivalence relations, sequences of maps $(M_{n-1} \otimes X_n \rightarrow M_n \otimes Y_n)_{n \in \mathbb{N}}$, where $M_n$ represents the memory at time $n$, and $X_n$ and $Y_n$ represent, for time $n$, input and output objects respectively.

\subsection{Systems as coalgebras}

Here, we give a brief introduction to the \emph{coalgebraic point of view on systems}. For more on this subject, see for instance \cite{Rutten_2000} \cite{Kurz} \cite{Staton_2009}  \cite{Geuvers_Jacobs_2021} \cite{Smithe_2023}.

\begin{definition}
    Let $\CC$ be a category, and $F: \CC \rightarrow \CC$ an endofunctor. 
    
    \begin{enumerate}
    \item A \emph{coalgebra} for $F$ is given by an object $X$, along with a morphism $d: X \rightarrow FX$.
    \item A morphism of coalgebras $\phi: (d: X \rightarrow FX) \rightarrow (e: Y \rightarrow FY)$ is a morphism $\phi: X \rightarrow Y$ such that the following square commutes:
    
    \begin{center}
\[\begin{tikzcd}
	X && Y \\
	FX && FY
	\arrow["\phi", from=1-1, to=1-3]
	\arrow["d"{description}, from=1-1, to=2-1]
	\arrow["e"{description}, from=1-3, to=2-3]
	\arrow["{F\phi}", from=2-1, to=2-3]
\end{tikzcd}\]
    \end{center}
    \end{enumerate}
\end{definition}

\begin{example}
	\begin{enumerate}
		\item     If $\CC$ is cartesian closed, then for any pair $(I, O)$ of objects in $\CC$, there is a category of Moore machines with interface $(I, O)$, defined as the category of coalgebras for the endofunctor $X \mapsto O \times X^I$.
		\item     If $\CC$ is cartesian closed, then for any pair $(I, O)$ of objects in $\CC$, there is a category of Mealy machines with interface $(I, O)$, defined as the category of coalgebras for the endofunctor $X \mapsto {(O \times X)}^I$.

	\end{enumerate}
\end{example}

\begin{remark}
	\begin{itemize}
		\item 	A morphism of coalgebras can be thought of as a \emph{simulation}, i.e. map between states \emph{that acts identically on the interfaces}, in a way that is compatible with the dynamics of the two systems. One can then consider a more symmetric notion of \emph{bisimulations} (see for instance \cite{Staton_2009}), which are, essentially, jointly monic spans of coalgebras.
		\item One of the motivations for the double categorical point of view is to get rid of this restriction that "simulations act identically on interfaces", by considering more general notions of morphispms of systems. Interestingly, recent work in the coalgebraic literature \cite{Nora_Rot_Schroder_Wild_2025} goes in a similar direction, defining a notion of \emph{heterogeneous (bi)simulations}.

	\end{itemize}
\end{remark}

%% file: subs_markov_cats.tex
\subsection{Markov categories}

Markov categories are used in \emph{synthetic probability theory}, where the aim is to give \emph{abstract high-level algebraic axioms} for the behavior of Markov kernels, which are measurable maps between measurable spaces that output probability distributions, instead of points. 
The goal is to be able to study objects, define concepts, and prove general results, without having to spend as much time handling low-level set-theoretic details (sigma-algebras, and so on). For more background, we refer the reader to Fritz's seminal paper \cite{FRITZ-MarkovCats}, which develops the theory substantially; its introduction gives an account of the motivations.

		\begin{definition}\label{defi_Markov_cat}[See \cite{FRITZ-MarkovCats}, Definition 2.1]
			\begin{itemize}
				\item A Markov category is a symmetric monoidal category whose unit $I$ is a terminal object, and such that each object $X$ is equipped with a commutative comonoid structure $(copy_X: X \rightarrow X \otimes X, \, del_x: X \rightarrow I)$, in a way that is uniform with respect to the monoidal structure.

				Note that, since the monoidal unit $I$ is terminal, there are maps $X \otimes Y \rightarrow X \otimes I \simeq X$ and $X \otimes Y \rightarrow Y$ for all $X, Y \in \CC$. We think of these as projections, or marginalizing maps.
				\item The Markov category $\CC$ has conditionals if, for all morphisms $\phi: A \rightarrow X \otimes Y$, there exists a morphism $\phi|X : A \otimes X \rightarrow Y$ such that $\phi$ equals the following composite:

				\ctikzfig{cond_def}
				
			\end{itemize}

		\end{definition}

		\begin{remark}
			From now on, to simplify notations, we shall assume that the Markov categories we consider are strict, i.e. tensor product is strictly unital and strictly associative. Since any Markov category is comonoid equivalent to a strict Markov category (see \cite[Theorem 10.17]{FRITZ-MarkovCats}), this restriction is not essential.
		\end{remark}
		
		Let us fix a Markov category $\CC$.

\begin{definition}
	Let $p: I \rightarrow X \otimes Y \otimes Z$ be a morphism in $\CC$. We say that \emph{it displays conditional independence} of $X$ and $Z$ over $Y$ if there exist morphisms $f: Y \rightarrow X$ and $g: Y \rightarrow Z$ such that $p = (p ; \pi_Y ; copy_Y ; Y \otimes copy_Y; f \otimes Y \otimes g)$.
\end{definition}
					\ctikzfig{cond_indep_def}
		
		\begin{fact}[See Definition 12.8 and Proposition 12.9 in \cite{FRITZ-MarkovCats}]
			If $\CC$ has conditionals, given morphisms $f: A \rightarrow X \otimes Y$ and $g: A \rightarrow Y \otimes Z$ which have the same marginal $h: A \rightarrow Y$, there exists a unique morphism $f \otimes_Y g: A \rightarrow X \otimes Y \otimes Z$, called the \emph{conditional product} of $f$ and $g$ over $Y$, with the following properties:
			
			\begin{itemize}
				\item The marginals of $f \otimes_Y g$ on $X \otimes Y$ and $Y \otimes Z$ are $f$ and $g$ respectively.
				\item The morphism $f \otimes_Y g$ displays conditional independence over $Y$ of its marginals on $X$ and $Z$.
			\end{itemize}
			
			This conditional product can be defined as follows: 
			
			\ctikzfig{cond_product}
			
			One can check that this does not depend on the choice of $f|Y$ and $g|Y$, only on $f$ and $g$.
			
		\end{fact}

\begin{definition}
	In the Markov category $\CC$, a morphism is called deterministic if it commutes with the $copy$ maps.
\end{definition}

\begin{ex}
	
	\begin{enumerate}
		\item Any cartesian category can be viewed as a Markov category where all morphisms are deterministic.
		\item The Kleisli category of any symmetric affine monoidal monad on a Markov category, is again a Markov category \cite[Corollary 3.2]{FRITZ-MarkovCats}.

	For instance, the following affine monoidal monads yield Markov categories:
	
	\begin{itemize}
		\item  The monad $P$ of nonempty subsets, on $Set$.

		\item  The Giry monad $G$ on measurable spaces. See for instance \cite[Section 4]{Giry}.
	\end{itemize}
	
	The Markov category associated to the Giry monad is usually denoted $Stoch$. Objects are measurable spaces, and morphisms are measurable kernels, for suitable $\sigma$-algebras. 
	
	\item The full subcategory $FinStoch \subseteq Stoch$, whose objects are \emph{finite} discrete measurable spaces.

	\item The full subcategory $BorelStoch \subseteq Stoch$, whose objects are standard Borel spaces, i.e. measurable spaces that are either discrete finite, discrete countable, or isomorphic to $\mathbb{R}$, is canonically a Markov category. 
	
	\end{enumerate}
\end{ex}

The Markov categories $Kl(P)$, $FinStoch$, $BorelStoch$ have conditionals. For $BorelStoch$, see \cite[Theorem 5]{faden1985existence}.

\begin{fact}[\cite{FRITZ-MarkovCats}, Remark 10.13]
	
The wide subcategory $\CC_{det} \subseteq \CC$, of deterministic morphisms in $\CC$, is a cartesian category that contains all structure maps.
\end{fact}

Given an ambient Markov category $\CC$, morphisms in $\CC$ will sometimes be called \emph{Markov morphisms}, and morphisms in $\CC_{det}$ will be referred to as \emph{deterministic morphisms}.

\begin{notation}\label{notation_indep}
	Given objects $X$, $Y$, $Z$ in $\CC$ and a morphism $p: I \rightarrow X \otimes Y \otimes Z$, and assuming the context is clear, we let $x$, $y$, $z$ denote the generalized elements $I \rightarrow X$, $I \rightarrow Y$ and $I \rightarrow Z$.

	Then, we write $x \bot_y^p z$, or $x \bot_y z$ if $p$ is clear from the context, if the distribution $p: I \rightarrow X \otimes Y \otimes Z$ displays conditional independence of $X$ and $Z$ over $Y$. We extend the notations $x \bot^p_y z$ and $x \bot_y z$ to the case where $p: I \rightarrow X \otimes Y \otimes Z \otimes T$, for any object $T \in \CC$.

	Let us now extend the notations even further: if we are also given \emph{deterministic} maps $f: X\otimes Y \otimes Z \rightarrow U$ or $g: X\otimes Y \otimes Z \rightarrow V$, or $h: X\otimes Y \otimes Z \rightarrow W$, we may write conditional independence relations such as $ u \bot_{ v}^{p, f, g, h} w$ or (abusing notations) $u \bot^p_v w$, or even $u \bot_{v} w$, to express conditional independence for the distribution $(p ; copy_{X \otimes Y \otimes Z} ; (X\otimes Y \otimes Z \otimes copy_{X \otimes Y \otimes Z}); (f \otimes g \otimes h)): I \rightarrow U \otimes V \otimes W$. We might also write expressions such as $u, x \bot_v w, z$, or $ux \bot_{v y} w z$, etc., with the natural interpretation.
\end{notation}

%% file: subs_triple_cats.tex
\subsection{Triple categories}

\subsubsection{Strict double and triple categories}

Let us first recall what we mean by "strict double categories". The reader may also refer to e.g. \cite[Introduction]{Dawson_Pare_1993}.

\begin{definition}
	A strict double category is given by:
	
	\begin{itemize}
		\item A collection of objects.
		\item A pair of categories on this collection of objects, which we may call the $x$- and $y$- categories.

		\item For each suitable pair of pairs of morphisms (arranged as the boundary of a square), a collection of squares, with (strictly) associative and unital $x$- and $y$- composition operations on adjacent squares.  We require $xy$ interchange equalities, for all suitable quadruples of squares. Squares shall also be called $xy$-morphisms.
	\end{itemize}
\end{definition}

Equivalently, a small strict double category is a category strictly internal to the category of small categories.

\begin{definition}
	A strict double functor, or morphism of (small) strict double categories, is a functor internal to the category of (small) categories. It preserves identities and composition strictly, in both directions.
\end{definition}

Let us now give definitions for strict triple categories. These were introduced by Ehresmann \cite{Ehresmann_1963}.

\begin{definition}

	A strict triple category is given by:
	
	\begin{itemize}
		\item A collection of objects.
		\item A triple of categories on this collection of objects, which  we call the $x$, $y$, and $z$ category respectively.
		\item A triple of strict double categories, suggestively named the $xy$, $yz$, and $xz$ categories, whose underlying $1$-dimensional categories are the $x$, $y$, $z$ categories, as appropriate. For instance, the underlying $1$-categories of the $xy$ category are the $x$ category and the $y$ category.
		\item For each suitable sextuple of squares (arranged as the boundary of a cube), a collection of cubes,  with associative and unital $x$-, $y$-, and $z$- composition operations on adjacent cubes.  We require $xy$-, $yz$-, and $xz$- interchange equalities, for all suitable quadruples of cubes. Cubes shall also be called $xyz$-morphisms.
		
	\end{itemize}

	In other words, a (small) strict triple category is a strict category internal to the category of (small) strict double categories and double functors.

\end{definition}

\begin{definition}
	A strict triple functor between triple categories maps objects to objects, and $*$-morphisms to $*$-morphisms, for $* \in \lbrace x, y, z, xy, yz, xz, xyz \rbrace$, sending identities to identities, and being strictly compatible with all compositions.
\end{definition}	

\begin{definition}\label{defi_hom_triple_cat}
	Let $\mathbb{C}$, $\mathbb{D}$ be strict triple categories. We define the triple category $[\mathbb{C}, \mathbb{D}]$ of strict triple functors, and $x$-, $y$-, $z$-, $xy$-, $yz$-, $xz$-, $xyz$-natural transformations as follows:
	\begin{itemize}
		\item Let $F, G: \mathbb{C} \rightarrow \mathbb{D}$ be triple functors. 
 An $x$-natural transformation $\alpha: F \rightarrow G$
 
 \begin{enumerate}
 	\item maps objects $c$ of $\mathbb{C}$ to $x$-morphisms $\alpha(c): F(c) \rightarrow G(c)$, naturally in the $x$ direction.
 	\item maps $y$ morphisms, resp. $z$ morphisms, of $\mathbb{C}$ to $xy$ morphisms, resp. $xz$ morphisms, of $\mathbb{D}$, with prescribed boundaries, functorially in the $y$ direction, resp. in the $z$ direction, with the usual naturality condition for $xy$-squares of $\mathbb{C}$, resp. $xz$-squares of $\mathbb{C}$.
 	\item maps $yz$ squares $s$ of $\mathbb{C}$ to $xyz$ cubes $\alpha(s)$ in $\mathbb{D}$, having prescribed boundaries, double functorially in the $yz$ square $s$, and with a naturality condition for $xyz$ cubes of $\mathbb{C}$.
\end{enumerate}

We define $y$ natural transformations and $z$ natural transformations in a similar way.
		
		\item Let $F_1, G_1, F_2, G_2 : \mathbb{C} \rightarrow \mathbb{D}$ be triple functors, let $\alpha_i: F_i \rightarrow G_i$ be $x$-natural transformations, $i=1,2$, and let $\beta_1 : F_1 \rightarrow F_2$, $\beta_2: G_1 \rightarrow G_2$ be $y$-natural transformations. An $xy$ natural transformation $\gamma$ with boundary

			\begin{tikzcd}
				F_1 \arrow[r, "\alpha_1"] \arrow[d, "\beta_1"]& G_1 \arrow[d, "\beta_2"]
				\\
				F_2 \arrow[r, "\alpha_2"]& G_2
			\end{tikzcd}
 is defined by the following data:

\begin{enumerate}
	\item It maps objects of $\mathbb{C}$ to $xy$ squares in $\mathbb{D}$ with suitable boundaries, satisfying naturality conditions for $x$ morphisms and $y$ morphisms in $\mathbb{C}$ similar to those for modifications.
	\item It maps $z$ morphisms of $\mathbb{C}$ to $xyz$ cubes in $\mathbb{D}$ with the required boundaries, functorially in the $z$ direction.
\end{enumerate}

		\item Given eight triple functors, four $x$-natural transformations, four $y$-natural transformations, four $z$-natural transformations, two $xy$-natural transformations, two $xz$-natural transformations, and two $yz$-natural transformations, arranged as the boundary of a cube, an $xyz$ natural transformation with this boundary maps objects of $\mathbb{C}$ to $xyz$ morphisms in $\mathbb{D}$ with suitable boundary, satisfying naturality conditions with respect to $x$-morphisms, $y$-morphisms and $z$-morphisms of $\mathbb{C}$. 

	\end{itemize}

\end{definition}

\begin{definition}
	Let $\CC$ be a $1$-category. For $* \in \lbrace x, y, z \rbrace$, we let $*(\CC)$ denote the triple category concentrated in the $*$-direction, with only trivial squares and cubes, built from $\CC$. Similarly, any strict double category can be viewed as a strict triple category, in (at least) three different ways.
	
	Conversely, if $\mathbb{C}$ is a strict triple category, we let $\mathbb{C}_{x}$ denote the $1$-category whose morphisms are $x$-morphisms in $\mathbb{C}$. Similarly, $\mathbb{C}_{xy}$ denotes the double category obtained by forgetting the $xz$, $yz$ and $xyz$ morphisms in $\mathbb{C}$. We also let $\mathbb{C}^{*}(a, b)$, or $Hom_{\mathbb{C}}^*(a,b)$, denote the collection of $*$-morphisms from $a$ to $b$, for $* \in \lbrace x, y, z \rbrace$, if $a$ and $b$ are objects of $\mathbb{C}$.
\end{definition}

\subsubsection{Analogues without identities}
Now, let us give slightly more general definitions, relaxing some conditions. Recall that a \emph{semicategory} is defined similarly to a category, except that identities are not assumed to exist.

\begin{definition}\label{defi_semi_double_cat}
	An $x$-semi double category is given by:
	
	\begin{itemize}
		\item A collection of objects.
		\item A pair of semi-categories on this collection of objects, which we may call the $x$- and $y$- semicategories, such that the $y$-semicategory is a category.

		\item For each suitable pair of pairs of morphisms (arranged as the boundary of a square), a collection of squares, with (strictly) associative $x$- and $y$- composition operations on adjacent squares.  We require that all $x$-morphisms have corresponding identity squares for $y$-composition. We also require $xy$ interchange equalities, for all suitable quadruples of squares. Squares shall also be called $xy$-morphisms.
	\end{itemize}
\end{definition}

\emph{In other words, a (small) strict x-semi double category is a category strictly internal to the complete category SemiCat of (small) semicategories}. Such a structure corresponds to the following data:

\begin{itemize}
	\item Objects $\CC_0$ and $\CC_1$ in $\mathcal{SemiCat}$, corresponding to the semicategory of objects and $x$-morphisms, and the semicategory of $y$-morphisms and $xy$-squares, respectively.
	\item \emph{Semifunctors} $\mathrm{dom}$ and  $\mathrm{cod}$, from $\CC_1$ to $\CC_0$, mapping $y$-morphisms to their source, resp. target, and mapping $xy$-squares, viewed here as morphisms between $y$-morphisms, to their source, resp. target.
	\item A \emph{semifunctor} $1: \CC_0 \rightarrow \CC_1$ that maps each object to its identity $y$-morphism, and each $x$-morphism to its identity $xy$-square.

	\item A composition \emph{semifunctor} $F: \CC_1 \times_{\CC_0} \CC_1 \rightarrow \CC_1$, acting on the objects of $\CC_1$ as $y$-composition of $y$-morphisms, and on the morphisms of $\CC_1$ as $y$-composition of $xy$-squares. Semifunctoriality corresponds to $xy$-interchange.
\end{itemize}

The conditions are strict associativity of $F$, the fact that the semifunctor $1: \CC_0 \rightarrow \CC_1$ is a section of $\mathrm{dom}$ and $\mathrm{cod}$, and unitality of it with respect to $F$.

\begin{definition}
	A morphism of $x$-semi double categories, or $x$-semi double functor, or semi double functor when the context is clear, is a semi double functor which is unital in the $y$ direction, i.e. an internal functor in the category $SemiCat$ of (small) semicategories.
\end{definition}

Now, let us do the same for triple categories:

\begin{definition}\label{defi_semi_triple_cat}
	A (small) strict $x$-semi triple category is a category strictly internal to the (complete) category of (small) $x$-semi double categories and $x$-semi double functors.
\end{definition}

Such a structure corresponds to the following data:

\begin{itemize}
	\item Objects $\CC_0$ and $\CC_1$ in $\mathcal{SemiCat}$, corresponding to the $x$-semi double category of objects, $x$-morphisms, $y$-morphisms and $xy$-squares, and the $x$-semi double category of $z$-morphisms, $xz$-squares, $yz$-squares, and $xyz$ cubes, respectively.
	\item \emph{Semi double functors} $\mathrm{dom}$ and  $\mathrm{cod}$, from $\CC_1$ to $\CC_0$, mapping $z$-morphisms, $xz$-morphisms, $yz$-morphisms and $xyz$-morphissms to their source, resp. target, in the $z$ direction.
	\item A \emph{semi double functor} $1: \CC_0 \rightarrow \CC_1$ that maps each object to its identity $z$-morphism, each $x$-morphism to its identity $xz$-square, each $y$-morphism to its identity $yz$-square, and each $xy$-square to its identity $xyz$-cube.

	\item A composition \emph{semi double functor} $F: \CC_1 \times_{\CC_0} \CC_1 \rightarrow \CC_1$, corresponding to compositions in the $z$ direction. Semi double functoriality corresponds to $xz$- and $yz$- interchange for squares and cubes.
\end{itemize}

The conditions are, as before, strict associativity of $F$, the fact that the semi double functor $1: \CC_0 \rightarrow \CC_1$ is a section of $\mathrm{dom}$ and $\mathrm{cod}$, and unitality of it with respect to $F$.

\begin{definition}
	A strict $x$-semi triple functor is a functor internal to the category of (small) $x$-semi double categories and $x$-semi double functors.
\end{definition}

We can then consider natural transformations between such semi triple functors.

\begin{fact}
	Given $x$-semi triple categories $\mathbb{C}$ and $\mathbb{D}$, there exists an $x$-semi triple category $[\mathbb{C}, \mathbb{D}]$ whose objects are $x$-semi triple functors, and whose morphisms of dimension $1$, $2$ and $3$ are natural transformations, in a sense similar to Definition \ref{defi_hom_triple_cat} (more precisely, to the complete Definition \ref{defi_nat_transfo_full}), \emph{mutatis mutandis}.

	This construction is functorial, in the sense that it defines a functor $[\cdot , \cdot]: T^{op} \times T \rightarrow T$, where $T$ denotes the complete $1$-category of (small) $x$-semi triple categories and $x$-semi triple functors. 
\end{fact}

%% file: subs_systems_theories.tex
\subsection{Modules of systems}

The following Definition was taken from \cite[Explication 4.8]{Libkind_Myers_2025}.

\begin{definition}\label{defi_modules_systems}
 A module of systems $\mathbb {S} \colon  \bullet  \mathrel {\mkern 3mu\vcenter {\hbox {$\shortmid $}}\mkern -10mu{\to }} \mathbb {I}$ consists of:

    \begin{itemize}
        \item A symmetric monoidal double category of interfaces, interface maps, and  interactions $\mathbb{I}$.
        \item A symmetric monoidal category of systems $\mathsf{Car}(\mathsf{S})$.
        \item A symmetric monoidal functor assigning systems to their interface, $\mathsf{Car}(\mathsf{S}) \rightarrow \mathsf{Tight}(\mathbb{I})$.
        \item An action of interactions on systems \(\mathsf {Loose}({\mathbb {I}}) \curvearrowright   \mathsf {Car}(\mathsf {S})\) that respects the interface.
    \end{itemize}

\end{definition}

\begin{remark}
    With our conventions and terminology, the "loose" morphisms are in the $y$ direction, and the "tight" ones are in the $x$ direction.
\end{remark}

\begin{remark}
    The intuition behind Definition \ref{defi_modules_systems} is the following:

    \begin{itemize}
        \item The symmetric monoidal double category $\mathbb{I}$ contains:
            \begin{itemize}
                \item Interfaces as objects.
                \item Interface maps/representations between interfaces, as $y$-morphisms/loose morphisms.
                \item Interactions/channels between interfaces, as $x$-morphisms/tight morphisms.
                \item Representations between channels/interaction maps, as squares.
            \end{itemize}
        \item The symmetric monoidal category $\mathsf{Car}(\mathsf{S})$ contains:
        \begin{itemize}
            \item Systems as objects.
            \item System maps/representations between systems, as morphisms.
        \end{itemize}
        \item For each system, its interface is defined as its image under the functor $\mathsf{Car}(\mathsf{S}) \rightarrow \mathsf{Tight}(\mathbb{I})$.
        \item The composite of a system with an interaction/channel, is again a system.

        \item The symmetric monoidal structures correspond to "parallel product" operations, either on systems, interactions, or maps between those.
    \end{itemize}
\end{remark}

\begin{example}
    Let $2$ denote the $1$-category with objects $0$ and $1$, and a single nonindentity arrow $0 \rightarrow 1$. Let $y(2)$ denote the double category built from $2$, with $2$ objects as well, concentrated in the $y$-direction, with only identity squares.

    Let $\mathbb{D}$ be a double category, and $H: \mathbb{D} \rightarrow y(2)$ be a double functor. Assume that:

    \begin{itemize}
        \item The double category $\mathbb{I}:= H^{-1}(1)$ is equipped with a symmetric monoidal structure.
        \item The category $H^{-1}(0 \rightarrow 1)$, whose objects are arrows above $0 \rightarrow 1$, and whose morphisms are squares in $\mathbb{D}$, is equipped with a symmetric monoidal structure, such that $\mathrm{cod}: H^{-1}(0 \rightarrow 1) \rightarrow {H^{-1}(1)}_x = \mathsf {Loose}({\mathbb {I}})$ is a strict symmetric monoidal functor.
    \end{itemize}

    Then, the data $\mathsf {Car}(\mathsf {S}):= H^{-1}(0 \rightarrow 1)$, $\mathrm{cod}: \mathsf{Car}(\mathsf{S}) \rightarrow \mathsf{Tight}(\mathbb{I})$, and the action \(\mathsf {Loose}({\mathbb {I}}) \curvearrowright   \mathsf {Car}(\mathsf {S})\) defined by composition in $\mathbb{D}$, yields a module of systems.
\end{example}

In our context, the constructions do not yield such objects, we only get semicategories, etc. So, let us extend the framework slightly:

\begin{definition}\label{defi_semi_modules_systems}
 A \emph{semimodule of systems} $\mathbb {S} \colon  \bullet  \mathrel {\mkern 3mu\vcenter {\hbox {$\shortmid $}}\mkern -10mu{\to }} \mathbb {I}$ consists of:

    \begin{itemize}
        \item A symmetric monoidal $x$-semi double category of interfaces, interface maps, and  interactions $\mathbb{I}$.
        \item A symmetric monoidal semicategory of systems $\mathsf{Car}(\mathsf{S})$.
        \item A symmetric monoidal semifunctor assigning systems to their interface, $\mathsf{Car}(\mathsf{S}) \rightarrow \mathsf{Tight}(\mathbb{I})$.
        \item An action of interactions on systems \(\mathsf {Loose}({\mathbb {I}}) \curvearrowright   \mathsf {Car}(\mathsf {S})\) that respects the interface.
    \end{itemize}

\end{definition}

    




%% file: subs_arena_moore.tex
\subsection{The triple category $Arena^{\mathrm{Moore}}_{\CC}$}

Let $\CC$ be a Markov category with conditionals. We wish to construct an $x$-semi triple category $Arena^{\mathrm{Moore}}_{\CC}$, where:

\begin{itemize}
	\item The objects are interfaces, i.e. pairs of objects of $\CC$.
	\item The $x$ morphisms are nondeterministic \emph{copy-composition} charts.
	\item The $y$ morphisms are deterministic lenses.

	\item The $z$ morphisms are pairs of deterministic maps.
\end{itemize}
 This one is intended as an intermediate step. Later on, given a notion of time given by a graph $\GG$, we shall build an $x$-semi double category $ArenaSys^{\mathrm{Moore}}_{\CC}(\GG)$, which will be our candidate for a ``theory of $\GG$-time nondeterministic systems''.

\begin{theorem}\label{theo_Arena_C}
	There exists a symmetric monoidal $x$-semi triple category $Arena^{\mathrm{Moore}}_{\CC}$ with the following properties:
	
	\begin{enumerate}
		\item Objects are pairs $\begin{pmatrix} a \\ c  \\ \end{pmatrix}$, where $a$ and $c$ are objects of $\CC$.
		\item An $x$-morphism ${\begin{pmatrix} a_1 \\ c_1  \\ \end{pmatrix}} \rightrightarrows {\begin{pmatrix} a_2 \\ c_2  \\ \end{pmatrix}}$ is given by an object ${\begin{pmatrix} a_{12} \\ c_{12}  \\ \end{pmatrix}}$ and a pair of Markov morphisms $g: c_1 \rightarrow c_{12} \otimes c_2$, $g^{\flat}: c_1 \otimes a_1 \rightarrow c_{12} \otimes c_2 \otimes a_{12} \otimes a_2$ such that $g^{\flat} ; \pi_{c_{12} \otimes c_2} = \pi_{c_1}; g$.
		The $x$-morphisms compose using \emph{copy-composition}.
		
		\item The $y$-morphisms are deterministic lenses ${\begin{pmatrix} a_1 \\ c_1  \\ \end{pmatrix}} \leftrightarrows {\begin{pmatrix} a_2 \\ c_2  \\ \end{pmatrix}}$, i.e. pairs of deterministic maps $f: c_1 \rightarrow c_2$, $f^{\sharp}: c_1 \otimes a_2 \rightarrow a_1$.
		
		\item The $z$-morphisms ${\begin{pmatrix} a_1 \\ c_1  \\ \end{pmatrix}} \rightrightarrows {\begin{pmatrix} a_2 \\ c_2  \\ \end{pmatrix}}$, are pairs of \emph{deterministic} morphisms $f: c_1 \rightarrow c_2$, $g: a_1 \rightarrow a_2$, in $\CC$.

		\item The $xy$-morphisms with boundary as below
		\begin{center}
\[\begin{tikzcd}
	\begin{array}{c} {\begin{pmatrix} a_1 \\ c_1  \\ \end{pmatrix}} \end{array} && \begin{array}{c} {\begin{pmatrix} a_2 \\ c_2  \\ \end{pmatrix}} \end{array} \\
	\begin{array}{c} {\begin{pmatrix} a_3 \\ c_3  \\ \end{pmatrix}} \end{array} && \begin{array}{c} {\begin{pmatrix} a_4 \\ c_4  \\ \end{pmatrix}} \end{array}
	\arrow["{(a_{12}, c_{12}, f_{12}, f^{\flat}_{12})}", shift left=2, from=1-1, to=1-3]
	\arrow[shift right=2, from=1-1, to=1-3]
	\arrow["{(f_{13}, f_{13}^{\flat})}"', shift right=2, from=1-1, to=2-1]
	\arrow[shift right=2, from=1-3, to=2-3]
	\arrow[shift right=2, from=2-1, to=1-1]
	\arrow[shift left=2, from=2-1, to=2-3]
	\arrow["{(a_{34}, c_{34}, f_{34}, f^{\flat}_{34})}"', shift right=2, from=2-1, to=2-3]
	\arrow["{(f_{24}, f_{24}^{\flat})}"', shift right=2, from=2-3, to=1-3]
\end{tikzcd}\]
		\end{center}
		
		are given by $y$-morphisms $( f_{1234}, f_{1234}^{\sharp}): \lens{a_{12}}{c_{12}}{a_{34}}{c_{34}}$, along with Markov maps $s_{}:  c_1 \otimes a_3 \rightarrow     c_{12} \otimes c_2 \otimes a_{34} \otimes a_4$ such that the following equations hold:
		\begin{enumerate}
			\item $f_{12}; (f_{1234} \otimes f_{24}) = f_{13} ; f_{34}$
			\item $(s ; \pi_{c_{12} \otimes c_2 \otimes a_{34} \otimes a_4}; f_{1234} \otimes f_{24} \otimes a_4) = (\pi_{c_1 \otimes a_3} ; f_{13} \otimes a_3 ; f_{34}^{\flat})$
			\item $(copy_{c_1}; \sigma; f_{13}^{\sharp}; f_{12}^{\flat}) = (s; copy_{c_{12} \otimes c_2}; f_{1234}^{\sharp} \otimes f_{24}^{\sharp})$
		\end{enumerate}
		
		Composition in the $x$-direction is copy-composition in $\CC$. Composition in the $y$-direction uses conditionals products in $\CC$.
		
		\item  The $yz$-morphisms with boundary as below
		
		\begin{center}
			\begin{tikzcd}
				{\begin{pmatrix} a_1 \\ c_1  \\ \end{pmatrix}}  \arrow[d, "{( f_{13}, f_{13}^{\sharp})}", swap, shift right] \arrow[r, "{(f_{12}, g_{12})}", shift left]\arrow[r, shift right] & {\begin{pmatrix} a_2 \\ c_2  \\ \end{pmatrix}}   \arrow[d, shift right]
				\\
				{\begin{pmatrix} a_3 \\ c_3  \\ \end{pmatrix}} \arrow[u, shift right]\arrow[r, "{(f_{34}, g_{34})}", shift left] \arrow[r, shift right] & {\begin{pmatrix} a_4 \\ c_4  \\ \end{pmatrix}}\arrow[u, "{( f_{24}, f_{24}^{\sharp})}", shift right, swap]
			\end{tikzcd}
		\end{center} 
		
		are given by the equation $f_{13}; f_{34} = f_{12}; f_{24}$. Compositions are just concatenations of commuting squares in $\CC_{det}$.
		
		\item The $xz$-morphisms with boundary as below

		\begin{center}

\[\begin{tikzcd}
	\begin{array}{c} {\begin{pmatrix} a_1 \\ c_1  \\ \end{pmatrix}} \end{array} && \begin{array}{c} {\begin{pmatrix} a_2 \\ c_2  \\ \end{pmatrix}} \end{array} \\
	\begin{array}{c} {\begin{pmatrix} a_3 \\ c_3  \\ \end{pmatrix}} \end{array} && \begin{array}{c} {\begin{pmatrix} a_4 \\ c_4  \\ \end{pmatrix}} \end{array}
	\arrow["{(a_{12}, c_{12}, f_{12}, f^{\flat}_{12})}", shift left=2, from=1-1, to=1-3]
	\arrow[shift right=2, from=1-1, to=1-3]
	\arrow["{(f_{13}, g_{13})}"', shift right=2, from=1-1, to=2-1]
	\arrow[shift left=2, from=1-1, to=2-1]
	\arrow[shift right=2, from=1-3, to=2-3]
	\arrow["{(f_{24}, g_{24})}", shift left=2, from=1-3, to=2-3]
	\arrow[shift left=2, from=2-1, to=2-3]
	\arrow["{(a_{34}, c_{34}, f_{34}, f^{\flat}_{34})}"', shift right=2, from=2-1, to=2-3]
\end{tikzcd}\]
\end{center}

are given by pairs of deterministic maps $f_{1234}: c_{12} \rightarrow c_{34}$, $g_{1234}: a_{12} \rightarrow a_{34}$ such that the following equations hold: 

\begin{enumerate}
	\item $(f_{12}; f_{1234} \otimes f_{24}) = f_{13}; f_{34}$
	\item $(f_{12}^{\flat}; f_{1234} \otimes f_{24} \otimes g_{1234} \otimes g_{24}) = (f_{13} \otimes g_{13}); f_{34}^{\flat}$

\end{enumerate}
		
		\item There is exactly one $xyz$-morphism, i.e. cube, for each boundary. 
		
	\end{enumerate}
\end{theorem}

\begin{proof}[Proof notes.]
		The most technical points lie in defining $y$-composition of $xy$-squares, proving that it is associative, and showing that $xy$-interchange holds for $xy$-squares. See Construction \ref{constr_Arena_C}.
\end{proof}


Now, we wish to extend $Arena^{\mathrm{Moore}}_{\CC}$ by adding systems, and morphisms of systems, as new kinds of morphisms, with the following requirements: in the $y$ direction, a system can only be composed with a $y$-morphism in $Arena^{\mathrm{Moore}}_{\CC}$, and, in the $x$ and $z$ directions, morphisms of systems can only be composed with other morphisms of systems.

Let $2$ denote the category with objects $0, 1$, and only one non-identity arrow $0 \rightarrow 1$. Recall that $y(2)$ is the triple category constructed from the category $2$, concentrated in the $y$-direction.

\begin{theorem}\label{theo_ArenaSys_C}
There exists an $x$-semi triple category $ArenaSys^{\mathrm{Moore}}_{\CC}$ and an ($x$-semi) triple functor $F: ArenaSys^{\mathrm{Moore}}_{\CC} \rightarrow y(2)$ such that the fiber $F^{-1}(1)$ is isomorphic to $Arena^{\mathrm{Moore}}_{\CC}$. It has the following properties:
\begin{enumerate}
	\item Objects in $F^{-1}(0)$ are deterministic morphisms $r: \widetilde{S} \rightarrow S$, also denoted ${\begin{pmatrix} \widetilde{S} \\ S  \\ \end{pmatrix}}$.
	\item The $x$-morphisms in $F^{-1}(0)$ are morphisms in the category of arrows of $\CC$.
	\item The $z$-morphisms in $F^{-1}(0)$ are morphisms in the category of arrows of $\CC_{det}$.
	\item The $y$-morphisms in the fiber $F^{-1}(0 \rightarrow 1)$ are nondeterministic lenses $(f, f^{\sharp}): {\begin{pmatrix} \widetilde{S} \\ S  \\ \end{pmatrix}} \leftrightarrows {\begin{pmatrix} a_1 \\ c_1  \\ \end{pmatrix}}$ with deterministic output map, such that the composite $S \otimes a_1 \xrightarrow{f^{\sharp}} \widetilde{S} \xrightarrow{r} S$ is equal to the projection $\pi_S$.
	
	\item The $yz$-squares from $z$-morphisms in $F^{-1}(0)$ to $z$-morphisms in $F^{-1}(1)$:	
	
	\begin{center}
	\begin{tikzcd}
		{\begin{pmatrix} \widetilde{S}_1 \\ S_1  \\ \end{pmatrix}}  \arrow[d, "{( f_{13}, f_{13}^{\sharp})}", swap, shift right] \arrow[r, "{(f_{12}, g_{12})}", shift left]\arrow[r, shift right] & {\begin{pmatrix} \widetilde{S}_2 \\ S_2  \\ \end{pmatrix}}   \arrow[d, shift right]
		\\
		{\begin{pmatrix} a_3 \\ c_3  \\ \end{pmatrix}} \arrow[u, shift right]\arrow[r, "{(f_{34}, g_{34})}", shift left] \arrow[r, shift right] & {\begin{pmatrix} a_4 \\ c_4  \\ \end{pmatrix}}\arrow[u, "{( f_{24}, f_{24}^{\sharp})}", shift right, swap]
	\end{tikzcd}
	\end{center}
	are given by the equality $f_{13}; f_{34} = f_{12}; f_{24}$, i.e. commutativity of certain squares in $\CC_{det}$. Compositions are given by concatenations of commuting squares.
	
	\item The $xz$-semi double category of the fiber $F^{-1}(0) \subseteq ArenaSys^{\mathrm{Moore}}_{\CC}$ is thin, with squares given by compatibility conditions.
	
	\item The $xy$-squares 	with boundary as below
	
	\begin{center}
		\begin{tikzcd}
			{\begin{pmatrix} \widetilde{S}_1 \\ S_1  \\ \end{pmatrix}}  \arrow[d, "{( f_{13}, f_{13}^{\sharp})}", swap, shift right] \arrow[rr, "{(f_{12}, f^{\flat}_{12})}", shift left]\arrow[rr, shift right] && {\begin{pmatrix} \widetilde{S}_2 \\ S_2  \\ \end{pmatrix}}   \arrow[d, shift right]
			\\
			{\begin{pmatrix} a_3 \\ c_3  \\ \end{pmatrix}} \arrow[u, shift right]\arrow[rr, "{(a_{34}, c_{34}, f_{34}, f^{\flat}_{34})}", shift left] \arrow[rr, shift right] && {\begin{pmatrix} a_4 \\ c_4  \\ \end{pmatrix}}\arrow[u, "{( f_{24}, f_{24}^{\sharp})}", shift right, swap]
		\end{tikzcd}
	\end{center} are Markov morphisms $s_{}:   c_1 \otimes a_3 \rightarrow   S_2 \otimes c_{34} \otimes a_{34} \otimes a_4$ such that the following equations hold:
	\begin{enumerate}
		\item $f_{12}; f_{24} = f_{13} ; f_{34}; \pi_{c_4}$
		\item $(s ; \pi_{S_2 \otimes c_{34} \otimes a_{34} \otimes a_4}; f_{24} \otimes c_{34} \otimes a_{34} \otimes a_4) = (\pi_{S_1 \otimes a_3} ; f_{13} \otimes a_3 ; f_{34}^{\flat})$
		\item $(copy_{c_1}; \sigma; f_{13}^{\sharp}; f_{12}^{\flat}) = (s; \pi_{  S_2 \otimes a_4}; f_{24}^{\sharp})$
	\end{enumerate}
	
	Compositions of $xy$-squares in $ArenaSys^{\mathrm{Moore}}_{\CC}$ are defined with similar ideas to those used for $Arena^{\mathrm{Moore}}_{\CC}$, i.e. $x$-composition uses copy-composition in $\CC$, and $y$-composition relies on conditional products.
	
	\item There is exactly one cube for each boundary.
\end{enumerate}

Moreover, consider the $x$-semi double category $\mathbb{D} = F^{-1}(0 \rightarrow 1)$, whose objects are $y$-morphisms over the arrow $0 \rightarrow 1$, whose $x$, resp. $z$, morphisms are $xy$, resp. $xz$, squares in $ArenaSys^{\mathrm{Moore}}_{\CC}$, and whose $xz$-morphisms are $xyz$-cubes in $ArenaSys^{\mathrm{Moore}}_{\CC}$. Then, there exists a symmetric monoidal structure on this $x$-semi double category $\mathbb{D}$, such that the codomain $x$-semi double functor $\mathrm{cod}: \mathbb{D} \rightarrow ArenaSys^{\mathrm{Moore}}_{\CC, xz}$ is symmetric monoidal.
\end{theorem}

\begin{proof}[Proof notes.]
	As above, the most technical points lie in defining $y$-composition of $xy$-squares, proving associativity, and interchange. See Construction \ref{constr_ArenaSys_C}.
\end{proof}

%% file: subs_t_moore_c_g.tex
\subsection{Dealing with trajectories: the systems theory $T^{\mathrm{Moore}}(\CC, \GG)$}

Our goal is to define, for each Directed Acyclic Graph $\GG$ (representing a notion of time) and each Markov category with conditionals $\CC$, a systems theory $T^{\mathrm{Moore}}(\CC, \GG)$, using the triple categories $ArenaSys^{\mathrm{Moore}}_{\CC} \supseteq  Arena^{\mathrm{Moore}}_{\CC}$.

\begin{notation}\label{notation_graphs}
Let $\GG$ be a directed acyclic graph. 

\begin{itemize}
	\item Abusing notations slightly, the free category generated by the graph $\GG$ shall be denoted by $\GG$ as well.
	\item Consider the graph $Ar(\GG)$ whose vertices are edges in $\GG$, where there is exactly one edge from $a \xrightarrow{e} b$ to $b' \xrightarrow{e'} c$ if $b=b'$, and no edge otherwise. As above, the graph and the free category it generates will be denoted $Ar(\GG)$.

	Note that there are functors $dom, cod: Ar(\GG) \rightarrow \GG$ mapping an edge in $\GG$ to its source, resp. target, and sending an edge $(a \rightarrow b) \Rightarrow (b \rightarrow c)$ in $Ar(\GG)$ to the edge $a \rightarrow b$, resp. $b \rightarrow c$.

	\item Taking duals of graphs, we also have a functor $cod: Ar(\GG^{op}) \rightarrow \GG^{op}$, mapping an edge of $\GG^{op}$ to its target i.e. to the source of the corresponding edge in $\GG$, and sending an edge $(c \rightsquigarrow b) \Rightarrow (b \rightsquigarrow a)$ in $Ar(\GG^{op})$, where $a \rightarrow b$ and $b \rightarrow c$ are edges in $\GG$, to the edge $b \rightsquigarrow a$ in $\GG^{op}$. Similarly, we have a functor $dom: Ar(\GG^{op}) \rightarrow \GG^{op}$.
\end{itemize}  

\end{notation}

For instance, we are interested in the example where the vertices of $\GG$ are natural numbers, and there is one edge $n \rightarrow n+1$ for all $n$.

Recall that $2$ denotes the category with two objects $0, 1$ and one nonidentity arrow $0 \rightarrow 1$, and that, any category $\mathcal{A}$ induces a triple category $z(\mathcal{A})$ concentrated in the $z$ direction.

\begin{theorem}\label{theo_main}
For each directed (acyclic) graph $\GG$ and each Markov category with conditionals $\CC$, there exists a semimodule of systems\footnote{See Definition \ref{defi_semi_modules_systems}.} $T^{\mathrm{Moore}}(\CC, \GG)$ containing \emph{at least} the systems and trajectories described in Subsection \ref{motiv_trajectories}.

\end{theorem}

Let us now prove Theorem \ref{theo_main}, by constructing the systems theory $T^{\mathrm{Moore}}(\CC, \GG)$. The idea is to construct the $x$-semi double category $ArenaSys^{\mathrm{Moore}}_{\CC}(\GG)$ as a sub-double category of the $x$-semi double category $[z(Ar(\GG^{op}));ArenaSys^{\mathrm{Moore}}_{\CC}]_{xy}$ of $x$-semi triple functors, $x$-natural transformations, $y$-natural transformations, and $xy$-natural transformations. The details can be found in Appendix \ref{appendix_Arena_Moore} and \ref{appendix_ArenaSys_Moore}.

\begin{construction}\label{constr_Arena_C^G}
	We construct an $x$-semi double category $ArenaSys^{\mathrm{Moore}}_{\CC}(\GG)$ with an $x$-semi  double functor $H: ArenaSys^{\mathrm{Moore}}_{\CC}(\GG) \rightarrow y(2)$, as follows:
	
	\begin{enumerate}
		\item The objects in $H^{-1}(0)$ are functors $\widetilde{S}: \GG^{op} \rightarrow \CC_{det}$. Note that any such $\widetilde{S}$ induces a triple functor $z(Ar(\GG^{op})) \rightarrow F^{-1}(0) \subseteq ArenaSys^{\mathrm{Moore}}_{\CC}$, mapping an edge $a \xrightarrow{e} b$ in $\GG$ to the object ${\begin{pmatrix} \widetilde{S}(b) \\ \widetilde{S}(a)  \\ \end{pmatrix}}$ with structure map $\widetilde{S}(e^{op}) : \widetilde{S}(b) \rightarrow \widetilde{S}(a)$, and mapping an arrow $(b \xrightarrow{e} c) \Rightarrow (a \xrightarrow{e'} b)$ in $Ar(\GG^{op})$ to the $z$-morphism in $ArenaSys^{\mathrm{Moore}}_{\CC}$ given by the deterministic maps $\widetilde{S}(e^{op})$, $\widetilde{S}({e'}^{op})$.
		We shall use the same notations for the functors $\GG^{op} \rightarrow \CC_{det}$ and for the triple functors $z(Ar(\GG^{op})) \rightarrow Arena^{\mathrm{Moore}}_{\CC}$.
		
		\item The objects in $H^{-1}(1)$ are pairs of functors $(a, c)$, where $a,c: \GG^{op} \rightarrow \CC_{det}$. They induce $x$-semi triple functors $z(Ar(\GG^{op})) \rightarrow Arena^{\mathrm{Moore}}_{\CC}$ via pulling back along the triple functors $z(cod), z(dom): z(Ar(\GG^{op})) \rightarrow z(\GG^{op})$. More explicitly, the $x$-semi triple functor induced by a given $(a, c)$ is ${\begin{pmatrix} z(a \circ dom) \\ z(c \circ cod)  \\ \end{pmatrix}}$: each arrow $w \rightsquigarrow v$ in $\GG^{op}$ is mapped to ${\begin{pmatrix} a(w) \\ c(v)  \\ \end{pmatrix}} \in Arena^{\mathrm{Moore}}_{\CC}$, and each arrow $(w \rightsquigarrow v) \Rightarrow (v \rightsquigarrow u)$ in $Ar(\GG^{op})$ is mapped to the $z$ morphism ${\begin{pmatrix} a(w) \\ c(v)  \\ \end{pmatrix}} \rightrightarrows {\begin{pmatrix} a(v) \\ c(u)  \\ \end{pmatrix}}$ in $Arena^{\mathrm{Moore}}_{\CC}$ obtained by applying the functors $a$ and $c$ to the arrows $w \rightsquigarrow v$ and $v \rightsquigarrow u$ respectively.
		We shall use the same notations for the pairs of functors $\GG^{op} \rightarrow \CC_{det}$ and for the triple functors $z(Ar(\GG^{op})) \rightarrow Arena^{\mathrm{Moore}}_{\CC}$ that they induce.
		
		\item Horizontal morphisms $\widetilde{S}_1 \rightarrow \widetilde{S}_2$, or $x$-morphisms, in the fiber $H^{-1}(0)$, are natural transformations $\widetilde{S}_1 \Rightarrow \widetilde{S}_2$ between functors from $\GG^{op}$ to $\CC$. These induce $x$-natural transformations between the $x$-semi triple functors $\widetilde{S}_1, \widetilde{S}_2: z(Ar(\GG^{op})) \rightarrow ArenaSys^{\mathrm{Moore}}_{\CC}$.
		
		\item In the fiber $H^{-1}(1)$, horizontal morphisms, or $x$-morphisms ${\begin{pmatrix} a_1 \\ c_1  \\ \end{pmatrix}} \rightarrow {\begin{pmatrix} a_2 \\ c_2  \\ \end{pmatrix}}$, are given by objects ${\begin{pmatrix} a_{12} \\ c_{12}  \\ \end{pmatrix}}$, along with pairs of natural transformations $a_1 \Rightarrow a_{12} \otimes a_2$, $c_1 \Rightarrow c_{12} \otimes c_2$, between functors $\GG^{op} \rightarrow \CC$. They induce $x$-natural transformations between $x$-semi triple functors $z(Ar(\GG^{op})) \rightarrow Arena^{\mathrm{Moore}}_{\CC}$.

		\item Vertical morphisms, or $y$-morphisms, are $y$-natural transformations. Let us be more explicit. In the fiber $H^{-1}(0)$, there are only identity $y$-morphisms. In the fiber $H^{-1}(1)$, a $y$-morphism ${\begin{pmatrix} a_1 \\ c_1  \\ \end{pmatrix}} \leftrightarrows {\begin{pmatrix} a_2 \\ c_2  \\ \end{pmatrix}}$ is a $y$-natural transformation between $x$-semi triple functors $z(Ar(\GG^{op})) \rightarrow ArenaSys^{\mathrm{Moore}}_{\CC}$, whose bottom part is induced by a natural transformation $c_2 \Rightarrow c_1$ in $[\GG^{op}; \, \CC_{det}]$. In other words, the set $Hom^{y}_{ArenaSys^{\mathrm{Moore}}_{\CC}(\GG)}({\begin{pmatrix} a_1 \\ c_1  \\ \end{pmatrix}}; {\begin{pmatrix} a_2 \\ c_2  \\ \end{pmatrix}})$ is equal to	$$Hom^{y}_{[z(Ar(\GG^{op})) ; Arena^{\mathrm{Moore}}_{\CC}]}({\begin{pmatrix} a_1 \\ c_1  \\ \end{pmatrix}}; {\begin{pmatrix} a_2 \\ c_2  \\ \end{pmatrix}}) \times_{Hom_{[Ar(\GG^{op}) ;\, \CC_{det}]}(c_1 ; c_2)} Hom_{[\GG^{op} ; \, \CC_{det}]}(c_1 ; c_2).$$

		Finally, in the fiber $H^{-1}(0 \rightarrow 1)$, a $y$-morphism $\widetilde{S} \leftrightarrows {\begin{pmatrix} a_1 \\ c_1  \\ \end{pmatrix}}$ is given by a $y$-natural transformation $\widetilde{S} \rightarrow {\begin{pmatrix} a_1 \\ c_1  \\ \end{pmatrix}}$ between $x$-semi triple functors $z(Ar(\GG^{op})) \rightarrow ArenaSys^{\mathrm{Moore}}_{\CC}$ and a natural transformation $f: \widetilde{S} \Rightarrow c_1$ in $[\GG^{op}; \, \CC_{det}]$, such that the bottom part of the $y$-natural transformation $\widetilde{S} \rightarrow {\begin{pmatrix} a_1 \\ c_1  \\ \end{pmatrix}}$ is induced by $f$. In other words, the set $Hom^{y}_{ArenaSys^{\mathrm{Moore}}_{\CC}(\GG)}(\widetilde{S}; {\begin{pmatrix} a_1 \\ c_1  \\ \end{pmatrix}})$ is equal to  the fiber product $$Hom^{y}_{[z(Ar(\GG^{op})) ; ArenaSys^{\mathrm{Moore}}_{\CC}]}(\widetilde{S}; {\begin{pmatrix} a_1 \\ c_1  \\ \end{pmatrix}}) \times_{Hom_{[Ar(\GG^{op}) ;\, \CC_{det}]}(\widetilde{S} ; c_1)} Hom_{[\GG^{op} ; \, \CC_{det}]}(\widetilde{S} ; c_1).$$
		
		Composition is defined as composition of $y$-natural transformations.
		%
		%
		
		\item Squares are $xy$-natural transformations in $[z(Ar(\GG^{op})), ArenaSys^{\mathrm{Moore}}_{\CC}]$. The only $xy$-squares in $H^{-1}(0)$ are identity squares of $x$-morphisms. Horizontal composition of squares is given by $x$-composition of $xy$-natural transformations, i.e. pointwise $x$-composition of $xy$-squares, and of $xyz$-cubes, in $ArenaSys^{\mathrm{Moore}}_{\CC}$. Similarly, vertical composition of squares is given by $y$-composition of $xy$-natural transformations, i.e. pointwise $y$-composition of $xy$-squares, and of $xyz$-cubes, in $ArenaSys^{\mathrm{Moore}}_{\CC}$.

	\end{enumerate}	
\end{construction}

\begin{definition}
	The systems theory $T^{\mathrm{Moore}}(\CC, \GG)$ is defined as $(ArenaSys^{\mathrm{Moore}}_{\CC}(\GG), H)$, where $H: ArenaSys^{\mathrm{Moore}}_{\CC}(\GG) \rightarrow y(2)$ is the ()$x$-semi) double functor defined in Construction \ref{constr_Arena_C^G}.
\end{definition}

\begin{remark}
	\begin{enumerate}
		\item One key difference with the theories defined in \cite[Section 3.5]{DCST-book} is that squares/morphisms of systems \emph{contain data}, namely joint distributions, as opposed to only \emph{witnessing compatibility conditions} between data in their boundaries. While we feel this is justified by the context and the end result, this change leads to more technical computations, such as managing interchange of squares. 
		\item One can check that the resulting double category is not spanlike (see \cite[Definition 5.3.1.5]{DCST-book}) in the examples of interest.
		\item We only get $x$-semi double categories, due to the absence of strict identities for copy-composition.

	\end{enumerate}
\end{remark}

%% file: sect_discuss.tex
\section{Discussion}\label{section_discussion}

	
	\begin{enumerate}
		\item\label{discussion_too_lax} Our construction implements \emph{almost all the constraints} described in Subsection \ref{motiv_trajectories}. The only missing point is naturality in $n$ of the maps $s^n: * \rightarrow  S( n) \otimes  I(n+1)$, i.e. compatibility with time-restriction. We conjecture that this condition can be implemented \emph{a posteriori} using a more restrictive notion for cubes in $ArenaSys^{\mathrm{Moore}}_{\CC}$; finding suitable hypotheses, or sufficient conditions to ensure such compatibility, for instance a version of the Markov property, remains future work. Note that the issue here is that \emph{the theory allows too many trajectories and behaviors to be defined}; it seems likely that the objects actually considered in practice always satisfy the naturality conditions. 
		
		\item \label{discussion_cond_indep} The use of conditional products for vertical composition of squares implements a \emph{modelling hypothesis}: Consider the case where $\GG = \mathbb{N}$, with edges $n \rightarrow n+1$, and a system ${\begin{pmatrix} S \\ S \\ \end{pmatrix}} \leftrightarrows {\begin{pmatrix} I_1 \\ O_1  \\ \end{pmatrix}}$, a (deterministic) lens $ {\begin{pmatrix} I_1 \\ O_1 \\ \end{pmatrix}} \leftrightarrows {\begin{pmatrix} I_2 \\ O_2  \\ \end{pmatrix}}$, a trajectory of the system, and compatible charts for the wiring of interfaces, i.e. nondeterministic tuples/generalized elements as below, for all $n$:
		
		\begin{itemize}
			\item 	$(s(n), i_1(n+1)) \in S(n) \otimes I_1(n+1)$
			
			\item 	 $(o_1(n), i_1(n+1)) \in O_1(n) \otimes I_1(n+1)$
			
			\item 	 $(o_2(n), i_2(n+1)) \in O_2(n) \otimes I_2(n+1)$
			
			\item $(o_1(n), i_2(n+1)) \in O_1(n) \otimes I_2(n+1)$ 
		\end{itemize}
		
		These are assumed to be compatible with the lens ${\begin{pmatrix} I_1(n+1) \\ O_1(n) \\ \end{pmatrix}} \leftrightarrows {\begin{pmatrix} I_2(n+1) \\ O_2(n)  \\ \end{pmatrix}}$, and with the system.
		
		Then, \emph{the} trajectory of the composite system $  {\begin{pmatrix} S \\ S \\ \end{pmatrix}} \leftrightarrows {\begin{pmatrix} I_2 \\ O_2  \\ \end{pmatrix}}$ is \emph{defined} so that, for all $n$, the states $s(n) \in S(n)$ is \emph{independent} from the nondeterministic inputs $i_2(n+1) \in I_2(n+1)$ \emph{conditional on the tuple} $(o_1(n), i_1(n+1)) \in O_1(n) \otimes I_1(n+1)$.	In other words, it is assumed that, in this context, the only mutual information between $s(n)$ and $i_2(n+1)$ lies in $(o_1(n), i_1(n+1))$. This \emph{modelling assumption} enables us to create joint distributions for composite trajectories, in a doubly functorial way, since we get a double category in the end. We believe this assumption is realistic.

		As explained in the previous point, at this level of generality, the construction does not guarantee coherence in time/compatibility with restrictions of the joint distributions on $S(n) \otimes I_2(n+1)$, even assuming coherence for the joint distributions on $S(n) \otimes I_1(n+1)$ and $O_1(n) otimes I_2(n+1)$. One can actually find simple counterexamples, for instance if some, but not all, of the $I_1(n+1)$ and $O_1(n)$ are terminal objects and $S$, $I_2$ are constant functors.
		
		\item 	Using parametric deterministic lenses enables us to represent nondeterministic channels, by explicitly naming the sources of nondeterminism for the update maps. \emph{One might want to allow channels to be represented as nondeterministic Markov morphisms, just like systems are; there are technical issues to deal with}, because one then has to find a way of managing all the implicit correlations/joint distributions at play. In fact, naively extending our definitions to allow nondeterministic lenses breaks both associativity of $y$-composition for $xy$-squares and $xy$-interchange. 
		
		\item Using Directed Acyclic graphs allows us to move beyond discrete-time: for instance, using the poset of nonnegative real numbers, one can represent some notion of continuous-time behaviors. 
		A key observation is that topological spaces of continuous or smooth paths into well-behaved spaces actually yield standard Borel spaces. 
		See Appendix \ref{appendix_continuous_time}.

		\item The constructions are functorial in the graph $\GG$, and also in the Markov category with conditionals $\CC$; regarding the latter point, morphisms between Markov categories with conditionals are (strict) monoidal functors that preserve copying, discarding, and conditional products. This functoriality could be used to represent discretization (of time) and coarsening (of the notion of nondeterminism), respectively. We leave the exposition of these facts to future work.
	\end{enumerate}

%% file: sect_summary.tex
\section{Summary and outlook}\label{section_outlook}
From a Markov category with conditionals $\CC$, we constructed $x$-semi triple categories $Arena^{\mathrm{Moore}}_{\CC}$ and $ArenaSys^{\mathrm{Moore}}_{\CC}$, whose objects are interfaces, resp. interfaces and state-objects of systems (see Theorems \ref{theo_Arena_C} and \ref{theo_ArenaSys_C}). In these constructions, the $z$ direction is used to deal with (deterministic) time-restriction maps, and $y$-composition of $xy$-squares uses conditional products to \emph{create joint distributions}. From these $x$-semi triple categories, we then built double systems theories $T^{\mathrm{Moore}}(\CC, \GG)$, i.e. semimodules of systems, for each Directed Acyclic Graph $\GG$, where $\GG$ is meant to represent a notion of time (Theorem \ref{theo_main}). While there are still a few technical questions to answer, building the ($x$-semi) triple categories $Arena^{\mathrm{Moore}}_{\CC} \subseteq ArenaSys^{\mathrm{Moore}}_{\CC}$, or variants thereof, looks like an important step.

Several research directions seem to emerge from our work.
\begin{enumerate}
	\item Finding sufficient conditions to ensure time-coherence of composite trajectories, and hypotheses/constraints that are natural, from the point of view of modelling, on trajectories of systems or on cubes in $ArenaSys^{\mathrm{Moore}}_{\CC}$, would help \emph{narrow down the class of objects} and refine the construction. 
	
	Another possibility is the existence of a deeper interpretation of some time-incoherence phenomena, possibly having to do with statistical inference, approximation procedures given finite-dimensional data, etc.
	
	\item\label{outlook_trajectories_tensors} Along the same lines, one can wish to \emph{define/construct} joint trajectories of subsystems of a system, given trajectories of each subsystem with compatible behaviors at the interface, using conditional products again. Studying the properties of this operation using the language of double categorical systems theory seems relevant to the theory of compositional modelling\footnote{One possible application is providing sound baseline assumptions for combining nondeterministic models; this issue does not appear in purely deterministic contexts.}.
	
	\item Finding a higher-level/more abstract description of the constructions could help internalize them, or adapt them more efficiently to other contexts, e.g. enriched ones.
	
	\item The theory built here is directed. One could look for undirected variants, for instance comparing this work with the treatment of Bayes nets and open factor graphs in \cite{CopyCompSmithe}. 
	
	\item Developing a \emph{dependent version} of our theory, where the inputs depend on the output of the system, would be interesting.
	
	\item Here, we used Markov categories with conditionals to model uncertainty. One might want to also allow exact conditioning, for instance in order to represent (evidential) decision theory problems. One framework for dealing with these questions is \emph{partial Markov categories} \cite{PartialMarkovCats}. It would be of interest to see if our constructions can be adapted to that context, or if the language of categorical systems theory can be useful for decision theory.
	
	\item In cartesian closed contexts, one might be able to deal with trajectories more directly, by manipulating nondeterministic infinite sequences. For instance, there is a strong affine commutative monad on the cartesian closed category of Quasi-Borel Spaces \cite[Theorem 21 and Proposition 22]{HigherOrderQBS}, which yields a Markov category whose deterministic subcategory is cartesian closed. Unfortunately, the Markov category of Quasi-Borel spaces does not have conditionals; thus adapting our work to this setting does not seem straightforward.
\end{enumerate}

%% file: appendix_markov.tex
\section{Computations in Markov Categories}

The following Lemma will be used for some of our computations.

	\begin{lemma}\label{lemma_CP_and_regen}
		Assume that $\CC$ has conditionals. Let $X$, $A$, $B$, $C$ be objects, let $\phi: X \rightarrow A \otimes B$, $\psi: X \rightarrow B \otimes C$ be maps with the same marginal on $B$. Let $f: A \otimes B \rightarrow A \otimes B$ be a map such that the equality $f=id_{A \otimes B}$ holds $\phi$-almost surely, i.e. $(\phi; copy_{A \otimes B}; f \otimes A \otimes B) = (\phi; copy_{A \otimes B})$. Then we have $(\phi \otimes_B \psi; f \otimes C) = \phi \otimes_B \psi$.
	\end{lemma}

	\begin{proof}
		It suffices to post compose the equality $(\phi; copy_{A \otimes B}; f \otimes A \otimes B) = (\phi; copy_{A \otimes B})$ with the map $A \otimes B \otimes del_A \otimes (\psi|B ; del_B): A \otimes B \otimes A \otimes B \rightarrow A \otimes B \otimes C$, and use counitality of $del_A$ with respect to $copy$. 
	\end{proof}

	\begin{remark}
		There is a symmetric statement: if we are given $g: B \otimes C \rightarrow B \otimes C$ such that $id_{B \otimes C} = g$ holds $\psi$-almost surely, then we have $(\phi\otimes_B \psi; A \otimes g) = \phi \otimes_B \psi$.
	\end{remark}

	We shall apply Lemma \ref{lemma_CP_and_regen} in the following context:

	\begin{lemma}\label{lemma_det_almost_sure_equality}
		Let $X$, $Y$, $Z$ be objects. Let $\chi: X \rightarrow Y$ be a morphism, and $d: Y \rightarrow Z$ be a deterministic morphism. Let $\phi: X \rightarrow Y \otimes Z$ be the composite $\chi; X \otimes copy_Y; X \otimes Y \otimes d$. Also, let $f: Y \otimes Z \rightarrow Y \otimes Z$ denote the map $del_Z; copy_Y; Y \otimes d$.

		Then, the equality $f=id_{Y \otimes Z}$ holds $\phi$-almost surely.
	\end{lemma}

	\begin{proof}
		The proof uses determinism of $d$, associativity and commutativity of $copy$, and counitality of $del$:

\ctikzfig{lemma_almost_sure_equ_0}
\ctikzfig{lemma_almost_sure_equ_1}

	\end{proof}

%% file: appendix_triple_cats.tex
\section{Triple categories}

If we are working in a (strict) double category, we might let $\cdot|\cdot$ denote composition in one direction, and $\frac{\,\,\cdot\,\,}{\,\,\cdot\,\,}$, composition in the other direction.
For triple categories, we shall also use the notation $;_x$, to denote $x$-composition of $x$-morphisms, or $xy$- or $xz$-squares, or $xyz$-cubes. Similarly, let $;_y$, resp. $;_z$, denote composition in the $y$-, resp. $z$-, direction.

\begin{definition}\label{defi_nat_transfo_full}

\begin{itemize}
	\item Let $F, G: \Cc \rightarrow \Dd$ be triple functors. An $x$-natural transformation $\alpha: F \rightarrow G$ is given by the following data and conditions:

\begin{enumerate}
	\item It maps objects $c$ of $\mathbb{C}$ to $x$-morphisms $\alpha(c): F(c) \rightarrow G(c)$, naturally in the $x$ direction, i.e. $\alpha(c);_x G(f) = F(f);_x \alpha(d)$ if $f: c \rightarrow d$.
	\item It maps $y$ morphisms, resp. $z$ morphisms, of $\mathbb{C}$ to $xy$ morphisms, resp. $xz$ morphisms, of $\mathbb{D}$, with prescribed boundaries, functorially in the $y$ direction, resp. in the $z$ direction, with the usual naturality condition for $xy$-squares of $\mathbb{C}$, resp. $xz$-squares of $\mathbb{C}$.
	
	\[\begin{tikzcd}
		{F(a)} && {G(a)} \\
		& {\alpha(f)} \\
		{F(b)} && {G(b)}
		\arrow["{\alpha(a)}", from=1-1, to=1-3]
		\arrow["{F(f)}"', from=1-1, to=3-1]
		\arrow["{G(f)}", from=1-3, to=3-3]
		\arrow["{\alpha(b)}"', from=3-1, to=3-3]
	\end{tikzcd}\]

	Functoriality in the $y$ direction means $\alpha(\frac{f}{g}) = \frac{\alpha(f)}{\alpha(g)}$, for all $y$-composable $y$-morphisms $f$, $g$, where $\frac{}{}$ denotes $y$-composition, along with the equality $\alpha(id_a)=id_{\alpha(a)}$ for all objects $a$.

	For each $xy$-square $s$ in $\Cc$ as below, with the $x$-morphisms denoted horizontally, and the $y$-morphisms denoted vertically, we have $F(s) | \alpha(f^{\prime}) = \alpha(f) | G(s) $, where $|$ denotes $x$-composition.

	\[\begin{tikzcd}
		a && {a^{\prime}} \\
		& s \\
		b && {b^{\prime}}
		\arrow["g", from=1-1, to=1-3]
		\arrow["f", from=1-1, to=3-1]
		\arrow["{f^{\prime}}", from=1-3, to=3-3]
		\arrow["h"', from=3-1, to=3-3]
	\end{tikzcd}\]
	
	\[\begin{tikzcd}
		{F(a)} && {F(a^{\prime})} && {G(a^{\prime})} \\
		& {F(s)} && {\alpha(f^{\prime})} \\
		{F(b)} && {F(b^{\prime})} && {G(b^{\prime})} \\
		&& {=} \\
		{F(a)} && {G(a)} && {G(a^{\prime})} \\
		& {\alpha(f)} && {G(s)} \\
		{F(b)} && {G(b)} && {G(b^{\prime})}
		\arrow["{F(g)}"{description}, from=1-1, to=1-3]
		\arrow["{F(f)}"{description}, from=1-1, to=3-1]
		\arrow["{\alpha(a^{\prime})}"{description}, from=1-3, to=1-5]
		\arrow["{F(f^{\prime})}"{description}, from=1-3, to=3-3]
		\arrow["{G(f^{\prime})}"{description}, from=1-5, to=3-5]
		\arrow["{F(h)}"{description}, from=3-1, to=3-3]
		\arrow["{\alpha(b^{\prime})}"{description}, from=3-3, to=3-5]
		\arrow["{\alpha(a)}"{description}, from=5-1, to=5-3]
		\arrow["{F(f)}"{description}, from=5-1, to=7-1]
		\arrow["{G(g)}"{description}, from=5-3, to=5-5]
		\arrow["{G(f)}"{description}, from=5-3, to=7-3]
		\arrow["{G(f^{\prime})}"{description}, from=5-5, to=7-5]
		\arrow["{\alpha(b)}"{description}, from=7-1, to=7-3]
		\arrow["{G(h)}"{description}, from=7-3, to=7-5]
	\end{tikzcd}\]
	
		We require similar equalities for $xz$-squares.
	
	\item It maps $yz$ squares $s$ of $\mathbb{C}$ to $xyz$ cubes $\alpha(s)$ in $\mathbb{D}$, having prescribed boundaries, double functorially in the $yz$ square $s$, and with a naturality condition for $xyz$ cubes of $\mathbb{C}$.

	\[\begin{tikzcd}
		&& \bullet &&& \bullet \\
		\\
		\bullet &&& \bullet \\
		&& \bullet &&& \bullet \\
		\\
		\bullet &&& \bullet
		\arrow[from=1-3, to=1-6]
		\arrow[""{name=0, anchor=center, inner sep=0}, "{F(f)}"{description}, from=1-3, to=3-1]
		\arrow[""{name=1, anchor=center, inner sep=0}, "{F(k)}"{description}, from=1-3, to=4-3]
		\arrow[""{name=2, anchor=center, inner sep=0}, "{G(f)}"{description}, from=1-6, to=3-4]
		\arrow[""{name=3, anchor=center, inner sep=0}, "{G(k)}"{description}, from=1-6, to=4-6]
		\arrow[from=3-1, to=3-4]
		\arrow[""{name=4, anchor=center, inner sep=0}, "{F(h)}"{description}, from=3-1, to=6-1]
		\arrow[""{name=5, anchor=center, inner sep=0}, "{G(h)}"{description}, from=3-4, to=6-4]
		\arrow[from=4-3, to=4-6]
		\arrow[""{name=6, anchor=center, inner sep=0}, "{F(g)}"{description}, from=4-3, to=6-1]
		\arrow[""{name=7, anchor=center, inner sep=0}, "{G(g)}"{description}, from=4-6, to=6-4]
		\arrow[from=6-1, to=6-4]
		\arrow["{\alpha(f)}"{description}, shorten <=19pt, shorten >=19pt, Rightarrow, from=0, to=2]
		\arrow["{\alpha(k)}"{description}, shorten <=19pt, shorten >=19pt, Rightarrow, from=1, to=3]
		\arrow["{F(s)}"{description}, shorten <=14pt, shorten >=14pt, equals, from=1, to=4]
		\arrow["{G(s)}"{description}, shorten <=14pt, shorten >=14pt, equals, from=3, to=5]
		\arrow["{\alpha(h)}"{description}, shorten <=19pt, shorten >=19pt, Rightarrow, from=4, to=5]
		\arrow["{\alpha(g)}"{description}, shorten <=19pt, shorten >=19pt, Rightarrow, from=6, to=7]
	\end{tikzcd}\]

	Double functoriality in $s$ means that $\alpha(id_f) = id_{\alpha(f)}$ if $f$ is a $y$-morphism or a $z$-morphism, that $\alpha(s ;_y t) = \alpha(s) ;_y \alpha(t)$ if $s$ and $t$ are $y$-composable squares, and that $\alpha(s ;_z t) = \alpha(s) ;_z \alpha(t)$ if $s$ and $t$ are $z$-composable squares in $\Cc$.

	The naturality condition for cubes is the following: if $c$ is a cube in $\Cc$ with $yz$ faces $s$ and $t$, then we have $\alpha(s);_x G(c) = F(c);_x \alpha(t)$.
	
\end{enumerate}

We define $y$-natural transformations and $z$-natural transformations in a similar way.

	\item Let $F_1, G_1, F_2, G_2 : \mathbb{C} \rightarrow \mathbb{D}$ be triple functors, let $\alpha_i: F_i \rightarrow G_i$ be $x$-natural transformations, $i=1,2$, and let $\beta_1 : F_1 \rightarrow F_2$, $\beta_2: G_1 \rightarrow G_2$ be $y$-natural transformations. An $xy$ natural transformation $\gamma$ with boundary

\begin{tikzcd}
	F_1 \arrow[r, "\alpha_1"] \arrow[d, "\beta_1"]& G_1 \arrow[d, "\beta_2"]
	\\
	F_2 \arrow[r, "\alpha_2"]& G_2
\end{tikzcd}
is defined by the following data:

\begin{enumerate}
	\item It maps objects $c$ of $\mathbb{C}$ to $xy$ squares $\gamma(c)$ in $\mathbb{D}$ with suitable boundaries, satisfying naturality conditions for $x$ morphisms and $y$ morphisms in $\mathbb{C}$ similar to modifications. 
	
	\[\begin{tikzcd}
		{F_1(c)} && {G_1(c)} \\
		& {\gamma(c)} \\
		{F_2(c)} && {G_2(c)}
		\arrow["{\alpha_1(c)}"{description}, from=1-1, to=1-3]
		\arrow["{\beta_1(c)}"{description}, from=1-1, to=3-1]
		\arrow["{\beta_2(c)}"{description}, from=1-3, to=3-3]
		\arrow["{\alpha_2(c)}"{description}, from=3-1, to=3-3]
	\end{tikzcd}\]

	The naturality conditions are as follows: if $f: c \rightarrow d$, resp. $g: c \rightarrow d$ is an $x$-morphism, resp. a $y$-morphism, then we have $\gamma(c) ;_x \beta_2(f) = \beta_1(f);_x \gamma(d)$, resp. $\gamma(c) ;_y \alpha_2(g) = \alpha_1(g) ;_y \gamma(d)$.

	\[\begin{tikzcd}
		{F_1(c)} && {G_1(c)} && {G_1(d)} \\
		& {\gamma(c)} && {\beta_2(f)} \\
		{F_2(c)} && {G_2(c)} && {G_2(d)} \\
		&& {=} \\
		{F_1(c)} && {F_1(d)} && {G_1(d)} \\
		& {\beta_1(f)} && {\gamma(d)} \\
		{F_2(c)} && {F_2(d)} && {G_2(d)}
		\arrow["{\alpha_1(c)}"{description}, from=1-1, to=1-3]
		\arrow["{\beta_1(c)}"{description}, from=1-1, to=3-1]
		\arrow["{G_1(f)}"{description}, from=1-3, to=1-5]
		\arrow["{\beta_2(c)}"{description}, from=1-3, to=3-3]
		\arrow["{\beta_2(d)}"{description}, from=1-5, to=3-5]
		\arrow["{\alpha_2(c)}"{description}, from=3-1, to=3-3]
		\arrow["{G_2(f)}"{description}, from=3-3, to=3-5]
		\arrow["{F_1(f)}"{description}, from=5-1, to=5-3]
		\arrow["{\beta_1(c)}"{description}, from=5-1, to=7-1]
		\arrow["{\alpha_1(d)}"{description}, from=5-3, to=5-5]
		\arrow["{\beta_1(d)}"{description}, from=5-3, to=7-3]
		\arrow["{\beta_2(d)}"{description}, from=5-5, to=7-5]
		\arrow["{F_2(f)}"{description}, from=7-1, to=7-3]
		\arrow["{\alpha_2(d)}"{description}, from=7-3, to=7-5]
	\end{tikzcd}\]

	\[\begin{tikzcd}
		{F_1(c)} && {G_1(c)} && {F_1(c)} && {G_1(c)} \\
		& {\gamma(c)} &&&& {\alpha_1(g)} \\
		{F_2(c)} && {G_2(c)} & {=} & {F_1(d)} && {G_1(d)} \\
		& {\alpha_2(g)} &&&& {\gamma(d)} \\
		{F_2(d)} && {G_2(d)} && {F_2(d)} && {G_2(d)}
		\arrow["{\alpha_1(c)}"{description}, from=1-1, to=1-3]
		\arrow["{\beta_1(c)}"{description}, from=1-1, to=3-1]
		\arrow["{\beta_2(c)}"{description}, from=1-3, to=3-3]
		\arrow["{\alpha_1(c)}"{description}, from=1-5, to=1-7]
		\arrow["{F_1(g)}"{description}, from=1-5, to=3-5]
		\arrow["{G_1(g)}"{description}, from=1-7, to=3-7]
		\arrow["{\alpha_2(c)}"{description}, from=3-1, to=3-3]
		\arrow["{F_2(g)}"{description}, from=3-1, to=5-1]
		\arrow["{G_2(g)}"{description}, from=3-3, to=5-3]
		\arrow["{\alpha_1(d)}"{description}, from=3-5, to=3-7]
		\arrow["{\beta_1(d)}"{description}, from=3-5, to=5-5]
		\arrow["{\beta_2(d)}"{description}, from=3-7, to=5-7]
		\arrow["{\alpha_2(d)}"{description}, from=5-1, to=5-3]
		\arrow["{\alpha_2(d)}"{description}, from=5-5, to=5-7]
	\end{tikzcd}\]

	\item It maps $z$ morphisms $f$ of $\mathbb{C}$ to $xyz$ cubes $\gamma(f)$ in $\mathbb{D}$ with the required boundaries, functorially in the $z$ direction.
	
	\[\begin{tikzcd}
		&& {F_1(c)} &&&&& {G_1(c)} \\
		\\
		{F_1(d)} &&&&& {G_1(d)} \\
		&& {F_2(c)} &&&&& {G_2(c)} \\
		\\
		{F_2(d)} &&&&& {G_2(d)}
		\arrow[""{name=0, anchor=center, inner sep=0}, "{\alpha_1(c)}"{description}, from=1-3, to=1-8]
		\arrow["{F_1(f)}"{description}, from=1-3, to=3-1]
		\arrow[""{name=1, anchor=center, inner sep=0}, "{\beta_1(c)}"{description}, from=1-3, to=4-3]
		\arrow["{G_1(f)}"{description}, from=1-8, to=3-6]
		\arrow[""{name=2, anchor=center, inner sep=0}, "{\beta_2(c)}"{description}, from=1-8, to=4-8]
		\arrow[""{name=3, anchor=center, inner sep=0}, "{\alpha_1(d)}"{description}, from=3-1, to=3-6]
		\arrow[""{name=4, anchor=center, inner sep=0}, "{\beta_1(d)}"{description}, from=3-1, to=6-1]
		\arrow[""{name=5, anchor=center, inner sep=0}, "{\beta_2(d)}"{description}, from=3-6, to=6-6]
		\arrow[""{name=6, anchor=center, inner sep=0}, "{\alpha_2(c)}"{description}, from=4-3, to=4-8]
		\arrow["{F_2(f)}"{description}, from=4-3, to=6-1]
		\arrow["{G_2(f)}"{description}, from=4-8, to=6-6]
		\arrow[""{name=7, anchor=center, inner sep=0}, "{\alpha_2(d)}"{description}, from=6-1, to=6-6]
		\arrow["{\alpha_1(f)}"{description}, shorten <=15pt, shorten >=15pt, Rightarrow, from=0, to=3]
		\arrow["{\beta_1(f)}"{description}, shorten <=15pt, shorten >=15pt, Rightarrow, from=1, to=4]
		\arrow["{\beta_2(f)}"{description}, shorten <=15pt, shorten >=15pt, Rightarrow, from=2, to=5]
		\arrow["{\alpha_2(f)}"{description}, shorten <=15pt, shorten >=15pt, Rightarrow, from=6, to=7]
	\end{tikzcd}\]
	Functoriality means that identity $z$-morphisms are mapped to identity cubes, and $\gamma(f;_z g) = \gamma(f) ;_z \gamma(g)$, for all $z$-composable $z$-morphisms $f$, $g$.

\end{enumerate}

\item Given eight triple functors, four $x$-natural transformations, four $y$-natural transformations, four $z$-natural transformations, two $xy$-natural transformations, two $xz$-natural transformations, and two $yz$-natural transformations, arranged as the boundary of a cube, an $xyz$ natural transformation with this boundary maps objects of $\mathbb{C}$ to $xyz$ morphisms in $\mathbb{D}$ with suitable boundary, satisfying naturality conditions with respect to $x$-morphisms, $y$-morphisms and $z$-morphisms of $\mathbb{C}$.

Let us use the following convention: for triple functors, subscripts denote coordinates on the diagrams, in the $xyz$ order; the vertical direction is the $y$-direction, the horizontal one is the $x$-direction, and the last one is the $z$-direction, so that our triple functors might be named $F_{010}$, $F_{111}$, etc. For $1$-dimensional natural transformations, we use names such as $\alpha^{x}_{01}$ or $\alpha^{z}_{11}$, where the superscript denotes the kind of natural transformation, and the subscript denotes the coordinates with respect to the other two directions.
\[\begin{tikzcd}
	&& {F_{000}} &&&&& {F_{100}} \\
	\\
	{F_{001}} &&&&& {F_{101}} \\
	\\
	\\
	&& {F_{010}} &&&&& {F_{110}} \\
	\\
	{F_{011}} &&&&& {F_{111}}
	\arrow[""{name=0, anchor=center, inner sep=0}, "{\alpha^{x}_{00}}"{description}, from=1-3, to=1-8]
	\arrow["{\alpha^{z}_{00}}"{description}, from=1-3, to=3-1]
	\arrow[""{name=1, anchor=center, inner sep=0}, "{\alpha^{y}_{00}}"{description}, from=1-3, to=6-3]
	\arrow["{\alpha^{z}_{10}}"{description}, from=1-8, to=3-6]
	\arrow[""{name=2, anchor=center, inner sep=0}, "{\alpha^{y}_{10}}"{description}, from=1-8, to=6-8]
	\arrow[""{name=3, anchor=center, inner sep=0}, "{\alpha^{x}_{01}}"{description}, from=3-1, to=3-6]
	\arrow[""{name=4, anchor=center, inner sep=0}, "{\alpha^{y}_{01}}"{description}, from=3-1, to=8-1]
	\arrow[""{name=5, anchor=center, inner sep=0}, "{\alpha^{y}_{11}}"{description}, from=3-6, to=8-6]
	\arrow[""{name=6, anchor=center, inner sep=0}, "{\alpha^{x}_{10}}"{description}, from=6-3, to=6-8]
	\arrow["{\alpha^{z}_{01}}"{description}, from=6-3, to=8-1]
	\arrow["{\alpha^{z}_{11}}"{description}, from=6-8, to=8-6]
	\arrow[""{name=7, anchor=center, inner sep=0}, "{\alpha^{x}_{11}}"{description}, from=8-1, to=8-6]
	\arrow["{\gamma^{xz}_0}"{description}, shorten <=14pt, shorten >=14pt, Rightarrow, from=0, to=3]
	\arrow["{\gamma^{xy}_0}"{description}, shorten <=32pt, shorten >=32pt, Rightarrow, from=1, to=2]
	\arrow["{\gamma^{yz}_0}"{description}, shorten <=14pt, shorten >=14pt, Rightarrow, from=1, to=4]
	\arrow["{\gamma^{yz}_1}"{description}, shorten <=14pt, shorten >=14pt, Rightarrow, from=2, to=5]
	\arrow["{\gamma^{xy}_1}"{description}, shorten <=32pt, shorten >=32pt, Rightarrow, from=4, to=5]
	\arrow["{\gamma^{xz}_1}"{description}, shorten <=14pt, shorten >=14pt, Rightarrow, from=6, to=7]
\end{tikzcd}\]

So, an $xyz$-natural transformation $\varepsilon$ with boundary as above maps each object $c \in \Cc$ to an $xyz$-cube $\varepsilon(c)$ in $\Dd$ with boundary as below:

\[\begin{tikzcd}
	&& {F_{000}(c)} &&&&& {F_{100}(c)} \\
	\\
	{F_{001}(c)} &&&&& {F_{101}(c)} \\
	\\
	\\
	&& {F_{010}(c)} &&&&& {F_{110}(c)} \\
	\\
	{F_{011}(c)} &&&&& {F_{111}(c)}
	\arrow[""{name=0, anchor=center, inner sep=0}, "{\alpha^{x}_{00}(c)}"{description}, from=1-3, to=1-8]
	\arrow["{\alpha^{z}_{00}(c)}"{description}, from=1-3, to=3-1]
	\arrow[""{name=1, anchor=center, inner sep=0}, "{\alpha^{y}_{00}(c)}"{description}, from=1-3, to=6-3]
	\arrow["{\alpha^{z}_{10}(c)}"{description}, from=1-8, to=3-6]
	\arrow[""{name=2, anchor=center, inner sep=0}, "{\alpha^{y}_{10}(c)}"{description}, from=1-8, to=6-8]
	\arrow[""{name=3, anchor=center, inner sep=0}, "{\alpha^{x}_{01}(c)}"{description}, from=3-1, to=3-6]
	\arrow[""{name=4, anchor=center, inner sep=0}, "{\alpha^{y}_{01}(c)}"{description}, from=3-1, to=8-1]
	\arrow[""{name=5, anchor=center, inner sep=0}, "{\alpha^{y}_{11}(c)}"{description}, from=3-6, to=8-6]
	\arrow[""{name=6, anchor=center, inner sep=0}, "{\alpha^{x}_{10}(c)}"{description}, from=6-3, to=6-8]
	\arrow["{\alpha^{z}_{01}(c)}"{description}, from=6-3, to=8-1]
	\arrow["{\alpha^{z}_{11}(c)}"{description}, from=6-8, to=8-6]
	\arrow[""{name=7, anchor=center, inner sep=0}, "{\alpha^{x}_{11}(c)}"{description}, from=8-1, to=8-6]
	\arrow["{\gamma^{xz}_0(c)}"{description}, shorten <=15pt, shorten >=15pt, Rightarrow, from=0, to=3]
	\arrow["{\gamma^{xy}_0(c)}"{description}, shorten <=35pt, shorten >=35pt, Rightarrow, from=1, to=2]
	\arrow["{\gamma^{yz}_0(c)}"{description}, shorten <=15pt, shorten >=15pt, Rightarrow, from=1, to=4]
	\arrow["{\gamma^{yz}_1(c)}"{description}, shorten <=15pt, shorten >=15pt, Rightarrow, from=2, to=5]
	\arrow["{\gamma^{xy}_1(c)}"{description}, shorten <=35pt, shorten >=35pt, Rightarrow, from=4, to=5]
	\arrow["{\gamma^{xz}_1(c)}"{description}, shorten <=15pt, shorten >=15pt, Rightarrow, from=6, to=7]
\end{tikzcd}\]

The naturality condition is as follows: if $f: c \rightarrow d$, resp. $g: c \rightarrow d$, resp. $h: c \rightarrow d$, is an $x$-, resp. $y$-, resp. $z$-, morphism in $\Cc$, then we have $\varepsilon(c);_x \gamma_1^{yz}(f) = \gamma_0^{yz}(f);_x \varepsilon(d)$, resp. $\varepsilon(c);_y \gamma_1^{xz}(g) = \gamma_0^{xz}(g);_y \varepsilon(d)$, resp. $\varepsilon(c);_z \gamma_1^{xy}(h) = \gamma_0^{xy}(h);_y \varepsilon(d)$.

\end{itemize}

\end{definition}

%% file: appendix_arena_moore.tex
\subsection{The triple category $Arena^{\mathrm{Moore}}_{\CC}$}\label{appendix_Arena_Moore}

Here, we prove Theorem \ref{theo_Arena_C}, by constructing the $x$-semi symmetric monoidal triple category $Arena^{\mathrm{Moore}}_{\CC}$.

\begin{construction}\label{constr_Arena_C}
	
Consider the symmetric monoidal $x$-semi triple category defined as follows.

	\begin{enumerate}

		\item The objects are pairs $\begin{pmatrix} a \\ c  \\ \end{pmatrix}$, where $a$ and $c$ are objects of $\CC$.

		\item The $x$ morphisms ${\begin{pmatrix} a_1 \\ c_1  \\ \end{pmatrix}} \rightrightarrows {\begin{pmatrix} a_2 \\ c_2  \\ \end{pmatrix}}$, also known as "co-parametric charts" or "copy-composition charts", are given by objects ${\begin{pmatrix} a_{12} \\ c_{12}  \\ \end{pmatrix}}$, along with pairs of morphisms $f_{12}: c_1 \rightarrow c_{12} \otimes c_2$, $f_{12}^{\flat}: c_1 \otimes a_1 \rightarrow c_{12} \otimes c_2 \otimes a_{12} \otimes a_2$, such that the following square commutes:

		\begin{center}
			\begin{tikzcd}
				c_1 \otimes a_1 \arrow[r, "f_{12}^{\flat}"] \arrow[d, "\pi"]& c_{12} \otimes c_2 \otimes a_{12} \otimes a_2  \arrow[d, "\pi"]
				\\
				c_1 \arrow[r, "f_{12}"]& c_{12} \otimes c_2
			\end{tikzcd}
		\end{center}
		Composition is defined using copy-composition in $\CC$, i.e. composition of co-parametric maps: using the same notations as above, the composite of $x$-morphisms ${\begin{pmatrix} a_1 \\ c_1  \\ \end{pmatrix}} \rightrightarrows {\begin{pmatrix} a_2 \\ c_2  \\ \end{pmatrix}}$ and ${\begin{pmatrix} a_2 \\ c_2  \\ \end{pmatrix}} \rightrightarrows {\begin{pmatrix} a_3 \\ c_3  \\ \end{pmatrix}}$ is given by the object ${\begin{pmatrix} a_{12} \otimes a_2 \otimes a_{23} \\ c_{12} \otimes c_2 \otimes c_{23}  \\ \end{pmatrix}}$, along with the morphisms $c_1 \rightarrow c_{12} \otimes c_2 \otimes c_{23} \otimes c_3$ and $c_1 \otimes a_1 \rightarrow c_{12} \otimes c_2 \otimes c_{23} \otimes c_3 \otimes a_{12} \otimes a_2 \otimes a_{23} \otimes a_3$ in $\CC$ obtained by copy-composition:

\[\begin{tikzcd}
	{c_1a_1} & {c_{12}c_2 a_{12}a_2} & {c_{12}c_2 a_{12}a_2 \,\, c_2 a_2} && {c_{12} c_2 c_{23} c_3 a_{12} a_2 a_{23} a_3} \\
	\\
	{c_1} & {c_{12}c_2} & {c_{12} c_2 \,\, c_2} && {c_{12} c_2 c_{23} c_3}
	\arrow["{f_{12}^{\flat}}", from=1-1, to=1-2]
	\arrow["{\sigma; copy_{c_2a_2}; \sigma}", from=1-2, to=1-3]
	\arrow["{(c_{12}c_2a_{12}a_2 \otimes f_{23}^{\flat}); \sigma}", from=1-3, to=1-5]
	\arrow["{f_{12}}", from=3-1, to=3-2]
	\arrow["{copy_{c_2}}", from=3-2, to=3-3]
	\arrow["{(c_{12}c_2 \otimes f_{23})}", from=3-3, to=3-5]
\end{tikzcd}\]

Let us show that this composition is well-defined, i.e. the required square commutes:

\[\begin{tikzcd}
	{c_1a_1} & {c_{12}c_2 a_{12}a_2} & {c_{12}c_2 a_{12}a_2 \,\, c_2 a_2} && {c_{12} c_2 c_{23} c_3 a_{12} a_2 a_{23} a_3} \\
	\\
	{c_1} & {c_{12}c_2} & {c_{12} c_2 \,\, c_2} && {c_{12} c_2 c_{23} c_3}
	\arrow["{f_{12}^{\flat}}", from=1-1, to=1-2]
	\arrow["\pi"{description}, from=1-1, to=3-1]
	\arrow["{\sigma; copy_{c_2a_2}; \sigma}", from=1-2, to=1-3]
	\arrow["{del_{a_{12} a_2}}"{description}, from=1-2, to=3-2]
	\arrow["{(c_{12}c_2a_{12}a_2 \otimes f_{23}^{\flat}); \sigma}", from=1-3, to=1-5]
	\arrow["{del_{a_{12} a_2} \otimes del_{a_2}}"{description}, from=1-3, to=3-3]
	\arrow["\pi"{description}, from=1-5, to=3-5]
	\arrow["{f_{12}}", from=3-1, to=3-2]
	\arrow["{copy_{c_2}}", from=3-2, to=3-3]
	\arrow["{(c_{12}c_2 \otimes f_{23})}", from=3-3, to=3-5]
\end{tikzcd}\]

The leftmost and rightmost outer squares in the diagram above commute thanks to $(f_{12}, f_{12}^{\flat})$ and $(f_{23}, f_{23}^{\flat})$ being $x$-morphisms. The center square commutes thanks to the properties of the maps $\sigma$, $del$, and $copy$, coming from the Markov category axioms.

Associativity follows from associativity of copying, along with functoriality of the tensor product. However, there are no identity $x$-morphisms in general, because of the copying; that is why we shall get an $x$-semi triple category.

The motivation for using copy-composition, and for such a definition for $x$-morphisms in the first place, is the need for $xy$-interchange; hopefully, the proof of Proposition \ref{prop_interchange_lenses} will make it clearer.

		\item The $y$ morphisms are \emph{deterministic} lenses $({\begin{pmatrix} a_1 \\ c_1  \\ \end{pmatrix}} \leftrightarrows {\begin{pmatrix} a_2 \\ c_2  \\ \end{pmatrix}})$, i.e. they are given by pairs of deterministic maps, i.e. morphisms in $\CC_{det}$: $f: c_1 \rightarrow c_2$, $f^{\sharp}:  c_1 \otimes a_2 \rightarrow a_1$. Composition is composition of lenses in the cartesian category $\CC_{det}$: given $({\begin{pmatrix} a_1 \\ c_1  \\ \end{pmatrix}} \leftrightarrows {\begin{pmatrix} a_2 \\ c_2  \\ \end{pmatrix}})$ and $({\begin{pmatrix} a_2 \\ c_2  \\ \end{pmatrix}} \leftrightarrows {\begin{pmatrix} a_3 \\ c_3  \\ \end{pmatrix}})$, the composite $( {\begin{pmatrix} a_1 \\ c_1  \\ \end{pmatrix}} \leftrightarrows {\begin{pmatrix} a_3 \\ c_3  \\ \end{pmatrix}})$ is given by the maps $c_1 \rightarrow c_2 \rightarrow c_3$ and $  c_1 \otimes a_3 \rightarrow c_1 \otimes c_1 \otimes a_3 \rightarrow  c_1 \otimes   c_2 \otimes a_3 \rightarrow  c_1 \otimes a_2 \rightarrow a_1$.
		
		\item The $z$ morphisms ${\begin{pmatrix} a_1 \\ c_1  \\ \end{pmatrix}} \rightrightarrows {\begin{pmatrix} a_2 \\ c_2  \\ \end{pmatrix}}$, are pairs of \emph{deterministic} morphisms $f: c_1 \rightarrow c_2$, $g: a_1 \rightarrow a_2$, in $\CC$.


		\item Consider the boundary of an $xy$ square as in the diagram below:
		\begin{center}
\[\begin{tikzcd}
	\begin{array}{c} {\begin{pmatrix} a_1 \\ c_1  \\ \end{pmatrix}} \end{array} && \begin{array}{c} {\begin{pmatrix} a_2 \\ c_2  \\ \end{pmatrix}} \end{array} \\
	\begin{array}{c} {\begin{pmatrix} a_3 \\ c_3  \\ \end{pmatrix}} \end{array} && \begin{array}{c} {\begin{pmatrix} a_4 \\ c_4  \\ \end{pmatrix}} \end{array}
	\arrow["{(a_{12}, c_{12}, f_{12}, f^{\flat}_{12})}", shift left=2, from=1-1, to=1-3]
	\arrow[shift right=2, from=1-1, to=1-3]
	\arrow["{(f_{13}, f_{13}^{\flat})}"', shift right=2, from=1-1, to=2-1]
	\arrow[shift right=2, from=1-3, to=2-3]
	\arrow[shift right=2, from=2-1, to=1-1]
	\arrow[shift left=2, from=2-1, to=2-3]
	\arrow["{(a_{34}, c_{34}, f_{34}, f^{\flat}_{34})}"', shift right=2, from=2-1, to=2-3]
	\arrow["{(f_{24}, f_{24}^{\flat})}"', shift right=2, from=2-3, to=1-3]
\end{tikzcd}\]
		\end{center}
		
		Then, a \emph{square} with this boundary is given by a $y$-morphism $(f_{1234}, f_{1234}^{\sharp}) : {\begin{pmatrix} a_{12} \\ c_{12}  \\ \end{pmatrix}} \leftrightarrows {\begin{pmatrix} a_{34} \\ c_{34}  \\ \end{pmatrix}}$, along with a morphism $s_{}:  c_1 \otimes a_3 \rightarrow   c_{12} \otimes c_2 \otimes a_{34} \otimes a_4$ in $\CC$, making the following squares commute:
		
\[\begin{tikzcd}
	{{c_1 a_3}} & {{c_{12}c_2 a_{34}a_4}} & {c_1} & {c_{12}c_2} \\
	{c_3 a_3} & {c_{34}c_4 a_{34}a_4} & {c_3} & {c_{34}c_4} \\
	{{c_1 a_3}} &&& {{c_{12}c_2 a_{34}a_4}} \\
	{c_1 a_1} &&& {c_{12}c_2 a_{12}a_2}
	\arrow["s", from=1-1, to=1-2]
	\arrow["{ f_{13}}"{description}, from=1-1, to=2-1]
	\arrow["{f_{1234} \otimes f_{24}}"{description}, from=1-2, to=2-2]
	\arrow["{f_{12}}", from=1-3, to=1-4]
	\arrow["{f_{13}}"{description}, from=1-3, to=2-3]
	\arrow["{f_{1234} \otimes f_{24}}"{description}, from=1-4, to=2-4]
	\arrow["{f_{34}^{\flat}}", from=2-1, to=2-2]
	\arrow["{f_{34}}", from=2-3, to=2-4]
	\arrow["s", from=3-1, to=3-4]
	\arrow["{copy_{c_1}  ; f_{13}^{\sharp}}"{description}, from=3-1, to=4-1]
	\arrow["{copy_{c_{12}c_2} ; \sigma ;(f_{1234}^{\sharp}\otimes f_{24}^{\sharp})}"{description}, from=3-4, to=4-4]
	\arrow["{f_{12}^{\flat}}", from=4-1, to=4-4]
\end{tikzcd}\]

		\item Composition of $xy$ squares in the $x$ direction is given by copy-composition in $\CC$: let $s$ and $t$ be two $x$-composable squares

\[\begin{tikzcd}
	\begin{array}{c} {\begin{pmatrix} a_1 \\ c_1  \\ \end{pmatrix}} \end{array} && \begin{array}{c} {\begin{pmatrix} a_2 \\ c_2  \\ \end{pmatrix}} \end{array} && \begin{array}{c} {\begin{pmatrix} a_2 \\ c_2  \\ \end{pmatrix}} \end{array} && \begin{array}{c} {\begin{pmatrix} a_3 \\ c_3  \\ \end{pmatrix}} \end{array} \\
	& s &&&& t \\
	\begin{array}{c} {\begin{pmatrix} a_4 \\ c_4  \\ \end{pmatrix}} \end{array} && \begin{array}{c} {\begin{pmatrix} a_5 \\ c_5  \\ \end{pmatrix}} \end{array} && \begin{array}{c} {\begin{pmatrix} a_5 \\ c_5  \\ \end{pmatrix}} \end{array} && \begin{array}{c} {\begin{pmatrix} a_6 \\ c_6  \\ \end{pmatrix}} \end{array}
	\arrow["{(a_{12}, c_{12}, f_{12}, f^{\flat}_{12})}", shift left=2, from=1-1, to=1-3]
	\arrow[shift right=2, from=1-1, to=1-3]
	\arrow["{(f_{14}, f_{14}^{\sharp})}"{description}, shift right=2, from=1-1, to=3-1]
	\arrow[shift right=2, from=1-3, to=3-3]
	\arrow["{(a_{23}, c_{23}, f_{23}, f^{\flat}_{23})}", shift left=2, from=1-5, to=1-7]
	\arrow[shift right=2, from=1-5, to=1-7]
	\arrow["{(f_{25}, f_{25}^{\sharp})}"{description}, shift right=2, from=1-5, to=3-5]
	\arrow[shift right=2, from=1-7, to=3-7]
	\arrow[shift right=2, from=3-1, to=1-1]
	\arrow[shift left=2, from=3-1, to=3-3]
	\arrow["{(a_{45}, c_{45}, f_{45}, f^{\flat}_{45})}"', shift right=2, from=3-1, to=3-3]
	\arrow["{(f_{25}, f_{25}^{\sharp})}"{description}, shift right=2, from=3-3, to=1-3]
	\arrow[shift right=2, from=3-5, to=1-5]
	\arrow[shift left=2, from=3-5, to=3-7]
	\arrow["{(a_{56}, c_{56}, f_{56}, f^{\flat}_{56})}"', shift right=2, from=3-5, to=3-7]
	\arrow["{(f_{36}, f_{36}^{\sharp})}"{description}, shift right=2, from=3-7, to=1-7]
\end{tikzcd}\]

		with $s:  c_1 \otimes a_4 \rightarrow   c_{12} \otimes c_2 \otimes a_{45} \otimes a_5$ and $t:  c_2 \otimes a_5 \rightarrow    c_{23} \otimes c_3 \otimes a_{56} \otimes a_6$. Then, the $x$-composite $s | t $ is given by the tensor product of lenses 
		$(f_{1245} \otimes f_{25}\otimes f_{2356}, f_{1245}^{\sharp} \otimes f_{25}^{\sharp} \otimes f_{2356}^{\sharp}): {\begin{pmatrix} a_{12} \\ c_{12}  \\ \end{pmatrix}} \otimes {\begin{pmatrix} a_{2} \\ c_{2}  \\ \end{pmatrix}} \otimes {\begin{pmatrix} a_{23} \\ c_{23}  \\ \end{pmatrix}} \leftrightarrows {\begin{pmatrix} a_{45} \\ c_{45}  \\ \end{pmatrix}} \otimes {\begin{pmatrix} a_{5} \\ c_{5}  \\ \end{pmatrix}} \otimes {\begin{pmatrix} a_{56} \\ c_{56}  \\ \end{pmatrix}}$, along with the following composite Markov map:

\[\begin{tikzcd}
	{c_1 a_4} \\
	{c_{12}c_2 a_{45}a_5} \\
	{c_{12}c_2 a_{45}a_5 \,\,c_2a_5} \\
	{c_{12}c_2 a_{45}a_5 \, \, c_{23}c_3 a_{56}a_6} \\
	{c_{12}c_2 c_{23}c_3 a_{45}a_5 a_{56}a_6}
	\arrow["s"{description}, from=1-1, to=2-1]
	\arrow["{\sigma; copy_{c_2a_5}; \sigma}"{description}, from=2-1, to=3-1]
	\arrow["{c_{12}c_2 a_{45}a_5 \otimes t}"{description}, from=3-1, to=4-1]
	\arrow["\sigma"{description}, from=4-1, to=5-1]
\end{tikzcd}\]

\begin{claim}\label{claim_x_comp_well_def}
	This $x$-composition is well-defined.
\end{claim}
		
\begin{proof}
		This relies crucially on determinism of the maps $f_{25}$ and $f_{25}^{\sharp}$; let us now draw some commutative diagrams proving just that:

\[\begin{tikzcd}
	{c_1 a_4} &&&& {c_{12}c_2 a_{45}a_5} \\
	{c_1 a_1} &&&& {c_{12} a_{45}c_2a_5} \\
	{c_{12}c_2 a_{12}a_2} \\
	&&&& {c_{12}c_2 a_{45}a_5 \,\,c_2a_5} \\
	&&&& {c_{12}c_2 a_{45}a_5 \, \, c_{23}c_3 a_{56}a_6} \\
	{c_{12}c_2 a_{12}a_2 \, \, c_2 a_2} &&&& {c_{12}c_2 c_{23}c_3 a_{45}a_5 a_{56}a_6} \\
	&&&& {c_{12}c_2 c_{23}c_3 a_{12}a_2 a_{23}a_3}
	\arrow["s"{description}, from=1-1, to=1-5]
	\arrow["{copy_{c_1}; f_{14}^{\sharp}}"{description}, from=1-1, to=2-1]
	\arrow["\sigma"{description}, from=1-5, to=2-5]
	\arrow["{f_{12}^{\flat}}"{description}, from=2-1, to=3-1]
	\arrow["{\sigma; copy_{c_{12}c_2}; \sigma; (c_{12}c_2 \otimes f_{1245}^{\sharp} \otimes f_{25}^{\sharp})}"{description}, curve={height=6pt}, from=2-5, to=3-1]
	\arrow["{copy_{c_2a_5}; \sigma}"{description}, from=2-5, to=4-5]
	\arrow["{\sigma; copy_{c_2 a_2}; \sigma}"{description}, from=3-1, to=6-1]
	\arrow["{c_{12}c_2 a_{45}a_5 \otimes t}"{description}, from=4-5, to=5-5]
	\arrow["{\sigma; copy_{c_{12}c_2c_2}; \sigma;f_{1245}^{\sharp} \otimes f_{25}^{\sharp} \otimes f_{25}^{\sharp}}"{description}, curve={height=12pt}, from=4-5, to=6-1]
	\arrow["\sigma"{description}, from=5-5, to=6-5]
	\arrow["{c_{12}c_2 a_{12}a_2 \otimes f_{23}^{\flat}; \sigma}"', curve={height=12pt}, from=6-1, to=7-5]
	\arrow["{copy_{c_{12}c_2 c_{23}c_3}; \sigma; f_{1245}^{\sharp}\otimes f_{25}^{\sharp} \otimes f_{2356}^{\sharp} \otimes f_{36}^{\sharp}}"{description}, from=6-5, to=7-5]
\end{tikzcd}\]

In the diagram above, the top and bottom pentagons commute by hypothesis on $s$ and $t$, and the middle part commutes thanks to the map $f_{25}^{\sharp}$ being deterministic.

Similarly, in the diagram below, the top and bottom polygons commute thanks to $s$ and $t$ being squares, and the central part commutes thanks to the map $f_{25}$ being deterministic:

\[\begin{tikzcd}
	{c_1 a_4} &&& {c_{12}c_2 a_{45}a_5} \\
	{c_4a_4} &&& {c_{12} a_{45}c_2a_5} \\
	{c_{45}c_5a_{45}a_5} &&& {c_{12}c_2 a_{45}a_5 \,\,c_2a_5} \\
	&&& {c_{12}c_2 a_{45}a_5 \, \, c_{23}c_3 a_{56}a_6} \\
	&&& {c_{12}c_2 c_{23}c_3 a_{45}a_5 a_{56}a_6} \\
	{c_{45}c_5a_{45}a_5 \, \, c_5 a_5} &&& {c_{45}c_5 c_{56}c_6 a_{45}a_5 a_{56}a_6}
	\arrow["s"{description}, from=1-1, to=1-4]
	\arrow["{f_{14} \otimes a_4}"{description}, from=1-1, to=2-1]
	\arrow["\sigma"{description}, from=1-4, to=2-4]
	\arrow["{f_{1245} \otimes f_{25} \otimes a_{45}a_5}"{description}, from=1-4, to=3-1]
	\arrow["{f_{45}^{\flat}}"{description}, from=2-1, to=3-1]
	\arrow["{copy_{c_2a_5}; \sigma}"{description}, from=2-4, to=3-4]
	\arrow["{\sigma; copy_{c_5 a_5}}"{description}, from=3-1, to=6-1]
	\arrow["{c_{12}c_2 a_{45}a_5 \otimes t}"{description}, from=3-4, to=4-4]
	\arrow["{f_{1245} \otimes f_{25} \otimes a_{45}a_5 \otimes f_{25} \otimes a_5}"{description}, from=3-4, to=6-1]
	\arrow["\sigma"{description}, from=4-4, to=5-4]
	\arrow["{f_{1245} \otimes f_{25} \otimes f_{2356} \otimes f_{36}\otimes a_{45}a_5 a_{56}a_6}"{description}, from=5-4, to=6-4]
	\arrow["{c_{45}c_5a_{45}a_5 \otimes f_{56}^{\flat}; \sigma}"', from=6-1, to=6-4]
\end{tikzcd}\]

Commutativity of the third square (the topright one in the definition of $xy$-squares, which doesn't involve the map $s$) is proved similarly; we leave this to the reader.
\end{proof}

		Associativity of this $x$-composition for $xy$-squares follows from associativity of copy-composition, which is itself a consequence of associativity of copying, along with functoriality of the tensor product.

		Recall that $x$-composition of $x$-morphisms has no identities; thus, there are no "$x$-identity $xy$-squares" either.

		As above, the use of copy-composition shall be justified \emph{a posteriori}, through the proof of $xy$-interchange, see Proposition \ref{prop_interchange_lenses}.
		
		\item\label{enum_xy_squares_ArenaC} Composition of $xy$ squares in the $y$ direction is more involved. Let $s$, $t$ be $y$-composable squares as below:

\[\begin{tikzcd}
	\begin{array}{c} {\begin{pmatrix} a_1 \\ c_1  \\ \end{pmatrix}} \end{array} && \begin{array}{c} {\begin{pmatrix} a_2 \\ c_2  \\ \end{pmatrix}} \end{array} \\
	& s \\
	\begin{array}{c} {\begin{pmatrix} a_3 \\ c_3  \\ \end{pmatrix}} \end{array} && \begin{array}{c} {\begin{pmatrix} a_4 \\ c_4  \\ \end{pmatrix}} \end{array} \\
	\begin{array}{c} {\begin{pmatrix} a_3 \\ c_3  \\ \end{pmatrix}} \end{array} && \begin{array}{c} {\begin{pmatrix} a_4 \\ c_4  \\ \end{pmatrix}} \end{array} \\
	& t \\
	\begin{array}{c} {\begin{pmatrix} a_5 \\ c_5  \\ \end{pmatrix}} \end{array} && \begin{array}{c} {\begin{pmatrix} a_6 \\ c_6  \\ \end{pmatrix}} \end{array}
	\arrow["{(a_{12}, c_{12}, f_{12}, f^{\flat}_{12})}", shift left=2, from=1-1, to=1-3]
	\arrow[shift right=2, from=1-1, to=1-3]
	\arrow["{(f_{13}, f_{13}^{\sharp})}"', shift right=2, from=1-1, to=3-1]
	\arrow[shift right=2, from=1-3, to=3-3]
	\arrow[shift right=2, from=3-1, to=1-1]
	\arrow[shift left=2, from=3-1, to=3-3]
	\arrow["{(f_{34}, f^{\flat}_{34})}"', shift right=2, from=3-1, to=3-3]
	\arrow["{(f_{24}, f_{24}^{\sharp})}"', shift right=2, from=3-3, to=1-3]
	\arrow["{(a_{34}, c_{34}, f_{34}, f_{34}^{\flat})}", shift left=2, from=4-1, to=4-3]
	\arrow[shift right=2, from=4-1, to=4-3]
	\arrow["{(f_{35}, f_{35}^{\sharp})}"', shift right=2, from=4-1, to=6-1]
	\arrow[shift right=2, from=4-3, to=6-3]
	\arrow[shift right=2, from=6-1, to=4-1]
	\arrow[shift left=2, from=6-1, to=6-3]
	\arrow["{(a_{56}, c_{56}, f_{56}, f_{56}^{\flat})}"', shift right=2, from=6-1, to=6-3]
	\arrow["{(f_{46}, f_{46}^{\sharp})}"', shift right=2, from=6-3, to=4-3]
\end{tikzcd}\]

		The $y$-morphism $\begin{pmatrix} a_{12} \\ c_{12}  \\ \end{pmatrix} \leftrightarrows \begin{pmatrix} a_{56} \\ c_{56}  \\ \end{pmatrix}$ component of $\frac{s}{t}$ is simply the composite of the (deterministic parametric) lenses $\begin{pmatrix} a_{12} \\ c_{12}  \\ \end{pmatrix} \leftrightarrows \begin{pmatrix} a_{34} \\ c_{34}  \\ \end{pmatrix}$ and $\begin{pmatrix} a_{34} \\ c_{34}  \\ \end{pmatrix} \leftrightarrows \begin{pmatrix} a_{56} \\ c_{56}  \\ \end{pmatrix}$ that are part of the data of $s$ and $t$ respectively.

		Now, let us define the map $\frac{s}{t}: c_1 a_5 \rightarrow c_{12} c_2 a_{56} a_6$, using conditional products. We begin with the following

		\begin{claim}\label{claim_diag_comm_arena_y_comp}
			The diagram below commutes:
			
\[\begin{tikzcd}
	{ c_1 a_5} & {  c_3 a_5} \\
	{  c_1 c_3 a_5} & { c_{34}c_4 a_{56}a_6} \\
	{ c_1 a_3} & { c_{34}c_4 a_{56}a_6 \otimes   c_{34}c_4 a_{56}a_6} \\
	{ c_{12}c_2 a_{34}a_4} & {c_{34}c_4 a_{34}a_4 \otimes  a_{56}a_6} \\
	{ c_{12} c_2 \otimes c_{34}c_4 a_{34}a_4} & {c_{34}c_4 a_{34}a_4}
	\arrow["{  f_{13} \otimes a_5}", from=1-1, to=1-2]
	\arrow["{  (copy_{c_1} ; f_{13}) \otimes a_5}"{description}, from=1-1, to=2-1]
	\arrow["{t}", from=1-2, to=2-2]
	\arrow["{\sigma ; c_1 \otimes f_{35}^{\sharp} }", from=2-1, to=3-1]
	\arrow["{copy_{ c_{34}c_4 a_{56}a_6}}"{description}, from=2-2, to=3-2]
	\arrow["{s_{}}", from=3-1, to=4-1]
	\arrow["{ c_{34}c_4 a_{56}a_6 \otimes (\sigma; f_{3456}^{\sharp} \otimes f_{46}^{\sharp}); \sigma}"{description}, from=3-2, to=4-2]
	\arrow["{ (copy_{c_{12}c_2}; c_{12}c_2 \otimes f_{1234} \otimes f_{24}) \otimes a_{34} a_4}"{description}, from=4-1, to=5-1]
	\arrow["\pi", from=4-2, to=5-2]
	\arrow["\pi"{description}, from=5-1, to=5-2]
\end{tikzcd}\]
			
		\end{claim}
		\begin{proof}[Proof notes.]
			This relies on the fact that $s$ and $t$ are squares, with the bottom $x$-morphism in $s$ being the same as the top $x$-morphism in $t$.
		\end{proof}

		Then, let $\alpha :   c_1 a_5 \rightarrow  c_{12} c_2 \otimes c_{34}c_4 a_{34}a_4 \otimes  a_{56}a_6$ be the morphism defined as the conditional product, over $c_{34} c_4 a_{34} a_4$, of the composites $  c_1 a_5\rightarrow c_{12} c_2\otimes c_{34}c_4 a_{34}a_4$ and $  c_1 a_5 \rightarrow c_{34}c_4 a_{34}a_4 \otimes  a_{56}a_6$ in the diagram above.

		Finally, let $\frac{s}{t} :   c_1 a_5 \rightarrow     c_{12} c_2 a_{56} a_6$ be defined as the composite $\alpha; del_{c_{34} c_4 a_{34} a_4}$, i.e. the marginal of the conditional product above.

		\begin{claim}\label{claim_regen_for_y_comp}
			The map $\alpha: c_1 a_5 \rightarrow  c_{12} c_2 \otimes c_{34}c_4 a_{34}a_4 \otimes  a_{56}a_6$ is equal to the composite $\frac{s}{t}; (copy_{c_{12} c_2}; c_{12} c_2 \otimes f_{1234} \otimes f_{24}) \otimes a_{56} a_6 \, ; c_{12} c_2 \otimes (copy_{c_{34} c_4 a_{56} a_6} ; f_{3456}^{\sharp} f_{46}^{\sharp})$.
		\end{claim}

		\begin{proof}
			We use Lemma \ref{lemma_CP_and_regen} twice; let us first apply it with $X= c_1 a_5$, $A = c_{12} c_2$, $B= c_{34}c_4 a_{34}a_4$, $C = a_{56}a_6$, with the maps $\phi: X \rightarrow A \otimes B$ and $\psi: X \rightarrow B \otimes C$ defined as the composites $  c_1 a_5\rightarrow c_{12} c_2\otimes c_{34}c_4 a_{34}a_4$ and $  c_1 a_5 \rightarrow c_{34}c_4 a_{34}a_4 \otimes  a_{56}a_6$ in the diagram of Claim \ref{claim_diag_comm_arena_y_comp}. The map $f: A \otimes B \rightarrow A \otimes B$ is the composite $c_{12} c_2 \,\, c_{34}c_4 a_{34}a_4 \xrightarrow{del_{c_{34}c_4}} c_{12} c_2 \,\, a_{34}a_4 \xrightarrow{(copy_{c_{12}c_2}; f_{1234}\otimes f_{24}) \otimes a_{34} a_4} c_{12} c_2 \,\, c_{34}c_4 a_{34}a_4  $. By Lemma \ref{lemma_det_almost_sure_equality}, the map $f$ satisfies the "almost sure equality" hypothesis of Lemma \ref{lemma_CP_and_regen}.

			We then use the symmetric statement of Lemma \ref{lemma_CP_and_regen}, with the same setting, except that we consider $g: B \otimes C \rightarrow B \otimes C$, defined as the composite $c_{34}c_4 a_{34}a_4 \,\,  a_{56}a_6 \xrightarrow{del_{a_{34}a_4}} c_{34}c_4  \,\,  a_{56}a_6 \xrightarrow{copy_{c_{34}c_4  a_{56}a_6}} c_{34}c_4  \, a_{56}a_6 \, c_{34}c_4  \, a_{56}a_6 \xrightarrow{\sigma; c_{34}c_4 \otimes f_{3456}^{\sharp} \otimes f_{46}^{\sharp} \otimes a_{56} a_6}c_{34}c_4 a_{34}a_4 \,\,  a_{56}a_6$.

			We conclude by sliding/functoriality of the tensor product. \end{proof}

		\begin{claim}\label{claim_y_comp_well_def_arena_moore}
			The composite $\frac{s}{t}$ makes the required squares commute.
		\end{claim}
		\begin{proof}
			This follows from the fact that the marginals of a conditional product are, by construction, equal to its factors; of course, we also use the hypothesis that $s$ and $t$ are $xy$-squares.

			Let us start with the top-left commuting square in the definition of $xy$-squares:

			\begin{center}
\[\begin{tikzcd}
	{c_1 a_5} &&&& {c_{12}c_2 \,\, a_{56} a_6} \\
	&&& {c_{12}c_2 \, c_{34} c_4 a_{34} a_4 \,  a_{56} a_6} \\
	&&& {c_{12}c_2 \, a_{34} a_4 \,  a_{56} a_6} \\
	&& {c_{12}c_2 \, c_{34} c_4 a_{34} a_4 \,  a_{56} a_6} \\
	\\
	{c_3 a_5} &&&& {c_{34}c_4 \,\, a_{56} a_6} \\
	\\
	{c_5 a_5} &&&& {c_{56}c_6 \,\, a_{56} a_6}
	\arrow["{\frac{s}{t}}", from=1-1, to=1-5]
	\arrow["\alpha", from=1-1, to=2-4]
	\arrow["\alpha", from=1-1, to=4-3]
	\arrow["{f_{13} \otimes a_5}"{description}, from=1-1, to=6-1]
	\arrow["{f_{1234} \otimes f_{24} \otimes a_{56} a_6}"{description}, from=1-5, to=6-5]
	\arrow["\pi", from=2-4, to=1-5]
	\arrow["{del_{c_{34} c_4}}"{description}, from=2-4, to=3-4]
	\arrow["{copy_{c_{12}c_2}; c_{12} c_2 \otimes f_{1234} \otimes f_{24} \otimes a_{56} a_6}"{description}, from=3-4, to=4-3]
	\arrow["\pi", from=4-3, to=6-5]
	\arrow["t"', from=6-1, to=6-5]
	\arrow["{f_{35} \otimes a_5}"{description}, from=6-1, to=8-1]
	\arrow["{f_{3456} \otimes f_{46} \otimes a_{56} a_6}"{description}, from=6-5, to=8-5]
	\arrow["{f_{56}^{\flat}}"', from=8-1, to=8-5]
\end{tikzcd}\]
			\end{center}

In the diagram above, the bottom rectangle commutes because $t$ is a square. Let us now focus on the inside of the top outer rectangle. The top triangle commutes by definition of $\frac{s}{t}$ as the marginal of the map $\alpha$. The bottom-left irregular quadrilateral commutes because the marginal of $\alpha$ on $c_{34} c_4 a_{56} a_6$ is prescribed (recall that $\alpha$ is a conditional product). The irregular pentagon on the left commutes thanks to basic properties of $copy$ and $del$. Finally, the irregular quadrilateral in the middle commutes thanks to Claim \ref{claim_regen_for_y_comp}, i.e. thanks to Lemma \ref{lemma_CP_and_regen}.

Then, for the top-right square in the definition of $xy$-squares, it suffices to concatenate, and invoke functoriality of the tensor product:
\begin{center}
	
\[\begin{tikzcd}
	{c_1} && {c_{12} c_2} \\
	{c_3} && {c_{34} c_4} \\
	{c_5} && {c_{56} c_6}
	\arrow["{f_{12}}", from=1-1, to=1-3]
	\arrow["{f_{13}}"{description}, from=1-1, to=2-1]
	\arrow["{f_{1234} \otimes f_{24}}"{description}, from=1-3, to=2-3]
	\arrow["{f_{34}}", from=2-1, to=2-3]
	\arrow["{f_{35}}"{description}, from=2-1, to=3-1]
	\arrow["{f_{3456} \otimes f_{46}}"{description}, from=2-3, to=3-3]
	\arrow["{f_{56}}", from=3-1, to=3-3]
\end{tikzcd}\]
\end{center}

Finally, let us consider the bottom commuting square in the definition of $xy$-squares.
\begin{center}
	
\[\begin{tikzcd}
	{c_1 a_5} &&&& {c_{12}c_2 \,\, a_{56} a_6} \\
	&& {c_{12}c_2 \,\, c_{34} c_4 a_{34} a_4\,\, a_{56} a_6} \\
	{c_1 c_1 a_5} &&&& {c_{12} c_2 \,\, c_{12}c_2 \,\, a_{56} a_6} \\
	{c_1 c_3 a_5} & {c_{12}c_2 \,\, c_{34} c_4 a_{34} a_4\,\, a_{56} a_6} \\
	&&&& {c_{12}c_2 \,\, c_{34} c_4 \,\, a_{56} a_6} \\
	{c_1 a_3} &&&& {c_{12}c_2 \,\, a_{34} a_4} \\
	\\
	{c_1 a_1} &&&& {c_{12}c_2 \,\, a_{12} a_2}
	\arrow["{\frac{s}{t}}", from=1-1, to=1-5]
	\arrow["\alpha"{description}, from=1-1, to=2-3]
	\arrow["{copy_{c_1}}"{description}, from=1-1, to=3-1]
	\arrow["\alpha"{description}, from=1-1, to=4-2]
	\arrow["{copy_{c_{12} c_2}}"{description}, from=1-5, to=3-5]
	\arrow["\pi"', shift right=2, from=2-3, to=5-5]
	\arrow["{c_1 \otimes f_{13} \otimes a_5}"{description}, from=3-1, to=4-1]
	\arrow["{c_{12} c_2 \otimes f_{1234} \otimes f_{24} \otimes a_{56} a_6}"{description}, from=3-5, to=5-5]
	\arrow["{c_1 \otimes f_{35}^{\sharp}}"{description}, from=4-1, to=6-1]
	\arrow["\pi"', from=4-2, to=6-5]
	\arrow["{c_{12} c_2 \otimes (\sigma; f_{3456}^{\sharp} \otimes f_{46}^{\sharp})}"{description}, from=5-5, to=6-5]
	\arrow["s"', from=6-1, to=6-5]
	\arrow["{copy_{c_1}; c_1 \otimes f_{13}^{\sharp}}"{description}, from=6-1, to=8-1]
	\arrow["{copy_{c_{12} c_1}; \sigma; c_{12} c_2 \otimes f_{1234}^{\sharp} \otimes f_{24}^{\sharp}}"{description}, from=6-5, to=8-5]
	\arrow["{f_{12}^{\flat}}"', from=8-1, to=8-5]
\end{tikzcd}\]

\end{center}

In the diagram above, the bottom rectangle commutes thanks to $s$ being an $xy$-square. Let us then focus on the top part of this diagram. The bottom-left polygon commutes thanks to $\alpha$ being a conditional product (so its marginals are known). The top-right polygon commutes thanks to Lemma \ref{lemma_CP_and_regen}, which can be applied because the maps $f_{1234}$ and $f_{24}$ are deterministic. Finally, the middle pentagon commutes thanks to Lemma \ref{lemma_CP_and_regen} again, relying on determinism of the maps $f_{3456}^{\sharp}$ and $f_{46}^{\sharp}$.
		\end{proof}

		\begin{remark}
			Here, we did not need the maps $f_{1234}^{\sharp}$ and $f_{24}^{\sharp}$ to be deterministic. This observation will be crucial when constructing the triple category $ArenaSys^{\mathrm{Moore}}_{\CC}$. See Claim \ref{claim_y_comp_well_def_arenasys}.
		\end{remark}
		
		\begin{claim}\label{claim_y_comp_assoc_arena_moore}
			This $y$-composition of $xy$ squares is associative, and there are identity squares.
		\end{claim}
		\begin{proof}[Proof notes.]
			Identity squares (for $y$-composition) are defined using identities of $\CC$. Associativity follows from associativity of lens composition, associativity of conditional products \cite[Proposition 4.3.4]{gleaves} and compatibility of conditional products with precomposing by deterministic maps. What we use is the following: given a composition setting such as depicted below

\[\begin{tikzcd}
	\begin{array}{c} {\begin{pmatrix} a_1 \\ c_1  \\ \end{pmatrix}} \end{array} && \begin{array}{c} {\begin{pmatrix} a_2 \\ c_2  \\ \end{pmatrix}} \end{array} \\
	& s \\
	\begin{array}{c} {\begin{pmatrix} a_3 \\ c_3  \\ \end{pmatrix}} \end{array} && \begin{array}{c} {\begin{pmatrix} a_4 \\ c_4  \\ \end{pmatrix}} \end{array} \\
	\begin{array}{c} {\begin{pmatrix} a_3 \\ c_3  \\ \end{pmatrix}} \end{array} && \begin{array}{c} {\begin{pmatrix} a_4 \\ c_4  \\ \end{pmatrix}} \end{array} \\
	& t \\
	\begin{array}{c} {\begin{pmatrix} a_5 \\ c_5  \\ \end{pmatrix}} \end{array} && \begin{array}{c} {\begin{pmatrix} a_6 \\ c_6  \\ \end{pmatrix}} \end{array} \\
	\begin{array}{c} {\begin{pmatrix} a_5 \\ c_5  \\ \end{pmatrix}} \end{array} && \begin{array}{c} {\begin{pmatrix} a_6 \\ c_6  \\ \end{pmatrix}} \end{array} \\
	& u \\
	\begin{array}{c} {\begin{pmatrix} a_7 \\ c_7  \\ \end{pmatrix}} \end{array} && \begin{array}{c} {\begin{pmatrix} a_8 \\ c_8  \\ \end{pmatrix}} \end{array}
	\arrow["{(a_{12}, c_{12}, f_{12}, f^{\flat}_{12})}", shift left=2, from=1-1, to=1-3]
	\arrow[shift right=2, from=1-1, to=1-3]
	\arrow["{(f_{13}, f_{13}^{\sharp})}"', shift right=2, from=1-1, to=3-1]
	\arrow[shift right=2, from=1-3, to=3-3]
	\arrow[shift right=2, from=3-1, to=1-1]
	\arrow[shift left=2, from=3-1, to=3-3]
	\arrow["{(a_{34}, c_{34}, f_{34}, f^{\flat}_{34})}"', shift right=2, from=3-1, to=3-3]
	\arrow["{(f_{24}, f_{24}^{\sharp})}"', shift right=2, from=3-3, to=1-3]
	\arrow["{(a_{34}, c_{34}, f_{34}, f_{34}^{\flat})}", shift left=2, from=4-1, to=4-3]
	\arrow[shift right=2, from=4-1, to=4-3]
	\arrow["{(f_{35}, f_{35}^{\sharp})}"', shift right=2, from=4-1, to=6-1]
	\arrow[shift right=2, from=4-3, to=6-3]
	\arrow[shift right=2, from=6-1, to=4-1]
	\arrow[shift left=2, from=6-1, to=6-3]
	\arrow["{(a_{56}, c_{56}, f_{56}, f_{56}^{\flat})}"', shift right=2, from=6-1, to=6-3]
	\arrow["{(f_{46}, f_{46}^{\sharp})}"', shift right=2, from=6-3, to=4-3]
	\arrow["{(a_{56}, c_{56}, f_{56}, f_{56}^{\flat})}"', shift right=2, from=7-1, to=7-3]
	\arrow[shift left=2, from=7-1, to=7-3]
	\arrow["{(f_{57}, f_{57}^{\sharp})}"', shift right=2, from=7-1, to=9-1]
	\arrow[shift right=2, from=7-3, to=9-3]
	\arrow[shift right=2, from=9-1, to=7-1]
	\arrow["{(a_{78}, c_{78}, f_{78}, f_{78}^{\flat})}"', shift right=2, from=9-1, to=9-3]
	\arrow[shift left=2, from=9-1, to=9-3]
	\arrow["{(f_{68}, f_{68}^{\sharp})}"', shift right=2, from=9-3, to=7-3]
\end{tikzcd}\]

the maps $f_{13}$, $f_{35}$, $f_{57}$, $f_{35}^{\sharp}$, $f_{57}^{\sharp}$ are, by definition, deterministic. Thus, the conditional products used to define $\frac{s}{t}$, $\frac{t}{u}$, when precomposed with those maps and with copy maps, remain conditional products. From this observation, one can then show that both $\frac{\frac{s}{t}}{u}$ and $\frac{s}{\frac{t}{u}}$ are marginals of  the following threefold conditional product: $$(f_s; s') \otimes_{c_{34}c_4 a_{34} a_4} (f_t; t') \otimes_{c_{56} c_6 a_{56} a_6} (f_u ; u'),$$
where:

\begin{enumerate}
	\item The maps $f_s: c_1 a_7 \rightarrow c_1 a_3$, $f_t: c_1 a_7 \rightarrow c_3 a_5$, $f_u: c_1 a_7 \rightarrow c_5 a_7$, are deterministic.
	
	\item We have $f_s = copy_{c_1} \otimes a_7; c_1 \otimes f_{13}\otimes a_7; c_1 \otimes copy_{c_3} \otimes a_7; c_1 \otimes c_3 \otimes f_{35} \otimes a_7;  c_1 \otimes c_3 \otimes f_{57}^{\sharp} ; c_1 \otimes f_{35}^{\sharp};$
	\item We have $f_t = f_{13} \otimes a_7; copy_{c_3} \otimes a_7; c_3 \otimes f_{35} \otimes a_7; c_3 \otimes f_{57}^{\sharp};$
	\item We have $f_u = (f_{13}; f_{35}) \otimes a_7$

	\item The maps $s':c_1 a_3 \rightarrow c_{12}c_2 \otimes c_{34}c_4 a_{34} a_4$, $t': c_3 a_5 \rightarrow c_{34}c_4 a_{34} a_4 \otimes c_{56}c_6 a_{56} a_6$ and $u': c_5 a_7 \rightarrow c_{56}c_6 a_{56} a_6 \otimes a_{78} a_8$, are defined by postcomposing $s$, $t$, and $u$, respectively, with deterministic maps built from $copy$, $f_{1234}$, $f_{24}$, $f_{3456}$, $f_{56}$, $f_{3456}^{\sharp}$, $f_{46}^{\sharp}$, $f_{5678}^{\sharp}$, and $f_{68}^{\sharp}$. 
	
	For instance, since $s: c_1 a_3 \rightarrow c_{12} c_2 a_{34} a_4$ and $s': c_1 a_3 \rightarrow c_{12}c_2 \otimes c_{34}c_4 a_{34} a_4$, one is looking for a deterministic map $c_{12} c_2 a_{34} a_4 \rightarrow c_{12}c_2 \otimes c_{34}c_4 a_{34} a_4$; such a map can be defined using $copy$, $f_{1234}$ and $f_{24}$. We leave the details to the reader.

\end{enumerate}

This, along with associativity of composition for $y$-morphisms/lenses, shows associativity of $y$-composition for $xy$-squares. \end{proof}

		\begin{prop}\label{prop_interchange_lenses}
			The $xy$ interchange law holds. 
		\end{prop}
		\begin{proof}
			Let $s,t,u,v$ be composable squares as in the diagram below:

\[\begin{tikzcd}
	\begin{array}{c} {\begin{pmatrix} a_1 \\ c_1  \\ \end{pmatrix}} \end{array} && \begin{array}{c} {\begin{pmatrix} a_2 \\ c_2  \\ \end{pmatrix}} \end{array} && \begin{array}{c} {\begin{pmatrix} a_3 \\ c_3  \\ \end{pmatrix}} \end{array} \\
	& s && u \\
	\begin{array}{c} {\begin{pmatrix} a_4 \\ c_4  \\ \end{pmatrix}} \end{array} && \begin{array}{c} {\begin{pmatrix} a_5 \\ c_5 \\ \end{pmatrix}} \end{array} && \begin{array}{c} {\begin{pmatrix} a_6 \\ c_6  \\ \end{pmatrix}} \end{array} \\
	& t && v \\
	\begin{array}{c} {\begin{pmatrix} a_7 \\ c_7  \\ \end{pmatrix}} \end{array} && \begin{array}{c} {\begin{pmatrix} a_8 \\ c_8  \\ \end{pmatrix}} \end{array} && \begin{array}{c} {\begin{pmatrix} a_9 \\ c_9  \\ \end{pmatrix}} \end{array}
	\arrow["{(a_{12}, c_{12},f_{12}, f_{12}^{\flat})}", shift left=2, from=1-1, to=1-3]
	\arrow[shift right=2, from=1-1, to=1-3]
	\arrow["{(f_{14}, f_{14}^{\sharp})}"', shift right=2, from=1-1, to=3-1]
	\arrow["{(a_{23}, c_{23},f_{23}, f_{23}^{\flat})}", shift left=2, from=1-3, to=1-5]
	\arrow[shift right=2, from=1-3, to=1-5]
	\arrow["{(f_{25}, f_{25}^{\sharp})}"{description}, shift right=2, from=1-3, to=3-3]
	\arrow[shift right=2, from=1-5, to=3-5]
	\arrow[shift right=2, from=3-1, to=1-1]
	\arrow["{(a_{45}, c_{45},f_{45}, f_{45}^{\flat})}", shift left=2, from=3-1, to=3-3]
	\arrow[shift right=2, from=3-1, to=3-3]
	\arrow["{(f_{47}, f_{47}^{\sharp})}"', shift right=2, from=3-1, to=5-1]
	\arrow[shift right=2, from=3-3, to=1-3]
	\arrow["{(a_{56}, c_{56},f_{56}, f_{56}^{\flat})}", shift left=2, from=3-3, to=3-5]
	\arrow[shift right=2, from=3-3, to=3-5]
	\arrow["{(f_{58}, f_{58}^{\sharp})}"{description}, shift right=2, from=3-3, to=5-3]
	\arrow["{(f_{36}, f_{36}^{\sharp})}"', shift right=2, from=3-5, to=1-5]
	\arrow[shift right=2, from=3-5, to=5-5]
	\arrow[shift right=2, from=5-1, to=3-1]
	\arrow["{(a_{78}, c_{78},f_{78}, f_{78}^{\flat})}", shift left=2, from=5-1, to=5-3]
	\arrow[shift right=2, from=5-1, to=5-3]
	\arrow[shift right=2, from=5-3, to=3-3]
	\arrow["{(a_{89}, c_{89},f_{89}, f_{89}^{\flat})}", shift left=2, from=5-3, to=5-5]
	\arrow[shift right=2, from=5-3, to=5-5]
	\arrow["{(f_{69}, f_{69}^{\sharp})}"', shift right=2, from=5-5, to=3-5]
\end{tikzcd}\]

			We wish to show that the $xy$-squares $\frac{s | u}{t | v}$ and $\frac{s}{t} | \frac{u}{v}$ are equal. First note that, by functoriality of the tensor product, the $y$-morphism components of $\frac{s | u}{t | v}$ and $\frac{s}{t} | \frac{u}{v}$, of signature ${\begin{pmatrix} a_{12}a_2 a_{23} \\ c_{12} c_2 c_{23}  \\ \end{pmatrix}} \leftrightarrows {\begin{pmatrix} a_{78}a_8 a_{89} \\ c_{78} c_8 c_{89}  \\ \end{pmatrix}} $, are equal.

			 Now, the definitions yield two Markov morphisms, namely $\frac{s | u}{t | v}$ and $\frac{s}{t} | \frac{u}{v}$, both of which are of signature $  c_1 a_7 \rightarrow     c_{12} c_2 c_{23} c_3 \,\, a_{78} a_8 a_{89} a_9$. To simplify notations, let $\star$ denote copy-composition, used, among other things, for $x$-composition of $x$-morphisms and $xy$-squares. Abusing notations, we omit the object over which copy-composition is performed; we assume it is clear from the context.
			 Let $f_{s}$ denote the \emph{deterministic} map $(copy_{c_1}; c_1 \otimes f_{14} \otimes a_7; c_1 \otimes f_{47}^{\sharp})$, whose signature is $c_1 a_7 \rightarrow c_1 a_4$. Let $c_{46} = c_{45} c_5 c_{56}$ and $a_{46} = a_{45} a_5 a_{56}$.

			 \begin{claim}\label{claim_interchange_Phi}
				With the notations above, the map $\frac{s | u}{t | v} = (s;_x u);_y (t;_x v)$ is the marginal on $c_{12} c_2 c_{23} c_3 \,\, a_{78} a_8 a_{89} a_9$ of the map $\Phi: c_1 a_7 \rightarrow c_{12} c_2 c_{23} c_3 \otimes c_{46} c_6 a_{46} a_6 \otimes a_{78} a_8 a_{89} a_9$, defined as  the following conditional product: $$(f_s; s \star u \star (f_{1245} \otimes f_{25} \otimes f_{2356} \otimes f_{36})) \otimes_{c_{46}c_6 a_{46}a_6} (f_{14} \otimes a_7; t \star v \star (f_{4578}^{\sharp} \otimes f_{58}^{\sharp} \otimes f_{5689}^{\sharp} \otimes f_{69}^{\sharp})).$$
			 \end{claim}

			 \begin{proof}
				This follows from the definitions of $x$- and $y$- compositions for $xy$-squares, along with Claim \ref{claim_regen_for_y_comp}.
			 \end{proof}

			 Note that all the maps involved in the definition of the conditional product in the Claim above are deterministic, except possibly $s$, $t$, $u$, $v$, and the common marginal $c_1 a_7 \xrightarrow{(f_{14}; copy_{c_4})\otimes a_7 ; c_4 \otimes f_{47}^{\sharp}} c_4 a_4 \xrightarrow{f_{45}^{\flat} \star f_{56}^{\flat}} c_{46} c_6 a_{46} a_6$. \emph{Also note that the maps $f_{1245}^{\sharp}$, $f_{25}^{\sharp}$, $f_{2356}^{\sharp}$ and $f_{36}^{\sharp}$ do not appear in the definition of $\Phi$.}

			 Now, let us consider $\frac{s}{t} | \frac{u}{v}$. Let $\alpha: c_1 a_7 \rightarrow c_{12} c_2 \otimes c_{45} c_5 a_{45} a_5 \otimes a_{78} a_8$ be the conditional product  whose marginal is, by definition, the map $\frac{s}{t}$. By Claim \ref{claim_regen_for_y_comp}, we know that $\alpha = \frac{s}{t}; f_{s,t}$, where $f_{s,t}: c_{12} c_2 a_{78} a_8 \rightarrow c_{12} c_2 \otimes c_{45}c_5 a_{45} a_5 \otimes a_{78} a_8$ is "the" deterministic map built using $copy$, $f_{1245}$, $f_{25}$, $f_{4578}^{\sharp}$ and $f_{58}^{\sharp}$.

			 \noindent Similarly, let $\beta: c_2 a_8 \rightarrow c_{23} c_3 \otimes c_{56} c_6 a_{56} a_6 \otimes a_{89} a_9$ be equal to $\frac{u}{v}; f_{u, v}$, where $f_{u,v}: c_{23} c_3 \otimes a_{89} a_9 \rightarrow c_{23} c_3 \otimes c_{56} c_6 a_{56} a_6 \otimes a_{89} a_9$ is "the" deterministic map built using $copy$, $f_{2356}$, $f_{36}$, $f_{5689}^{\sharp}$ and $f_{69}^{\sharp}$.

			 \begin{claim}\label{claim_interchange_Psi}

				The map $\frac{s}{t} | \frac{u}{v} = (s;_y t) ;_x (u;_y v)$ is the marginal on $c_{12} c_2 c_{23} c_3 \,\, a_{78} a_8 a_{89} a_9$ of the map $\Psi: c_1 a_7 \rightarrow c_{12} c_2 c_{23} c_3 \otimes c_{46} c_6 a_{46} a_6 \otimes a_{78} a_8 a_{89} a_9$ defined as the following composite: $c_1 a_7 \xrightarrow{\alpha} c_{12} c_2 \otimes c_{45} c_5 a_{45} a_5 \otimes a_{78} a_8  \xrightarrow{\sigma; copy_{c_2 a_8}} c_{12} c_2 \otimes c_{45} c_5 a_{45} a_5 \otimes c_2 a_8 \otimes a_{78} a_8 \xrightarrow{c_{12} c_2 \otimes c_{45} c_5 a_{45} a_5 \otimes \beta \otimes a_{78} a_8} c_{12} c_2 \otimes c_{45} c_5 a_{45} a_5 \otimes (c_{23} c_3 c_{56} c_6 a_{56} a_6  a_{89} a_9) \otimes a_{78} a_8 \xrightarrow{\sigma} c_{12} c_2 c_{23} c_3  c_{46} c_6 a_{46} a_6 a_{78} a_8 a_{89} a_9$.

				(Recall that $c_{46} = c_{45} c_5 c_{56}$ and $a_{46} = a_{45} a_5 a_{56}$.)

			 \end{claim}

			 \begin{proof}
				This follows from the definition of $x$-composition of $xy$-squares as copy-composition.
			 \end{proof}

			 \begin{claim}\label{claim_interchange_same_marginals}
				The maps $\Phi$ and $\Psi$ have the same marginals on $c_{12} c_2 c_{23} c_3 \otimes c_{46} c_6 a_{46} a_6$ and $c_{46} c_6 a_{46} a_6 \otimes a_{78} a_8 a_{89} a_9$.
			 \end{claim}

			 \begin{proof}
				This relies on the fact that the marginals of a conditional product are the maps comprising of it, along with the "postcomposition" property of Claim \ref{claim_regen_for_y_comp}.
			 \end{proof}

			 To conclude this proof, let us show that the maps $\Phi$ and $\Psi$ are equal. We use properties of conditional products: 	since, by Claim \ref{claim_interchange_same_marginals}, the maps $\Phi$ and $\Psi$ have the same marginals, and $\Phi$ is a conditional product over $c_{45} c_5 c_{56} c_6 a_{45} a_5 a_{56} a_6$, it suffices to show that $\Psi$ is also a conditional product over $c_{45} c_5 c_{56} c_6 a_{45} a_5 a_{56} a_6$. Let us now prove that.

 			Let $f_{46}^{\flat}: c_4 a_4 \rightarrow c_{46} c_6 a_{46} a_6$ denote the copy-composite $f_{45}^{\flat} \star f_{56}^{\flat}$. 
			 Let $s': c_1 a_4 \rightarrow c_{12} c_2 c_{45} c_5 a_{45} a_5$ denote the copy-composite $s \star (f_{1245} \otimes f_{25})$. Similarly, let $t': c_4 a_7 \rightarrow c_{45} c_5 a_{45} a_5 a_{78} a_8$, $u': c_2 a_5 \rightarrow c_{23} c_3 c_{56} c_6 a_{56} a_6$, and $v': c_5 a_8 \rightarrow c_{56} c_6 a_{56} a_6 a_{89} a_9$, denote the copy-composites of $t$, $u$, $v$, respectively, with the appropriate \emph{deterministic} maps.

			 By definition of the maps $\alpha$ and $\beta$ as conditional products, the map $\Psi$ is, up to postcomposing with a suitable symmetry, equal to the following:

			\ctikzfig{interchange_0}

			 Here, we wrote $45$ instead of $c_{45} c_5 a_{45} a_5$, and $56$ instead of $c_{56} c_6 a_{56} a_6$; thus we wrote $s'|45$ instead of $s'|c_{45} c_5 a_{45} a_5$, similarly for $t'|45$, $u'|56$ and $v'|56$. 
			 The map $f_4$ denotes $(f_{14} \otimes a_7)\star f_{47}^{\sharp}: c_1 a_7 \rightarrow c_4 a_4$. Similarly, the map $f_5$ denotes $(f_{25} \otimes a_8)\star f_{58}^{\sharp}: c_2 a_8 \rightarrow c_5 a_5$.
			 Now, thanks to Claim \ref{claim_regen_for_y_comp} applied to $s$, $t$ and $\alpha$, we can rewrite the middle of the diagram above, and get that $\Psi$ is, up to permutation, equal to:

			 \ctikzfig{interchange_1}

			 For the same reason, which ultimately follows from determinism of the maps $f_{25}$ and $f_{58}^{\sharp}$ and Lemma \ref{lemma_CP_and_regen}, this morphism is also equal to:

			 \ctikzfig{interchange_2}

			 Then, by sliding, it equals 

			\ctikzfig{interchange_3}

			By associativity of copying, along with counitality of deleting, we can rearrange the center part of the diagram as follows:

			\ctikzfig{interchange_4}

			We now recognize this morphism as a conditional product over $45 \otimes 56$, i.e. over $c_{45} c_5 a_{45} a_5 c_{56} c_6 a_{56} a_6$. This concludes the proof of interchange.	\end{proof}

		\item Consider the boundary of a $yz$ square as below, where the vertical maps are $y$-morphisms, and the horizontal ones are $z$-morphisms:
		
		\begin{center}
			\begin{tikzcd}
				{\begin{pmatrix} a_1 \\ c_1  \\ \end{pmatrix}}  \arrow[d, "{(f_{13}, f_{13}^{\sharp})}", swap, shift right] \arrow[r, "{(f_{12}, g_{12})}", shift left]\arrow[r, shift right] & {\begin{pmatrix} a_2 \\ c_2  \\ \end{pmatrix}}   \arrow[d, shift right]
				\\
				{\begin{pmatrix} a_3 \\ c_3  \\ \end{pmatrix}} \arrow[u, shift right]\arrow[r, "{(f_{34}, g_{34})}", shift left] \arrow[r, shift right] & {\begin{pmatrix} a_4 \\ c_4  \\ \end{pmatrix}}\arrow[u, "{(f_{24}, f_{24}^{\sharp})}", shift right, swap]
			\end{tikzcd}
		\end{center} 
		
		Then, a square with this boundary is given by commutation of the following square in $\CC_{det}$:
		\[\begin{tikzcd}
			{c_1} & {c_2} \\
			{c_3} & {c_4}
			\arrow["{f_{12}}", from=1-1, to=1-2]
			\arrow["{f_{13}}"', from=1-1, to=2-1]
			\arrow["{f_{24}}", from=1-2, to=2-2]
			\arrow["{f_{34}}", from=2-1, to=2-2]
		\end{tikzcd}\]
		
		Composition of $yz$-squares, in the directions $y$ and $z$, is simply concatenation of commuting squares in $\CC_{det}$. Associativity and unitality are straightforward. The $yz$-interchange law holds automatically, because any given boundary has at most one square filling it.

		\item The $xz$-morphisms with boundary as below

		\begin{center}

\[\begin{tikzcd}
	\begin{array}{c} {\begin{pmatrix} a_1 \\ c_1  \\ \end{pmatrix}} \end{array} && \begin{array}{c} {\begin{pmatrix} a_2 \\ c_2  \\ \end{pmatrix}} \end{array} \\
	\begin{array}{c} {\begin{pmatrix} a_3 \\ c_3  \\ \end{pmatrix}} \end{array} && \begin{array}{c} {\begin{pmatrix} a_4 \\ c_4  \\ \end{pmatrix}} \end{array}
	\arrow["{(a_{12}, c_{12}, f_{12}, f^{\flat}_{12})}", shift left=2, from=1-1, to=1-3]
	\arrow[shift right=2, from=1-1, to=1-3]
	\arrow["{(f_{13}, g_{13})}"', shift right=2, from=1-1, to=2-1]
	\arrow[shift left=2, from=1-1, to=2-1]
	\arrow[shift right=2, from=1-3, to=2-3]
	\arrow["{(f_{24}, g_{24})}", shift left=2, from=1-3, to=2-3]
	\arrow[shift left=2, from=2-1, to=2-3]
	\arrow["{(a_{34}, c_{34}, f_{34}, f^{\flat}_{34})}"', shift right=2, from=2-1, to=2-3]
\end{tikzcd}\]
\end{center}

are given by pairs of deterministic maps $f_{1234}: c_{12} \rightarrow c_{34}$, $g_{1234}: a_{12} \rightarrow a_{34}$ such that the following equations hold: 

\begin{enumerate}
	\item $(f_{12}; f_{1234} \otimes f_{24}) = f_{13}; f_{34}$
	\item $(f_{12}^{\flat}; f_{1234} \otimes f_{24} \otimes g_{1234} \otimes g_{24}) = (f_{13} \otimes g_{13}); f_{34}^{\flat}$

\end{enumerate}

In other words, the condition is commutativity of the following squares:

\begin{center}
\[\begin{tikzcd}
	{c_1a_1} && {c_{12}c_2 a_{12}a_2} && {c_1} & {c_{12}c_2} \\
	{c_3a_3} && {c_{34}c_4 a_{34} a_4} && {c_3} & {c_{34}c_4}
	\arrow["{f_{12}^{\flat}}", from=1-1, to=1-3]
	\arrow["{f_{13} \otimes g_{13}}"{description}, from=1-1, to=2-1]
	\arrow["{f_{1234}\otimes f_{24} \otimes g_{1234} \otimes g_{24}}"{description}, from=1-3, to=2-3]
	\arrow["{f_{12}}", from=1-5, to=1-6]
	\arrow["{f_{13}}"{description}, from=1-5, to=2-5]
	\arrow["{f_{1234} \otimes f_{24}}"{description}, from=1-6, to=2-6]
	\arrow["{f_{34}^{\flat}}"', from=2-1, to=2-3]
	\arrow["{f_{34}}"', from=2-5, to=2-6]
\end{tikzcd}\]
\end{center}

Composition of $xz$-squares in the $z$ direction corresponds to composition of maps in $\CC_{det}$ and concatenation of commutative squares. On the other hand, composition in the $x$ direction relies on taking tensor products of (deterministic) maps, and using properties of copy-composition; More precisely, consider $x$-composable $xz$-squares as below:

\begin{center}
\[\begin{tikzcd}
	\begin{array}{c} {\begin{pmatrix} a_1 \\ c_1  \\ \end{pmatrix}} \end{array} && \begin{array}{c} {\begin{pmatrix} a_2 \\ c_2  \\ \end{pmatrix}} \end{array} && \begin{array}{c} {\begin{pmatrix} a_2 \\ c_2  \\ \end{pmatrix}} \end{array} && \begin{array}{c} {\begin{pmatrix} a_3 \\ c_3  \\ \end{pmatrix}} \end{array} \\
	& s &&&& t \\
	\begin{array}{c} {\begin{pmatrix} a_4 \\ c_4  \\ \end{pmatrix}} \end{array} && \begin{array}{c} {\begin{pmatrix} a_5 \\ c_5  \\ \end{pmatrix}} \end{array} && \begin{array}{c} {\begin{pmatrix} a_5 \\ c_5  \\ \end{pmatrix}} \end{array} && \begin{array}{c} {\begin{pmatrix} a_6 \\ c_6  \\ \end{pmatrix}} \end{array}
	\arrow["{(a_{12}, c_{12}, f_{12}, f^{\flat}_{12})}", shift left=2, from=1-1, to=1-3]
	\arrow[shift right=2, from=1-1, to=1-3]
	\arrow["{(f_{14}, g_{14})}"', shift right=2, from=1-1, to=3-1]
	\arrow[shift left=2, from=1-1, to=3-1]
	\arrow[shift right=2, from=1-3, to=3-3]
	\arrow["{(f_{25}, g_{25})}", shift left=2, from=1-3, to=3-3]
	\arrow["{(a_{23}, c_{23}, f_{23}, f^{\flat}_{23})}", shift left=2, from=1-5, to=1-7]
	\arrow[shift right=2, from=1-5, to=1-7]
	\arrow["{(f_{25}, g_{25})}"', shift right=2, from=1-5, to=3-5]
	\arrow[shift left=2, from=1-5, to=3-5]
	\arrow[shift right=2, from=1-7, to=3-7]
	\arrow["{(f_{36}, g_{36})}", shift left=2, from=1-7, to=3-7]
	\arrow[shift left=2, from=3-1, to=3-3]
	\arrow["{(a_{45}, c_{45}, f_{45}, f^{\flat}_{45})}"', shift right=2, from=3-1, to=3-3]
	\arrow[shift left=2, from=3-5, to=3-7]
	\arrow["{(a_{56}, c_{56}, f_{56}, f^{\flat}_{56})}"', shift right=2, from=3-5, to=3-7]
\end{tikzcd}\]
\end{center}

Recall that $x$-morphisms are composed using copy-composition, so that the $x$-morphisms comprising the boundary of the composite $xz$ square $s;_x t$, are given by maps $f_{13}: c_1 \rightarrow c_{12} c_2 c_{23} c_3$, $f_{13}^{\flat}: c_1 a_1 \rightarrow c_{12} c_2 c_{23} c_3 a_{12} a_2 a_{23} a_3$, and $f_{46}: c_4 \rightarrow c_{45} c_5 c_{56} c_6$, $f_{46}^{\flat}: c_4 a_4 \rightarrow c_{45} c_5 c_{56} c_6 a_{45} a_5 a_{56} a_6$.

Then, the deterministic maps $f_{1346}: c_{12} c_2 c_{23} \rightarrow c_{45} c_5 c_{56}$, $g_{1346}: a_{12} a_2 a_{23} \rightarrow a_{45} a_5 a_{56}$ we are looking for can simply defined as $f_{1346} = f_{1245} \otimes f_{25} \otimes f_{2356}$ and $g_{1346} = g_{1245} \otimes g_{25} \otimes g_{2356}$.

Now, note that associativity of such $x$-composition follows from associativity of the tensor product. Also, $xz$-interchange follows from functoriality of the tensor product. Thus, the only thing left to check is that this $x$-composition is actually well-defined. As one might expect, this relies on compatibility of deterministic maps with copying. Let us show that one of the required squares commutes; we leave the other one to the reader:

\begin{center}
\[\begin{tikzcd}
	{c_1a_1} && {c_{12}c_2 a_{12}a_2} && {c_{12}c_2 a_{12}a_2 \, \, c_2 a_2} && {c_{12}c_2 a_{12}a_2 \, \, c_{23} c_3 a_{23} a_3} \\
	&&&&&& {c_{12}c_2 c_{23} c_3  a_{12}a_2 a_{23} a_3} \\
	{c_4a_4} && {c_{45}c_5 a_{45} a_5} && {c_{45}c_5 a_{45}a_5 \, \, c_5 a_5} && {c_{45}c_5 c_{56} c_6  a_{45}a_5 a_{56} a_6}
	\arrow["{f_{12}^{\flat}}", from=1-1, to=1-3]
	\arrow["{f_{14} \otimes g_{14}}"{description}, from=1-1, to=3-1]
	\arrow["{\sigma;copy_{c_{2}a_{2}}; \sigma}", from=1-3, to=1-5]
	\arrow["{f_{1245}\otimes f_{25} \otimes g_{1245} \otimes g_{25}}"{description}, from=1-3, to=3-3]
	\arrow["{c_{12}c_2 a_{12}a_2 \otimes f_{23}^{\flat}}", from=1-5, to=1-7]
	\arrow["{(f_{1245}\otimes f_{25} \otimes g_{1245} \otimes g_{25}) \otimes (f_{25} \otimes g_{25})}"{description}, from=1-5, to=3-5]
	\arrow["\sigma", from=1-7, to=2-7]
	\arrow["{f_{1245}\otimes f_{25}\otimes f_{2356} \otimes f_{36} \otimes g_{1245} \otimes g_{25} \otimes g_{2356} \otimes g_{36}}"{description}, from=2-7, to=3-7]
	\arrow["{f_{45}^{\flat}}"', from=3-1, to=3-3]
	\arrow["{\sigma;copy_{c_{5}a_{5}}; \sigma}"', from=3-3, to=3-5]
	\arrow["{c_{45}c_5 a_{45}a_5 \otimes f_{56}^{\flat}}"', from=3-5, to=3-7]
\end{tikzcd}\]
	
\end{center}

The leftmost and rightmost squares in the diagram above commute because $s$ and $t$ are $xz$-squares, respectively. The center square commutes thanks to naturality of $\sigma$ and determinism of the maps $f_{25}$ and $g_{25}$.
		
		\item There is exactly one cube, i.e. $xyz$-morphism, for each boundary. All properties of cubes are then automatically satisfied.

		\item The symmetric monoidal structure is given by the tensor product in $\CC$, which restricts into the cartesian product in $\CC_{det}$.

	\end{enumerate}

\end{construction}

%% file: appendix_arenasys_moore.tex
\subsection{The triple category $ArenaSys^{\mathrm{Moore}}_{\CC} \supseteq Arena^{\mathrm{Moore}}_{\CC}$}\label{appendix_ArenaSys_Moore}

Here, we give details for the proof of Theorem \ref{theo_ArenaSys_C}. Recall that $2$ denotes the $1$-category freely generated by the graph $0 \rightarrow 1$, and that $y(2)$ denotes the triple category built from that, concentrated in the $y$-direction.

\begin{construction}\label{constr_ArenaSys_C} We define the $x$-semi triple category $ArenaSys^{\mathrm{Moore}}_{\CC}$, along with the triple functor $F: ArenaSys^{\mathrm{Moore}}_{\CC} \rightarrow y(2)$ such that $F^{-1}(1) \simeq Arena^{\mathrm{Moore}}_{\CC}$, as follows:
	
	\begin{enumerate}
		\item The class of objects is $Ob(Arena^{\mathrm{Moore}}_{\CC}) \sqcup Mor(\CC_{det})$. In fact, we shall have $Ob(F^{-1}(1)) = Ob(Arena^{\mathrm{Moore}}_{\CC}) = {Ob(\CC)}^{2}$ and $Ob(F^{-1}(0)) = Mor(\CC_{det})$. Abusing notations, an element $r: \widetilde{S} \rightarrow S$ of $Mor(\CC_{det})$ may be denoted ${\begin{pmatrix} \widetilde{S} \\ S  \\ \end{pmatrix}}$.
		\item If $r_1: \widetilde{S}_1 \rightarrow S_1$, $r_2: \widetilde{S}_2 \rightarrow S_2$ are in $Mor(\CC_{det})$, an $x$-morphism ${\begin{pmatrix} \widetilde{S}_1 \\ S_1  \\ \end{pmatrix}} \rightrightarrows {\begin{pmatrix} \widetilde{S}_2 \\ S_2  \\ \end{pmatrix}}$ is a pair of morphisms $f^{\flat}: \widetilde{S}_1 \rightarrow \widetilde{S}_2$, $f: S_1 \rightarrow S_2$ in $\CC$, such that $f^{\flat} ; r_2 = r_1 ; f$. Composition is given by composition in $\CC$. Checking that it is well-defined amounts to concatenating commutative squares in $\CC$.
		
		Note that we do not use copy-composition here, contrary to the case of $Arena^{\mathrm{Moore}}_{\CC}$. 
		
		%

		\item A $z$-morphism ${\begin{pmatrix} \widetilde{S}_1 \\ S_1  \\ \end{pmatrix}} \rightrightarrows {\begin{pmatrix} \widetilde{S}_2 \\ S_2  \\ \end{pmatrix}}$ is a pair of \emph{deterministic} morphisms $g: \widetilde{S}_1 \rightarrow \widetilde{S}_2$, $f: S_1 \rightarrow S_2$ in $\CC$, making the following square commute: 
		
		\begin{center}
			\begin{tikzcd}
				\widetilde{S}_1 \arrow[r, "g^{}"] \arrow[d, "r_1"]& \widetilde{S}_2  \arrow[d, "r_2"]
				\\
				S_1 \arrow[r, "f"]& S_2
			\end{tikzcd}
		\end{center}
		
		\item If $r: \widetilde{S} \rightarrow S \in Mor(\CC_{det})$ and $	{\begin{pmatrix} a_1 \\ c_1  \\ \end{pmatrix}} \in Ob(Arena^{\mathrm{Moore}}_{\CC})$, a $y$-morphism ${\begin{pmatrix} \widetilde{S} \\ S  \\ \end{pmatrix}} \rightarrow {\begin{pmatrix} a_1 \\ c_1  \\ \end{pmatrix}}$ is a \emph{nondeterministic lens} $({\begin{pmatrix} \widetilde{S} \\ S  \\ \end{pmatrix}} \leftrightarrows {\begin{pmatrix} a_1 \\ c_1  \\ \end{pmatrix}})$ \emph{with deterministic output map}, given by an object $ \in \CC$, a \emph{deterministic} map $f: S \rightarrow c_1$ in $\CC_{det}$, and a map $f^{\sharp} :  S \otimes a_1 \rightarrow \widetilde{S}$ in $\CC$, such that the following triangle commutes:
		
		\begin{center}
			\begin{tikzcd}
				 S \otimes a_1 \arrow[r, "f^{\sharp}"] \arrow[rd, "\pi"]& \widetilde{S}  \arrow[d, "r"]
				\\
				& S
			\end{tikzcd}
		\end{center}

		The only $y$ morphisms between objects in $F^{-1}(0)$ are identities. Composition is given by composition of lenses, and composition in $\CC_{}$. Checking that it is well-defined amounts to basic computations in the Markov category $\CC_{}$.

		\item The $xz$-double category of the fiber $F^{-1}(0)$ is thin. Consider a boundary as below, where the vertical maps are $z$ morphisms, and the horizontal maps are $x$ morphisms: 
		\begin{center}
			\begin{tikzcd}
				{\begin{pmatrix} \widetilde{S}_1 \\ S_1  \\ \end{pmatrix}}  \arrow[d, shift left]\arrow[d, "{(f_{12}, g_{12})}", swap, shift right] \arrow[r, "{(f_{13}, f_{13}^{\flat})}", shift left]\arrow[r, shift right] & {\begin{pmatrix} \widetilde{S}_3 \\ S_3  \\ \end{pmatrix}}   \arrow[d, shift right]\arrow[d, "{(f_{34}, g_{34})}", shift left]
				\\
				{\begin{pmatrix} \widetilde{S}_2 \\ S_2  \\ \end{pmatrix}} \arrow[r, "{(f_{24}, f_{24}^{\flat})}", shift left] \arrow[r, shift right] & {\begin{pmatrix} \widetilde{S}_4 \\ S_4  \\ \end{pmatrix}}
			\end{tikzcd}
		\end{center} 
		
		Then, there is a square with this boundary if and only if the following squares commute in $\CC$:
		
		\[\begin{tikzcd}
			{\widetilde{S}_1} & {\widetilde{S}_3} && {S_1} & {S_3} \\
			{\widetilde{S}_2} & {\widetilde{S}_4} && {S_2} & {S_4}
			\arrow["{f^{\flat}_{13}}", from=1-1, to=1-2]
			\arrow["{g_{12}}"{description}, from=1-1, to=2-1]
			\arrow["{g_{34}}"{description}, from=1-2, to=2-2]
			\arrow["{f_{13}}", from=1-4, to=1-5]
			\arrow["{f_{12}}"{description}, from=1-4, to=2-4]
			\arrow["{f_{34}}"{description}, from=1-5, to=2-5]
			\arrow["{f^{\flat}_{24}}", from=2-1, to=2-2]
			\arrow["{f_{24}}", from=2-4, to=2-5]
		\end{tikzcd}\]
		
		Note that the vertical arrows are deterministic here, only the horizontal ones are nondeterministic. One can check that compositions are well-defined by concatenating, i.e. composing, commutative squares. Associativity and interchange are then automatic. Note that these new $xz$-squares only involve systems/morphisms in the fiber $F^{-1}(0)$, so there are no compositions with previously defined $xz$ squares in $Arena^{\mathrm{Moore}}_{\CC}$.
		\item\label{yz_squares_ArenaSys_C} The only $yz$ squares in $F^{-1}(0)$ are identities of $z$-morphisms. Now, for the nontrivial $yz$ squares, from $z$-morphisms in $F^{-1}(0)$ to $z$-morphisms in $F^{-1}(1)$, consider the boundary of a $yz$ square as below, where the vertical maps are $y$-morphisms, and the horizontal ones are $z$-morphisms:
		\begin{center}
		\begin{tikzcd}
			{\begin{pmatrix} \widetilde{S}_1 \\ S_1  \\ \end{pmatrix}}  \arrow[d, "{(f_{13}, f_{13}^{\sharp})}", swap, shift right] \arrow[r, "{(f_{12}, g_{12})}", shift left]\arrow[r, shift right] & {\begin{pmatrix} \widetilde{S}_2 \\ S_2  \\ \end{pmatrix}}   \arrow[d, shift right]
			\\
			{\begin{pmatrix} a_3 \\ c_3  \\ \end{pmatrix}} \arrow[u, shift right]\arrow[r, "{(f_{34}, g_{34})}", shift left] \arrow[r, shift right] & {\begin{pmatrix} a_4 \\ c_4  \\ \end{pmatrix}}\arrow[u, "{(f_{24}, f_{24}^{\sharp})}", shift right, swap]
		\end{tikzcd}
		\end{center}
		Then, a square with this boundary is given by commutation of the following square in $\CC_{det}$:
		
		\begin{center}
		\begin{tikzcd}
			{S_1} & {S_2} \\
			{c_3} & {c_4}
			\arrow["{f_{12}}", from=1-1, to=1-2]
			\arrow["{f_{13}}"', from=1-1, to=2-1]
			\arrow["{f_{24}}", from=1-2, to=2-2]
			\arrow["{f_{34}}", from=2-1, to=2-2]
		\end{tikzcd}
		\end{center}
		Composition of $yz$-squares in the direction $z$ is concatenating commuting squares. Associativity is automatic.
		
		Let us now define $y$-composition for the new $yz$-squares. Let $s_{}$, $t_{}$ be nontrivial $y$-composable squares as below:

		\begin{center}
		\begin{tikzcd}
			\begin{array}{c} {\begin{pmatrix} \widetilde{S}_1 \\ S_1  \\ \end{pmatrix}} \end{array} && \begin{array}{c} {\begin{pmatrix} \widetilde{S}_2 \\ S_2  \\ \end{pmatrix}} \end{array} \\
			& {s_{}} \\
			\begin{array}{c} {\begin{pmatrix} a_3 \\ c_3  \\ \end{pmatrix}} \end{array} && \begin{array}{c} {\begin{pmatrix} a_4 \\ c_4  \\ \end{pmatrix}} \end{array} \\
			& {t_{}} \\
			\begin{array}{c} {\begin{pmatrix} a_5 \\ c_5  \\ \end{pmatrix}} \end{array} && \begin{array}{c} {\begin{pmatrix} a_6 \\ c_6  \\ \end{pmatrix}} \end{array}
			\arrow["{(f_{12}, g_{12})}", shift left=2, from=1-1, to=1-3]
			\arrow[shift right=2, from=1-1, to=1-3]
			\arrow["{(f_{13}, f_{13}^{\sharp})}"', shift right=2, from=1-1, to=3-1]
			\arrow[shift right=2, from=1-3, to=3-3]
			\arrow[shift right=2, from=3-1, to=1-1]
			\arrow[shift left=2, from=3-1, to=3-3]
			\arrow["{(f_{34}, g_{34})}"', shift right=2, from=3-1, to=3-3]
			\arrow["{(f_{35}, f_{35}^{\sharp})}"', shift right=2, from=3-1, to=5-1]
			\arrow["{(f_{24}, f_{24}^{\sharp})}"', shift right=2, from=3-3, to=1-3]
			\arrow[shift right=2, from=3-3, to=5-3]
			\arrow[shift right=2, from=5-1, to=3-1]
			\arrow[shift left=2, from=5-1, to=5-3]
			\arrow["{(f_{56}, g_{56})}"', shift right=2, from=5-1, to=5-3]
			\arrow["{(f_{46}, f_{46}^{\sharp})}"', shift right=2, from=5-3, to=3-3]
		\end{tikzcd}
		\end{center}
		Note that the bottom $yz$-square belongs to $Arena^{\mathrm{Moore}}_{\CC}$; it is also given by a "commuting square condition". Concatenating the two commuting squares together, one gets that the required square (for the composite $\frac{s}{t}$) commutes. Associativity and interchange are automatic.

		\item The only $xy$ squares in $F^{-1}(0)$ are identities of $x$-morphisms. For the nontrivial ones, consider the boundary of an $xy$ square as in the diagram below:
		\begin{center}
			\begin{tikzcd}
				{\begin{pmatrix} \widetilde{S}_1 \\ S_1  \\ \end{pmatrix}}  \arrow[d, "{(f_{13}, f_{13}^{\sharp})}", swap, shift right] \arrow[rrr, "{(f_{12}, f^{\flat}_{12})}", shift left]\arrow[rrr, shift right] &&& {\begin{pmatrix} \widetilde{S}_2 \\ S_2  \\ \end{pmatrix}}   \arrow[d, shift right]
				\\
				{\begin{pmatrix} a_3 \\ c_3  \\ \end{pmatrix}} \arrow[u, shift right]\arrow[rrr, "{(a_{34}, c_{34}, f_{34}, f^{\flat}_{34})}", shift left] \arrow[rrr, shift right] &&& {\begin{pmatrix} a_4 \\ c_4  \\ \end{pmatrix}}\arrow[u, "{(f_{24}, f_{24}^{\sharp})}", shift right, swap]
			\end{tikzcd}
		\end{center}
		
		Then, an $xy$ square with this boundary is given by a morphism $s_{}:  S_1 a_3 \rightarrow  S_2 c_{34} a_{34} a_4$ in $\CC$, making the following squares commute:
\[\begin{tikzcd}
	{{S_1  a_3}} && {{S_2 c_{34} a_{34} a_4}} & {S_1} & {S_2} \\
	{c_3 a_3} && {c_{34} c_4 a_{34} a_4} & {c_3} & {c_4} \\
	& {{S_1 a_3}} && {{S_2 a_4}} \\
	& {\widetilde{S}_1} && {\widetilde{S}_2}
	\arrow["s", from=1-1, to=1-3]
	\arrow["{f_{13}\otimes a_3}"{description}, from=1-1, to=2-1]
	\arrow["{(\sigma; c_{34} \otimes f_{24}) \otimes a_{34} a_4}"{description}, from=1-3, to=2-3]
	\arrow["{f_{12}}", from=1-4, to=1-5]
	\arrow["{f_{13}}"{description}, from=1-4, to=2-4]
	\arrow["{f_{24}}"{description}, from=1-5, to=2-5]
	\arrow["{f_{34}^{\flat}}", from=2-1, to=2-3]
	\arrow["{f_{34}; del_{c_{34}}}", from=2-4, to=2-5]
	\arrow["{s; del_{c_{34}a_{34}}}", from=3-2, to=3-4]
	\arrow["{f_{13}^{\sharp}}"{description}, from=3-2, to=4-2]
	\arrow["{f_{24}^{\sharp}}"{description}, from=3-4, to=4-4]
	\arrow["{f_{12}^{\flat}}", from=4-2, to=4-4]
\end{tikzcd}\]

		\item Let us now define $x$-composition for these new $xy$-squares. Recall that $x$-composition for the bottom $x$-morphisms of such squares uses copy-composition, yet the top $x$-morphisms (those in the fiber $F^{-1}(0)$) compose in the usual way. Thus, our $x$-composition operation for $xy$-squares shall use a mix of usual and copy compositions. Consider a setting as below, with $x$-composable squares $s$ and $t$:
		
\[\begin{tikzcd}
	\begin{array}{c} {\begin{pmatrix} \widetilde{S}_1 \\ S_1  \\ \end{pmatrix}} \end{array} && \begin{array}{c} {\begin{pmatrix} \widetilde{S}_2 \\ S_2  \\ \end{pmatrix}} \end{array} && \begin{array}{c} {\begin{pmatrix} \widetilde{S}_3 \\ S_3  \\ \end{pmatrix}} \end{array} \\
	& s && t \\
	\begin{array}{c} {\begin{pmatrix} a_4 \\ c_4  \\ \end{pmatrix}} \end{array} && \begin{array}{c} {\begin{pmatrix} a_5 \\ c_5 \\ \end{pmatrix}} \end{array} && \begin{array}{c} {\begin{pmatrix} a_6 \\ c_6  \\ \end{pmatrix}} \end{array}
	\arrow["{(f_{12}, f^{\flat}_{12})}", shift left=2, from=1-1, to=1-3]
	\arrow[shift right=2, from=1-1, to=1-3]
	\arrow["{(f_{14}, f_{14}^{\sharp})}"', shift right=2, from=1-1, to=3-1]
	\arrow["{(f_{23}, f^{\flat}_{23})}", shift left=2, from=1-3, to=1-5]
	\arrow[shift right=2, from=1-3, to=1-5]
	\arrow["{(f_{25}, f_{25}^{\sharp})}"{description}, shift right=2, from=1-3, to=3-3]
	\arrow[shift right=2, from=1-5, to=3-5]
	\arrow[shift right=2, from=3-1, to=1-1]
	\arrow["{(a_{45}, c_{45}, f_{45}, f^{\flat}_{45})}", shift left=2, from=3-1, to=3-3]
	\arrow[shift right=2, from=3-1, to=3-3]
	\arrow[shift right=2, from=3-3, to=1-3]
	\arrow["{(a_{56}, c_{56}, f_{56}, f^{\flat}_{56})}", shift left=2, from=3-3, to=3-5]
	\arrow[shift right=2, from=3-3, to=3-5]
	\arrow["{( f_{36}, f_{36}^{\sharp})}"', shift right=2, from=3-5, to=1-5]
\end{tikzcd}\]
		
The map $s |t = s;_x t: S_1 a_4 \rightarrow S_3 \,\, c_{45}c_5c_{56} a_{45} a_5 a_{56} \,\, a_6$ is defined as the following composite:

$S_1 a_4 \xrightarrow{s} S_2 \, c_{45} a_{45}  \, a_5 \xrightarrow{copy_{S_2}\otimes c_{45} a_{45}   a_5} S_2 S_2 \, c_{45} a_{45}  \, a_5 \xrightarrow{\sigma; S_2 \otimes c_{45} a_{45} \otimes  f_{25} \otimes a_5} S_2 \, c_{45}  a_{45}  \, c_5 a_5 \xrightarrow{S_2  c_{45}  a_{45}  c_5 \otimes copy_{a_5}} S_2 \, c_{45}  a_{45}  \, c_5 a_5 a_5 \xrightarrow{\sigma} c_{45}c_5  a_{45} a_5 \,  S_2 a_5 \xrightarrow{c_{45}c_5  a_{45} a_5 \otimes t} c_{45}c_5  a_{45} a_5 \,  S_3 c_{56} a_{56} a_6 \xrightarrow{\sigma}S_3 \, c_{45}c_5 c_{56} a_{45} a_5 a_{56} \, a_6.$

Omitting the various "tensoring with identities" in the notation, this composite looks like this: 

$S_1 a_4 \xrightarrow{s} S_2 \, c_{45} a_{45}  \, a_5 \xrightarrow{copy_{S_2}} S_2 S_2 \, c_{45} a_{45}  \, a_5 \xrightarrow{\sigma; f_{25}} S_2 \, c_{45}  a_{45}  \, c_5 a_5 \xrightarrow{copy_{a_5}} S_2 \, c_{45}  a_{45}  \, c_5 a_5 a_5 \xrightarrow{\sigma} c_{45}c_5  a_{45} a_5 \,  S_2 a_5 \xrightarrow{t} c_{45}c_5  a_{45} a_5 \,  S_3 c_{56} a_{56} a_6 \xrightarrow{\sigma}S_3 \, c_{45}c_5 c_{56} a_{45} a_5 a_{56} \, a_6.$

Associativity of this composition holds for the same reasons as associativity of copy-composition in $\CC$: associativity of copying and functoriality of the tensor product. Recall that there are no $x$-identities, so we don't have identity $xy$-squares either.

It remains to show that this composition is well-defined, i.e. that the three squares in the definition of $xy$-morphisms commute. For the top-right and bottom ones, it suffices to concatenate the corresponding squares for $s$ and $t$, and recall the basic properties of $del$ and $copy$: copy-composition followed by deletion is the same as usual composition! For the top-left square, the proof is very similar to the one we gave in the case of $Arena^{\mathrm{Moore}}_{\CC}$, and relies on determinism of the map $f_{25}$. We refer the reader to the proof of Claim \ref{claim_x_comp_well_def} for inspiration.

Note that we do not need the map $f_{25}^{\sharp}$ to be deterministic here: the use of standard composition, rather than copy-composition, for the "top parts" of the squares, allows for nondeterministic lenses.

		\item\label{enum_xy_squares_ArenaSysC}As in $Arena^{\mathrm{Moore}}_{\CC}$, composition of $xy$ squares in the $y$ direction is more involved, and relies on conditional products. Let $s$, $t$ be $y$-composable squares as below: 

\[\begin{tikzcd}
	\begin{array}{c} {\begin{pmatrix} \widetilde{S}_1 \\ S_1  \\ \end{pmatrix}} \end{array} && \begin{array}{c} {\begin{pmatrix} \widetilde{S}_2 \\ S_2  \\ \end{pmatrix}} \end{array} \\
	& s \\
	\begin{array}{c} {\begin{pmatrix} a_3 \\ c_3  \\ \end{pmatrix}} \end{array} && \begin{array}{c} {\begin{pmatrix} a_4 \\ c_4  \\ \end{pmatrix}} \end{array} \\
	\begin{array}{c} {\begin{pmatrix} a_3 \\ c_3  \\ \end{pmatrix}} \end{array} && \begin{array}{c} {\begin{pmatrix} a_4 \\ c_4  \\ \end{pmatrix}} \end{array} \\
	& t \\
	\begin{array}{c} {\begin{pmatrix} a_5 \\ c_5  \\ \end{pmatrix}} \end{array} && \begin{array}{c} {\begin{pmatrix} a_6 \\ c_6  \\ \end{pmatrix}} \end{array}
	\arrow["{(f_{12}, f^{\flat}_{12})}", shift left=2, from=1-1, to=1-3]
	\arrow[shift right=2, from=1-1, to=1-3]
	\arrow["{(f_{13}, f_{13}^{\sharp})}"', shift right=2, from=1-1, to=3-1]
	\arrow[shift right=2, from=1-3, to=3-3]
	\arrow[shift right=2, from=3-1, to=1-1]
	\arrow["{(a_{34}, c_{34}, f_{34}, f^{\flat}_{34})}", shift left=2, from=3-1, to=3-3]
	\arrow[shift right=2, from=3-1, to=3-3]
	\arrow["{(f_{24}, f_{24}^{\sharp})}"', shift right=2, from=3-3, to=1-3]
	\arrow["{(a_{34}, c_{34}, f_{34}, f^{\flat}_{34})}", shift left=2, from=4-1, to=4-3]
	\arrow[shift right=2, from=4-1, to=4-3]
	\arrow["{(f_{35}, f_{35}^{\sharp})}"', shift right=2, from=4-1, to=6-1]
	\arrow[shift right=2, from=4-3, to=6-3]
	\arrow[shift right=2, from=6-1, to=4-1]
	\arrow[shift left=2, from=6-1, to=6-3]
	\arrow["{(a_{56}, c_{56}, f_{56}, f_{56}^{\flat})}"', shift right=2, from=6-1, to=6-3]
	\arrow["{(f_{46}, f_{46}^{\sharp})}"', shift right=2, from=6-3, to=4-3]
\end{tikzcd}\]

		\begin{claim}
			The following diagram commutes:
\[\begin{tikzcd}
	{S_1 a_5} & { c_3 a_5} \\
	{S_1  c_3  a_5} & {c_{34} c_4 a_{56} a_6} \\
	{S_1  a_3} & {c_{34} c_4  \, c_{34}c_4 \,  a_{56}a_6 \, a_{56} a_6} \\
	{S_2 c_{34} a_{34} a_4} & {c_{34} c_4 a_{34}  a_4 \otimes c_{56}a_{56} a_6} \\
	{ S_2 \otimes c_{34}  c_4 a_{34} a_4} & {c_{34} c_4 a_{34}  a_4}
	\arrow["{f_{13} \otimes a_5}", from=1-1, to=1-2]
	\arrow["{copy_{S_1} ; S_1 \otimes f_{13} \otimes a_5}"', from=1-1, to=2-1]
	\arrow["t", from=1-2, to=2-2]
	\arrow["{S_1 \otimes f_{35}^{\sharp}}"', from=2-1, to=3-1]
	\arrow["{copy_{c_{34}c_4} \otimes copy_{a_{56}a_6}}", from=2-2, to=3-2]
	\arrow["{s_{}}"', from=3-1, to=4-1]
	\arrow["{copy_{c_{34}}; f_{3456};\sigma; f_{46}^{\sharp}\otimes f_{3456}^{\sharp}}", from=3-2, to=4-2]
	\arrow["{copy_{S_2}; \sigma; S_2  c_{34} \otimes f_{24} \otimes a_{34}a_4}"', from=4-1, to=5-1]
	\arrow["\pi", from=4-2, to=5-2]
	\arrow["\pi", from=5-1, to=5-2]
\end{tikzcd}\]
			
		\end{claim}
		\begin{proof}[Proof notes.]
			As in the case of $Arena^{\mathrm{Moore}}_{\CC}$, this relies on the fact that $s$ and $t$ are squares, with the bottom horizontal morphism in $s$ being the same as the top horizontal morphism in $t$.
		\end{proof}
		
		Then, let $\alpha: {  S_1 \otimes a_5} \rightarrow {  S_2 \otimes c_{34} c_4  a_{34}a_4 \otimes c_{56} a_{56} a_6}$ denote the conditional product over $c_{34} c_4 a_{34} a_4$ of the composites above, and let $\frac{s}{t}$ be the composite $(\alpha ; del_{c_{34} c_4 a_{34} a_4}) : {  S_1 \otimes a_5} \rightarrow {  S_2 \otimes c_{56} a_{56} a_6}$.

		\begin{claim}\label{claim_regen_for_y_comp_sys}
			The map $\alpha: S_1 a_5 \rightarrow  S_2 \, c_{34}c_4 a_{34}a_4 \,  c_{56} a_{56} a_6$ is equal to the composite $S_1 a_5 \xrightarrow{\alpha}  S_2 \, c_{34}c_4 a_{34}a_4 \,  c_{56} a_{56} a_6 \xrightarrow{del_{c_4 a_4}} S_2 \, c_{34} a_{34} \,  c_{56} a_{56} a_6 \xrightarrow{copy_{S_2}; copy_{S_2 a_6}; \sigma}$

			\noindent $S_2 \, c_{34} a_{34} \, S_2 \,  c_{56} a_{56} a_6 \, S_2 a_6 \xrightarrow{S_2 c_{34} a_{34} \otimes f_{24} \otimes c_{56} a_{56} a_6 \otimes (f_{24} \otimes a_6; f_{46}^{\sharp})} S_2 \, c_{34} a_{34} \, c_4 \,  c_{56} a_{56} a_6 \, a_4 \xrightarrow{\sigma} S_2 \, c_{34} c_4 a_{34} a_4\,  c_{56} a_{56} a_6$.
			
			Abusing notations by omitting tensors with identities, this composite is:

			$S_1 a_5 \xrightarrow{\alpha}  S_2 \, c_{34}c_4 a_{34}a_4 \,  c_{56} a_{56} a_6 \xrightarrow{del_{c_4 a_4}} S_2 \, c_{34} a_{34} \,  c_{56} a_{56} a_6 \xrightarrow{copy_{S_2}; copy_{S_2 a_6}; \sigma}$

			\noindent $S_2 \, c_{34} a_{34} \, S_2 \,  c_{56} a_{56} a_6 \, S_2 a_6 \xrightarrow{ f_{24} \otimes (f_{24} ; f_{46}^{\sharp})} S_2 \, c_{34} a_{34} \, c_4 \,  c_{56} a_{56} a_6 \, a_4 \xrightarrow{\sigma} S_2 \, c_{34} c_4 a_{34} a_4\,  c_{56} a_{56} a_6$.
		\end{claim}

		Note that this Claim is weaker than its counterpart Claim \ref{claim_regen_for_y_comp} that held in $Arena^{\mathrm{Moore}}_{\CC}$: we do not expect to recover $\alpha$ from its marginal $\frac{s}{t}$, because there is no way of getting an output in $c_{34}$ this time. However, this Claim will suffice for the proof of interchange; actually, one can check that the corresponding weaker statement would have been enough for the proof of $xy$-interchange in $Arena^{\mathrm{Moore}}_{\CC}$ as well.

		\begin{proof}
			The proof is very similar to that of Claim \ref{claim_regen_for_y_comp}: it relies on using Lemma \ref{lemma_CP_and_regen} twice, with $X= S_1 a_5$, $A= S_2$, $B = c_{34} c_4 a_{34} a_4$ and $C = c_{56} a_{56} a_6$. We leave the details to the reader.
		\end{proof}

		\begin{claim}\label{claim_y_comp_well_def_arenasys}
			This composition is well-defined, i.e. the required squares commute.
		\end{claim}
		\begin{proof}[Proof notes.]
		As in the case of $Arena^{\mathrm{Moore}}_{\CC}$, proving that this is well-defined relies on the fact that the marginals of conditional products are known, and on Lemma \ref{lemma_CP_and_regen}, which can be applied to the deterministic maps $f_{3456}^{\sharp}$, $f_{46}^{\sharp}$, and $f_{24}$.
		\end{proof}

		 Associativity uses determinism of the lenses that are in $Arena^{\mathrm{Moore}}_{\CC}$, i.e. in $F^{-1}(1)$,  determinism of the output maps of the lenses in $F^{-1}(0 \rightarrow 1)$, and associativity of conditional products. We refer the reader to the proof of Claim \ref{claim_y_comp_assoc_arena_moore}.

		\begin{prop}\label{prop_interchange_nondet_systems}
			The $xy$-interchange law holds for $xy$-squares in $ArenaSys^{\mathrm{Moore}}_{\CC}$.
		\end{prop}
		
		\begin{proof}[Proof notes.]
			By construction, the only new case we have to check is the following: 

\[\begin{tikzcd}
	\begin{array}{c} {\begin{pmatrix} \widetilde{S}_1 \\ S_1  \\ \end{pmatrix}} \end{array} && \begin{array}{c} {\begin{pmatrix} \widetilde{S}_2 \\ S_2  \\ \end{pmatrix}} \end{array} && \begin{array}{c} {\begin{pmatrix} \widetilde{S}_3 \\ S_3  \\ \end{pmatrix}} \end{array} \\
	& s && u \\
	\begin{array}{c} {\begin{pmatrix} a_4 \\ c_4  \\ \end{pmatrix}} \end{array} && \begin{array}{c} {\begin{pmatrix} a_5 \\ c_5 \\ \end{pmatrix}} \end{array} && \begin{array}{c} {\begin{pmatrix} a_6 \\ c_6  \\ \end{pmatrix}} \end{array} \\
	& t && v \\
	\begin{array}{c} {\begin{pmatrix} a_7 \\ c_7  \\ \end{pmatrix}} \end{array} && \begin{array}{c} {\begin{pmatrix} a_8 \\ c_8  \\ \end{pmatrix}} \end{array} && \begin{array}{c} {\begin{pmatrix} a_9 \\ c_9  \\ \end{pmatrix}} \end{array}
	\arrow["{(f_{12}, f^{\flat}_{12})}", shift left=2, from=1-1, to=1-3]
	\arrow[shift right=2, from=1-1, to=1-3]
	\arrow["{(f_{14}, f_{14}^{\sharp})}"', shift right=2, from=1-1, to=3-1]
	\arrow["{(f_{23}, f^{\flat}_{23})}", shift left=2, from=1-3, to=1-5]
	\arrow[shift right=2, from=1-3, to=1-5]
	\arrow["{(f_{25}, f_{25}^{\sharp})}"{description}, shift right=2, from=1-3, to=3-3]
	\arrow[shift right=2, from=1-5, to=3-5]
	\arrow[shift right=2, from=3-1, to=1-1]
	\arrow["{(a_{45}, c_{45}, f_{45}, f^{\flat}_{45})}", shift left=2, from=3-1, to=3-3]
	\arrow[shift right=2, from=3-1, to=3-3]
	\arrow["{(f_{47}, f_{47}^{\sharp})}"', shift right=2, from=3-1, to=5-1]
	\arrow[shift right=2, from=3-3, to=1-3]
	\arrow["{(a_{56}, c_{56}, f_{56}, f^{\flat}_{56})}", shift left=2, from=3-3, to=3-5]
	\arrow[shift right=2, from=3-3, to=3-5]
	\arrow["{(f_{58}, f_{58}^{\sharp})}"{description}, shift right=2, from=3-3, to=5-3]
	\arrow["{( f_{36}, f_{36}^{\sharp})}"', shift right=2, from=3-5, to=1-5]
	\arrow[shift right=2, from=3-5, to=5-5]
	\arrow[shift right=2, from=5-1, to=3-1]
	\arrow["{(a_{78}, c_{78}, f_{78}, f^{\flat}_{78})}", shift left=2, from=5-1, to=5-3]
	\arrow[shift right=2, from=5-1, to=5-3]
	\arrow[shift right=2, from=5-3, to=3-3]
	\arrow["{(a_{89}, c_{89}, f_{89}, f^{\flat}_{89})}", shift left=2, from=5-3, to=5-5]
	\arrow[shift right=2, from=5-3, to=5-5]
	\arrow["{(f_{69}, f_{69}^{\sharp})}"', shift right=2, from=5-5, to=3-5]
\end{tikzcd}\]

	The proof is very similar to that of Proposition \ref{prop_interchange_lenses}; our computations there did not need the fact that the $y$-morphisms in the top squares were deterministic, only that the output maps there were. Thus, applying the same proof ideas, invoking Claim \ref{claim_regen_for_y_comp_sys}, yields the result.
		\end{proof}

		\item There is exactly one cube, i.e. $xyz$-morphism, for each boundary. All properties of cubes are then automatically satisfied.

		\item The symmetric monoidal structure on the $x$-semi double category $F^{-1}(0 \rightarrow 1)$ is given by the tensor product in $\CC$, which restricts into the cartesian product in $\CC_{det}$; checking that this structure makes the codomain double functor $\mathrm{cod}: \mathbb{D} \rightarrow ArenaSys^{\mathrm{Moore}}_{\CC, xz}$ symmetric monoidal is straightforward, and left to the reader.
		
	\end{enumerate}
	\end{construction}

%% file: appendix_examples.tex
\section{Examples of discrete-time systems and behaviors}

Let us describe in more detail the kinds of discrete-time systems and behaviors mentioned in Subsection \ref{motiv_trajectories}. So, we work in the Markov category with conditionals BorelStoch, and the graph $\GG$ is the graph on natural numbers with edges $n \rightarrow n+1$.
Here, we consider systems that are given by a Markov kernel between standard Borel spaces: $update: S \times I \rightarrow S$, along with
a measurable map $expose: S \rightarrow O$. Then, for each integer $n \geq 0$, we define $update^n: S^{n+1} \times I^{n+1} \rightarrow S^{n+2}$ and $expose^n: S^{n+1} \rightarrow O^{n+1}$ via $update^n(s_0, \ldots, s_{n}, i_0, \ldots, i_n) = (s_0, \ldots, s_{n}, update(s_{n}, i_n)) \in S^{n+2}$ and $expose^n(s_0, \ldots, s_n) = (expose(s_0), \ldots, expose(s_n))$. Recall that, in this formalism, we have $S(n) = S^{n+1}$, $O(n) = O^{n+1}$, and $I(n+1) = I^{n+1}$

The restriction maps are given by the projections that forget the last coordinate. Unfolding the definitions (Point \ref{yz_squares_ArenaSys_C} in Construction \ref{constr_ArenaSys_C}, and Construction \ref{constr_Arena_C^G}), one can check that the condition for $y$-morphisms in the double category $ArenaSys^{\mathrm{Moore}}_{\CC}(\GG)$, i.e. $y$-natural transformations between $x$-semi triple functors in $[z(Ar(\GG)); ArenaSys^{\mathrm{Moore}}_{\CC}]$, corresponds to compatibility of the maps $expose^n$ with the restriction maps of $O$ and $S$; it holds by construction. Thus, each pair $(expose, update)$ as above defines a system in the sense of $T^{\mathrm{Moore}}(\CC, \GG)$.

%% file: appendix_continuous_time.tex
\section{Continuous-time systems}\label{appendix_continuous_time}

One can represent continuous-time systems in our framework by using suitable graphs $\GG$. For instance, let $\GG = (\mathbb{R}_{\geq 0}, <)$ be the graph of nonnegative real numbers, with the strict order relation, and let us keep working with the Markov category with conditionals $\CC= BorelStoch$. We now describe a class of systems in the theory $T^{\mathrm{Moore}}(\CC, \GG)$.

For any Euclidean space $E$ and any nonnegative real $r$, let $C^0(r, E)$ denote the separable Banach space of continuous maps from $[0, r]$ to $E$, with the norm $||f||_{C^0} = \sup_{t} ||f(t)||$. Similarly, let $C^1(r, E)$ denote the separable Banach space of functions $f: [0, r] \rightarrow E$ that are continuous on $[0, r]$, continuously differentiable on $(0, r)$, and whose derivatives have finite limits at $0$ and $r$, with norm $||f||_{C^1} = \sup_{t} ||f(t)|| + \sup_{t} ||f^{'}(t)||$.

These separable Banach spaces yield standard Borel spaces, if one only keeps the measurable structure. In fact, each Euclidean space $E$ defines functors $E_0, E_1: \GG^{op} \rightarrow BorelStoch$, via $r \mapsto C^0(r, E)$, resp. $r \mapsto C^1(r, E)$, with restriction morphisms given by restriction of functions. Note that there is a natural transformation $E_1 \Rightarrow E_0$, given by the continuous embeddings of Banach spaces $C^1(r, E) \hookrightarrow C^0(r, E)$.

As explained in Notation \ref{notation_graphs}, one can pull back these functors along the functors $dom, cod : Ar(\GG^{op}) \rightarrow \GG^{op}$, to define objects in the double category $ArenaSys^{\mathrm{Moore}}_{\CC}(\GG)$.
Thus, given Euclidean spaces $S$, $O$, $I$, one can consider systems given by families of maps $update(r \rightarrow s):  C^1(r, S) \times C^0(s, I) \rightarrow C^1(s, S)$ and $expose(r): C^1(r, S) \rightarrow C^0(r, O)$, for all $0 \leq r < s$.
In this case, we have ${\begin{pmatrix} I \\ O  \\ \end{pmatrix}}(r \rightarrow s) = {\begin{pmatrix} C^0(s, I) \\ C^0(r, O)  \\ \end{pmatrix}}$ for all $0 \leq r < s$. For instance, this could represent a system driven by an Ordinary Differential Equation with continuous parameters\footnote{Adding an error/exception value to the state spaces, one can deal with the cases where the input is not compatible with the existing state trajectory. Developing a dependent version of the theory would yield a cleaner definition}.

\begin{remark}
	\begin{enumerate}
		\item This approach is not restricted to Euclidean spaces; one can generalize it to smooth manifolds.
		\item One can go beyond spaces of continuous, or continuously differentiable, maps. For instance, to allow for jumps, spaces of piecewise smooth maps can be relevant.
	\end{enumerate}
\end{remark}


%% file: appendix_conditional_independence.tex
\section{Conditional independence}\label{appendix_cond_indep}

Here, we wish to demonstrate the relevance of our key underlying assumption of conditional independence  (see point \ref{discussion_cond_indep} in Section \ref{section_discussion}), by providing sufficient conditions for it to hold. Again, we work with $\CC=BorelStoch$ and the graph $\GG$ of natural numbers.

Recall (Notation \ref{notation_indep}) that, given standard Borel spaces $X$, $Y$, $Z$ and a probability distribution $p$ on $X \times Y \times Z$, viewed as a Markov map $p: * \rightarrow X \otimes Y \otimes Z$, we let $x$, $y$, $z$ denote the generalized elements $* \rightarrow X$, $* \rightarrow Y$ and $* \rightarrow Z$, and write $x \bot_y^p z$, or $x \bot_y z$, if the distribution $p: * \rightarrow X \otimes Y \otimes Z$ displays conditional independence of $x$ and $z$ over $y$. We also extend the notation to the contexts with deterministic maps $f: X \otimes Y \otimes Z \rightarrow U$, $g: X \otimes Y \otimes Z \rightarrow v$, and $h: X \otimes Y \otimes Z \rightarrow W$; we may then write $u \bot^{f, g, h, p}_v w$, or $u \bot^{p}_v w$, or even $u \bot_v w$.

Let us now recall the formal properties of this independence relation, sometimes referred to as the semigraphoid properties. See \cite[Lemma 12.5]{FRITZ-MarkovCats} for the proofs.

\begin{prop}
	Let $\CC$ be a Markov category with conditionals. Let $p: I \rightarrow X \otimes Y \otimes Z$, $q: I \rightarrow X_1 \otimes X_2 \otimes Y \otimes Z$ and $r: I \rightarrow X \otimes Y \otimes T \otimes Z$ be morphisms in $\CC$.
	
	\begin{enumerate}
		\item (Symmetry) If $x \bot^p_{y} z$, then $z \bot^{\tau \circ p}_y x$, where $\tau: X \otimes Y \otimes Z \rightarrow Z \otimes Y \otimes X$ is the deterministic map given by the symmetric monoidal structure.
		\item (Decomposition) If $x_1, x_2 \bot^q_y z$, then $x_1 \bot^q_y z$.
		\item (Contraction) If $x \bot_{y, t}^r z$ and $x \bot_{y}^p t$, then $x \bot_{y}^p z, t$. 
		\item (Weak union) If $x_1, x_2 \bot^q_y z$, then $x_1 \bot^q_{y, x_2} z$.
	\end{enumerate}
\end{prop}

In fact, the first three properties hold even if $\CC$ does not have conditionals.

Consider a system $S= {\begin{pmatrix} S \\ S  \\ \end{pmatrix}} \leftrightarrows {\begin{pmatrix} I_1 \\ O_1  \\ \end{pmatrix}}$ in the general sense of the theory $T^{\mathrm{Moore}}(\CC, \GG)$ and a lens $L= {\begin{pmatrix} I_1 \\ O_1  \\ \end{pmatrix}} \leftrightarrows {\begin{pmatrix} I_2 \\ O_2  \\ \end{pmatrix}}$, so that the composite yields a system $S^{'} = {\begin{pmatrix} S \\ S  \\ \end{pmatrix}} \leftrightarrows {\begin{pmatrix} I_2 \\ O_2  \\ \end{pmatrix}}$.

Now, consider a trajectory of the composite system $S^{'}$, i.e. a square $t^{\prime}$ in $ArenaSys^{\mathrm{Moore}}_{\CC}(\GG)$ as below:

\[\begin{tikzcd}
	\begin{array}{c} {\begin{pmatrix} * \\ *  \\ \end{pmatrix}} \end{array} && \begin{array}{c} {\begin{pmatrix} S \\ S  \\ \end{pmatrix}} \end{array} \\
	\begin{array}{c} {\begin{pmatrix} * \\ *  \\ \end{pmatrix}} \end{array} && \begin{array}{c} {\begin{pmatrix} I_1 \\ O_1  \\ \end{pmatrix}} \end{array} \\
	\begin{array}{c} {\begin{pmatrix} * \\ *  \\ \end{pmatrix}} \end{array} && \begin{array}{c} {\begin{pmatrix} I_2 \\ O_2  \\ \end{pmatrix}} \end{array}
	\arrow[shift right, from=1-1, to=1-3]
	\arrow[shift left, from=1-1, to=1-3]
	\arrow[shift right, from=1-1, to=2-1]
	\arrow[shift right, from=1-3, to=2-3]
	\arrow[shift right, from=2-1, to=1-1]
	\arrow[shift right, no head, from=2-1, to=3-1]
	\arrow[shift right, from=2-3, to=1-3]
	\arrow[shift right, from=2-3, to=3-3]
	\arrow[shift right, no head, from=3-1, to=2-1]
	\arrow[shift right, from=3-1, to=3-3]
	\arrow[shift left, from=3-1, to=3-3]
	\arrow[shift right, from=3-3, to=2-3]
\end{tikzcd}\]

\begin{remark}
	In this context, we have, for all $n \in \mathbb{N}$, a Markov map/probability distribution $t^{\prime, n}: * \rightarrow S(n) \otimes I_2(n+1)$. Then, using the expose maps of the system $S$ and the lens $L$, the update map of the lens $L$, and $copy$ maps, all of which are \emph{deterministic}, such a $t^{\prime, n}$ induces a probability distribution $\mu^n: * \rightarrow S(n) \otimes I_2(n+1) \otimes I_1(n+1) \otimes O_1(n) \otimes O_2(n)$.
\end{remark}

\begin{claim}
	The collection of Markov maps $(\mu^n)_n$ yields squares $t$ and $t_{12}$, by marginalizing: the maps $t^n: * \rightarrow S(n) \otimes I_1(n+1)$ and $t_{12}^n: * \rightarrow O_1(n) \otimes I_2(n+1)$ defined as marginals of $\mu^n$, for all $n$, define squares. 
\end{claim}
\begin{proof}[Proof notes.]
	These computations rely on Lemma \ref{lemma_CP_and_regen}.
\end{proof}

Now, we would like to know whether the equation $t^{\prime} = \frac{t}{t_{12}}$ holds.

\begin{claim}
	Let $n \in \mathbb{N}_{\geq 1}$. Let $q$ be the Markov map defined as $q:= (t^{\prime,n} ;  res_S(0 \rightarrow 1 \rightarrow \cdots \rightarrow n) \otimes I_2(n+1))$, so that    $q: * \rightarrow S(0) \otimes I_2(n+1)$. Assume the following
	
	\begin{enumerate}
		\item The Markov map $t^{\prime, n}: * \rightarrow S(n) \otimes I_2(n+1)$ is generated from the map $q$ using the update map of the system $S^{'}$, and the (deterministic) restriction maps given by the functor $I_2$. 
		\item We have $i_2(n+1) \bot^{\mu^n}_{o_1(n)} s(0)$.
	\end{enumerate}

	Then, we have the conditional independence property $i_2(n+1)  \bot_{i_1(n+1) o_1(n)} s(n)$. In other words, the map $t^{\prime, n}$ is the $n^{th}$ component of the $y$-composite $\frac{t}{t_{12}}$.  
	
\end{claim}

The first hypothesis of the Claim means the following: we define a family of Markov maps $q_k: * \rightarrow S(k) \otimes I_2(n+1)$, for $0 \leq k \leq n$, by induction.
We put $q_0 =q$, and $q_{k+1}$ is the composite $* \xrightarrow{q_k} S(k) I_2(n+1) \xrightarrow{S(k) \otimes (copy_{I_2(n+1)}; res \otimes I_2(n+1))} S(k) I_2(k+1) I_2 (n+1) \xrightarrow{up_{S^{'}} \otimes I_2(n+1)} S(k+1) I_2(n+1)$. Then, the assumption is that $t^{\prime, n}: * \rightarrow S(n) \otimes I_2(n+1)$ is equal to $q_{n}$.

\begin{proof}
	We prove by induction on $0 \leq k \leq n$ that $i_2(n+1) \bot^{\mu^n}_{o_1(n) i_1(n+1)} s(k) $. For $k=0$, this follows from the fact that, under the distribution $\mu^n$, we know that $i_1(n+1) \in I_1(n+1)$ is the image of $(o_1(n), i_2(n+1))$ under a \emph{deterministic} map, namely the $n^{th}$ component of the update map of the lens $L$, almost surely. More explicitly, this observation, along with the second hypothesis, implies the independence $i_2(n+1) i_1(n+1) \bot^{\mu^n}_{o_1(n)} s(0) $, which then implies, using the semigraphoid properties of conditional independence, $i_2(n+1)  \bot^{\mu^n}_{o_1(n) i_1(n+1)} s(0)$.

	To conclude, let us prove the property for $k+1$ assuming it for $k$, given $0 \leq k < n$. By induction hypothesis, we have $i_2(n+1)  \bot^{\mu^n}_{o_1(n) i_1(n+1)} s(k)$. By assumption, we know that under the distribution $\mu^n$, the marginal $s(k+1) \in S(k+1)$ is equal to $update_S^k(s(k), i_1(k+1))$. By compatibility of conditional independence with post-composition, this implies that $i_2(n+1)  \bot^{\mu^n}_{o_1(n) i_1(n+1)} s(k+1)$, as required.
\end{proof}

%% file: appendix_behaviors_of_tensor_products.tex
\section{Behaviors of tensor products of systems}

In this section, we discuss behaviors for tensor products of systems. More specifically, we define operations that yield behaviors of tensor products, given behaviors of the components and extra compatibility data. This is a very first step towards answering the questions described in Section \ref{section_outlook}, \ref{outlook_trajectories_tensors}. We fix arbitrary $\CC$ and $\GG$.

Let us begin with the following observation: in the triple category $ArenaSys^{\mathrm{Moore}}_{\CC}$, there are canonical morphisms from tensor products, when these are defined, to their factors. For instance, if $(f_i, f_i^{\sharp}): {\begin{pmatrix} S_i \\ S_i  \\ \end{pmatrix}} \leftrightarrows {\begin{pmatrix} I_i \\ O_i  \\ \end{pmatrix}}$
are $y$-morphisms in the fiber over $0 \rightarrow 1$, for $i=1,2$, then there exist canonical $xy$-squares $\pi_1$, $\pi_2$ of the following shapes:

\[\begin{tikzcd}
	\begin{array}{c} {\begin{pmatrix} \widetilde{S}_1 \otimes \widetilde{S}_2 \\ S_1 \otimes S_2  \\ \end{pmatrix}} \end{array} && \begin{array}{c} {\begin{pmatrix} \widetilde{S}_1 \\ S_1  \\ \end{pmatrix}} \end{array} \\
	& {\pi_{1}} \\
	\begin{array}{c} {\begin{pmatrix} I_1 \otimes I_2 \\ O_1 \otimes O_2  \\ \end{pmatrix}} \end{array} && \begin{array}{c} {\begin{pmatrix} I_1 \\ O_1  \\ \end{pmatrix}} \end{array} \\
	\begin{array}{c} {\begin{pmatrix} \widetilde{S}_1 \otimes \widetilde{S}_2 \\ S_1 \otimes S_2  \\ \end{pmatrix}} \end{array} && \begin{array}{c} {\begin{pmatrix} \widetilde{S}_2 \\ S_2  \\ \end{pmatrix}} \end{array} \\
	& {\pi_2} \\
	\begin{array}{c} {\begin{pmatrix} I_1 \otimes I_2 \\ O_1 \otimes O_2  \\ \end{pmatrix}} \end{array} && \begin{array}{c} {\begin{pmatrix} I_2 \\ O_2  \\ \end{pmatrix}} \end{array}
	\arrow["{(\pi_1, \pi_1)}", shift left=2, from=1-1, to=1-3]
	\arrow[shift right=2, from=1-1, to=1-3]
	\arrow["{(f_1 \otimes f_2, f_{1}^{\sharp} \otimes f_{2}^{\sharp})}"', shift right=2, from=1-1, to=3-1]
	\arrow[shift right=2, from=1-3, to=3-3]
	\arrow[shift right=2, from=3-1, to=1-1]
	\arrow[shift left=2, from=3-1, to=3-3]
	\arrow["{(\pi_{1}, \pi_1)}"', shift right=2, from=3-1, to=3-3]
	\arrow["{(f_{1}, f_{1}^{\sharp})}"', shift right=2, from=3-3, to=1-3]
	\arrow["{(\pi_2, \pi_2)}", shift left=2, from=4-1, to=4-3]
	\arrow[shift right=2, from=4-1, to=4-3]
	\arrow["{(f_1 \otimes f_2, f_{1}^{\sharp} \otimes f_{2}^{\sharp})}"', shift right=2, from=4-1, to=6-1]
	\arrow[shift right=2, from=4-3, to=6-3]
	\arrow[shift right=2, from=6-1, to=4-1]
	\arrow[shift left=2, from=6-1, to=6-3]
	\arrow["{(\pi_2, \pi_2)}"', shift right=2, from=6-1, to=6-3]
	\arrow["{(f_{2}, f_{2}^{\sharp})}"', shift right=2, from=6-3, to=4-3]
\end{tikzcd}\]

Note that these squares only contain deterministic maps.

\begin{definition}
	We work in the triple category $ArenaSys^{\mathrm{Moore}}_{\CC}$.
	
	\begin{itemize}
		\item Let $s_1$ and $s_2$ be $xy$-squares as below:

\[\begin{tikzcd}
	\begin{array}{c} {\begin{pmatrix} \widetilde{T} \\ T  \\ \end{pmatrix}} \end{array} && \begin{array}{c} {\begin{pmatrix} \widetilde{S}_1 \\ S_1  \\ \end{pmatrix}} \end{array} \\
	& {s_{1}} \\
	\begin{array}{c} {\begin{pmatrix} I_0 \\ O_0  \\ \end{pmatrix}} \end{array} && \begin{array}{c} {\begin{pmatrix} I_1 \\ O_1  \\ \end{pmatrix}} \end{array} \\
	\begin{array}{c} {\begin{pmatrix} \widetilde{T} \\ T  \\ \end{pmatrix}} \end{array} && \begin{array}{c} {\begin{pmatrix} \widetilde{S}_2 \\ S_2  \\ \end{pmatrix}} \end{array} \\
	& {s_2} \\
	\begin{array}{c} {\begin{pmatrix} I_0 \\ O_0  \\ \end{pmatrix}} \end{array} && \begin{array}{c} {\begin{pmatrix} I_2 \\ O_2  \\ \end{pmatrix}} \end{array}
	\arrow["{(\phi_{01}, {\phi}^{\flat}_{01})}", shift left=2, from=1-1, to=1-3]
	\arrow[shift right=2, from=1-1, to=1-3]
	\arrow["{(f_0, f_{0}^{\sharp})}"', shift right=2, from=1-1, to=3-1]
	\arrow[shift right=2, from=1-3, to=3-3]
	\arrow[shift right=2, from=3-1, to=1-1]
	\arrow[shift left=2, from=3-1, to=3-3]
	\arrow["{(g_{01}, g^{\flat}_{01})}"', shift right=2, from=3-1, to=3-3]
	\arrow["{(f_{1}, f_{1}^{\sharp})}"', shift right=2, from=3-3, to=1-3]
	\arrow["{(\phi_{02}, {\phi}^{\flat}_{02})}", shift left=2, from=4-1, to=4-3]
	\arrow[shift right=2, from=4-1, to=4-3]
	\arrow["{(f_{0}, f_{0}^{\sharp})}"', shift right=2, from=4-1, to=6-1]
	\arrow[shift right=2, from=4-3, to=6-3]
	\arrow[shift right=2, from=6-1, to=4-1]
	\arrow[shift left=2, from=6-1, to=6-3]
	\arrow["{(g_{02}, g^{\flat}_{02})}"', shift right=2, from=6-1, to=6-3]
	\arrow["{(f_{2}, f_{2}^{\sharp})}"', shift right=2, from=6-3, to=4-3]
\end{tikzcd}\]

	Let $r_0: \widetilde{T} \rightarrow T$, $r_1: \widetilde{S}_1 \rightarrow S_1$ and $r_2: \widetilde{S}_2 \rightarrow S_2$ denote the deterministic structure maps.
	
	\item Let $(g_{012}, g_{012}^{\flat}): {\begin{pmatrix} I_0 \\ O_0  \\ \end{pmatrix}} \rightrightarrows {\begin{pmatrix} I_1 \otimes I_2 \\ O_1 \otimes O_2  \\ \end{pmatrix}}$ be an $x$-morphism whose projections are equal to $(g_{01}, g_{01}^{\flat})$ and $(g_{02}, g_{02}^{\flat})$ respectively.
	
	\item Assume that $I_0$ is isomorphic to the monoidal unit, and that the map $f_0^{\sharp}:  T \otimes I_0 \simeq T \rightarrow \widetilde{T}$ is an isomorphism in $\CC$.

	Then, we define the composite of $s_1$ and $s_2$ \emph{over} $(g_{012}, g_{012}^{\flat})$, which we denote $s_1 \nabla_{g_{012}, g_{012}^{\flat}} s_2$ or $s_1 \nabla s_2$, as the following $xy$-square:

\[\begin{tikzcd}
	\begin{array}{c} {\begin{pmatrix} \widetilde{T} \\ T  \\ \end{pmatrix}} \end{array} && \begin{array}{c} {\begin{pmatrix} \widetilde{S}_1 \otimes \widetilde{S}_2\\ S_1 \otimes S_2  \\ \end{pmatrix}} \end{array} \\
	& {s_{1}\nabla s_2} \\
	\begin{array}{c} {\begin{pmatrix} I_0 \\ O_0  \\ \end{pmatrix}} \end{array} && \begin{array}{c} {\begin{pmatrix} I_1\otimes I_2 \\ O_1\otimes O_2  \\ \end{pmatrix}} \end{array}
	\arrow["{(\phi_{012}, {\phi}^{\flat}_{012})}", shift left=2, from=1-1, to=1-3]
	\arrow[shift right=2, from=1-1, to=1-3]
	\arrow["{(f_0, f_{0}^{\sharp})}"', shift right=2, from=1-1, to=3-1]
	\arrow[shift right=2, from=1-3, to=3-3]
	\arrow[shift right=2, from=3-1, to=1-1]
	\arrow[shift left=2, from=3-1, to=3-3]
	\arrow["{(g_{012}, g^{\flat}_{012})}"', shift right=2, from=3-1, to=3-3]
	\arrow["{(f_{1}\otimes f_2, f_{1}^{\sharp}\otimes f_{2}^{\sharp})}"', shift right=2, from=3-3, to=1-3]
\end{tikzcd}\]

\begin{enumerate}
	\item For $i= 1, 2$, let $t_i:  T \otimes I_0 \rightarrow 
	S_i \otimes I_1 \otimes I_2 \otimes  O_1 \otimes O_2$ be the conditional product over $I_i \otimes O_i$ of the maps $(s_i; copy_{S_i}; f_i):  T \otimes I_0 \rightarrow 
	 S_i \otimes I_i \otimes O_i$ and $(f_0 ; g_{012}^{\flat}):  T \otimes I_0 \rightarrow I_1 \otimes I_2 \otimes O_1 \otimes O_2$.
	
	Then, let $t:  T \otimes I_0 \rightarrow   S_1 \otimes S_2 \otimes I_1 \otimes I_2 \otimes O_1 \otimes O_2$ be the conditional product over $I_1 \otimes I_2 \otimes O_1 \otimes O_2$ of the maps $t_i$.
	\item The morphism $s_1 \nabla s_2:  T \otimes I_0 \rightarrow   S_1 \otimes S_2 \otimes I_1 \otimes I_2$ is then defined as the composite $t; del_{O_1 \otimes O_2}$.
	
	\item Let us now define the morphisms $\phi_{012}: T \rightarrow S_1 \otimes S_2$ and $\phi_{012}^{\flat}: \widetilde{T} \rightarrow \widetilde{S}_1 \otimes \widetilde{S}_2$. The morphism $\phi_{012}$ is defined as the composite $$T \xrightarrow{\sim}  T \otimes I_0 \xrightarrow{s_1 \nabla s_2}   S_1 \otimes S_2 \otimes I_1 \otimes I_2 \xrightarrow{\pi} S_1 \otimes S_2.$$

	The morphism $\phi_{012}^{\flat}: \widetilde{T} \rightarrow \widetilde{S}_1 \otimes \widetilde{S}_2$ is defined as the following composite:
	 $$\widetilde{T} \xrightarrow{r_0} T \xrightarrow{\sim}  T \otimes I_0 \xrightarrow{s_1 \nabla s_2}   S_1 \otimes S_2 \otimes I_1 \otimes I_2 \xrightarrow{f_1^{\sharp} \otimes f_{2}^{\sharp}} \widetilde{S}_1 \otimes \widetilde{S}_2.$$
\end{enumerate}

\end{itemize}
	
\end{definition}

\begin{claim}
	This composition is well-defined, i.e. these maps define an $xy$-square in $ArenaSys^{\mathrm{Moore}}_{\CC}$.
\end{claim}
\begin{proof}
	
	\begin{itemize}
		\item Let us first show that the following square commutes:

\[\begin{tikzcd}
	{\widetilde{T}} & {\widetilde{S}_1 \otimes \widetilde{S}_2} \\
	T & {S_1 \otimes S_2}
	\arrow["{\phi_{012}^{\flat}}", from=1-1, to=1-2]
	\arrow["{r_0}"', from=1-1, to=2-1]
	\arrow["{r_1 \otimes r_2}", from=1-2, to=2-2]
	\arrow["{\phi_{012}}"', from=2-1, to=2-2]
\end{tikzcd}\]

We claim that the following diagram commutes:

\[\begin{tikzcd}
	{\widetilde{T}} &&&&& {\widetilde{S}_1 \otimes \widetilde{S}_2} \\
	& T \\
	&& {{ T \otimes I_0}} && {{S_1 \otimes S_2 \otimes I_1 \otimes I_2}} \\
	\\
	T &&&&& {S_1 \otimes S_2}
	\arrow["{\phi_{012}^{\flat}}", from=1-1, to=1-6]
	\arrow["{r_0}"{description}, from=1-1, to=2-2]
	\arrow["{r_0}"', from=1-1, to=5-1]
	\arrow["{r_1 \otimes r_2}", from=1-6, to=5-6]
	\arrow["\sim"{description}, from=2-2, to=3-3]
	\arrow["{s_1 \nabla s_2}", from=3-3, to=3-5]
	\arrow["{\sigma; (f_1^{\sharp} \otimes f_2^{\sharp})}"{description}, from=3-5, to=1-6]
	\arrow["\pi", from=3-5, to=5-6]
	\arrow["\sim", from=5-1, to=3-3]
	\arrow["{\phi_{012}}"', from=5-1, to=5-6]
\end{tikzcd}\]

The top and bottom polygons commute by definition of the maps $\phi_{012}$ and $\phi_{012}^{\flat}$. The rightmost triangle commutes thanks to the properties of the maps $f_i^{\sharp}$ and $r_i$, i.e. that fact that ${\begin{pmatrix} \widetilde{S}_i \\ S_i  \\ \end{pmatrix}} \leftrightarrows {\begin{pmatrix} I_i \\ O_i  \\ \end{pmatrix}}$ are $y$-morphisms, for the structure maps $r_i: \widetilde{S}_i \rightarrow S_i$.

\item Then, let us prove that the square below commutes:

\[\begin{tikzcd}
	{{T \otimes I_0}} && {{S_1 \otimes S_2 \otimes I_1 \otimes I_2}} \\
	\\
	{I_0 \otimes O_0} && {I_1 \otimes I_2 \otimes O_1 \otimes O_2}
	\arrow["{s_1 \nabla s_2}", from=1-1, to=1-3]
	\arrow["{\pi; f_0 ; \sigma}"', from=1-1, to=3-1]
	\arrow["{\pi; (f_1 \otimes f_2) ; \sigma}", from=1-3, to=3-3]
	\arrow["{g_{012}^{\flat}}"', from=3-1, to=3-3]
\end{tikzcd}\]

We claim that the following diagram commutes:

\[\begin{tikzcd}
	{{T \otimes I_0}} && {{S_1 \otimes S_2 \otimes I_1 \otimes I_2}} \\
	\\
	& {{S_1 \otimes S_2 \otimes I_1 \otimes I_2 \otimes O_1 \otimes O_2}} \\
	\\
	{I_0 \otimes O_0} && {I_1 \otimes I_2 \otimes O_1 \otimes O_2}
	\arrow["{s_1 \nabla s_2}", from=1-1, to=1-3]
	\arrow["t", from=1-1, to=3-2]
	\arrow["{\pi; f_0 ; \sigma}"', from=1-1, to=5-1]
	\arrow["{\pi; (f_1 \otimes f_2) ; \sigma}", from=1-3, to=5-3]
	\arrow["\pi"{description}, from=3-2, to=5-3]
	\arrow["{g_{012}^{\flat}}"', from=5-1, to=5-3]
\end{tikzcd}\]

Indeed, the bottom-left polygon commutes by marginalisation of the map $t$, which was defined as a conditional product. The top-right polygon commutes thanks to Lemma \ref{lemma_CP_and_regen}.

\item Now, commutativity of this square implies another commutativity property required for $xy$-squares; we claim that the diagram below commutes:

\[\begin{tikzcd}
	T &&&& {S_1 \otimes S_2} \\
	& {{T \otimes I_0}} && {{S_1 \otimes S_2 \otimes I_1 \otimes I_2}} \\
	\\
	& {I_0 \otimes O_0} && {I_1 \otimes I_2 \otimes O_1 \otimes O_2} \\
	{O_0} &&&& {O_1 \otimes O_2}
	\arrow["{\phi_{012}}", from=1-1, to=1-5]
	\arrow["\sim", from=1-1, to=2-2]
	\arrow["{f_0}"', from=1-1, to=5-1]
	\arrow["{f_1 \otimes f_2}", from=1-5, to=5-5]
	\arrow["{s_1 \nabla s_2}", from=2-2, to=2-4]
	\arrow["{\pi; f_0 ; \sigma}"', from=2-2, to=4-2]
	\arrow["\pi", from=2-4, to=1-5]
	\arrow["{\pi; (f_1 \otimes f_2) ; \sigma}", from=2-4, to=4-4]
	\arrow["{g_{012}^{\flat}}"', from=4-2, to=4-4]
	\arrow["\pi"', from=4-2, to=5-1]
	\arrow["\pi", from=4-4, to=5-5]
	\arrow["{g_{012}}"', from=5-1, to=5-5]
\end{tikzcd}\]

\item Finally, it remains to show that the maps $s_1 \nabla s_2$ and $\phi_{012}^{\flat}$ are compatible with the (deterministic) maps $f_0^{\sharp}$ and $f_1^{\sharp} \otimes f_2^{\sharp}$, i.e. that the following square commutes:

\[\begin{tikzcd}
	{{T \otimes I_0}} && {{S_1 \otimes S_2 \otimes I_1 \otimes I_2}} \\
	\\
	{\widetilde{T}} && {\widetilde{S}_1 \otimes \widetilde{S}_2}
	\arrow["{s_1 \nabla s_2}", from=1-1, to=1-3]
	\arrow["{f_0^{\sharp}}"', from=1-1, to=3-1]
	\arrow["{\sigma; f_1^{\sharp} \otimes f_2^{\sharp}}", from=1-3, to=3-3]
	\arrow["{\phi_{012}^{\flat}}"', from=3-1, to=3-3]
\end{tikzcd}\]

This holds essentially by definition of the map $\phi_{012}^{\flat}: \widetilde{T} \rightarrow \widetilde{S}_1 \otimes \widetilde{S}_2$. Said definition corresponds to commutativity of the bottom polygon in the commutative diagram below: 
\[\begin{tikzcd}
	{{T \otimes I_0}} &&& {{S_1 \otimes S_2 \otimes I_1 \otimes I_2}} \\
	& {T \otimes I_0} & {{S_1 \otimes S_2 \otimes I_1 \otimes I_2}} \\
	& T \\
	{\widetilde{T}} &&& {\widetilde{S}_1 \otimes \widetilde{S}_2}
	\arrow["{s_1 \nabla s_2}", from=1-1, to=1-4]
	\arrow["id"', from=1-1, to=2-2]
	\arrow["{f_0^{\sharp}}"', from=1-1, to=4-1]
	\arrow["{\sigma; f_1^{\sharp} \otimes f_2^{\sharp}}"{description}, from=1-4, to=4-4]
	\arrow["{s_1 \nabla s_2}", from=2-2, to=2-3]
	\arrow["id"{description}, from=2-3, to=1-4]
	\arrow["{\sigma; f_1^{\sharp} \otimes f_2^{\sharp}}", from=2-3, to=4-4]
	\arrow["\sim", from=3-2, to=2-2]
	\arrow["{r_0}", from=4-1, to=3-2]
	\arrow["{\phi_{012}^{\flat}}"', from=4-1, to=4-4]
\end{tikzcd}\]

The leftmost polygon in this diagram commutes thanks to the hypothesis that $f_0^{\sharp}: T \otimes I_0 \rightarrow \widetilde{T}$ is an isomorphism. \end{itemize} Thus, we have checked all the properties required for the maps $\phi_{012}$, $\phi_{012}^{\flat}$ and $s_1 \nabla s_2$ to define an $xy$-square.
\end{proof}
\begin{claim}
	The projections of this square are the original squares $s_1$ and $s_2$. In other words, we have $(s_1 \nabla s_2) | \pi_i = s_i$, for $i= 1, 2$.
\end{claim}
\begin{proof}
	\begin{itemize}
		\item First note that, by hypothesis on the $x$-morphism $(g_{012}, g_{012}^{\flat}): {\begin{pmatrix} I_0 \\ O_0  \\ \end{pmatrix}} \rightrightarrows {\begin{pmatrix} I_1 \otimes I_2 \\ O_1 \otimes O_2  \\ \end{pmatrix}}$, the composites with the projection $x$-morphisms are equal to the $x$-morphisms $(g_{0i}, g_{0i}^{\flat}): {\begin{pmatrix} I_0 \\ O_0  \\ \end{pmatrix}} \rightrightarrows {\begin{pmatrix} I_i \\ O_i  \\ \end{pmatrix}}$, for $i= 1,2$.
	
\item Then, for $i=1,2$, the composite $(s_1 \nabla s_2 ; \pi_{S_i \otimes I_i}):  T \otimes I_0$ is equal to $s_i$, because the conditional product $t:  T \otimes I_0 \rightarrow   S_1 \otimes S_2 \otimes I_1 \otimes I_2 \otimes O_1 \otimes O_2$ marginalizes to the map $t_i:  T \otimes I_0 \rightarrow S_i \otimes I_1 \otimes I_2 \otimes  O_1 \otimes O_2$, which in turn marginalizes to $s_i$.

\item	Now, let us show that the composite $(\phi_{012} ; \pi_i): T \rightarrow S_i$ equals $\phi_{0i}$, for $i= 1,2$. It suffices to check that the following diagram commutes:

\[\begin{tikzcd}
	T &&&& {S_1 \otimes S_2} \\
	& {T \otimes I_0} && {S_1 \otimes S_2 \otimes I_1 \otimes I_2} \\
	&&& {S_i \otimes I_i \otimes O_i} \\
	& {\widetilde{T}} && {S_i \otimes I_i} \\
	\\
	&&& {\widetilde{S}_i} \\
	T &&&& {S_i}
	\arrow["{\phi_{012}}", from=1-1, to=1-5]
	\arrow["\sim", from=1-1, to=2-2]
	\arrow["{id_T}"{description}, from=1-1, to=7-1]
	\arrow["{\pi_i}", from=1-5, to=7-5]
	\arrow["{s_1 \nabla s_2}", from=2-2, to=2-4]
	\arrow["{t_i}"{description}, from=2-2, to=3-4]
	\arrow["{f_0^{\sharp}}"{description}, from=2-2, to=4-2]
	\arrow["{s_i}"{description}, from=2-2, to=4-4]
	\arrow["\pi", from=2-4, to=1-5]
	\arrow["\pi", from=2-4, to=3-4]
	\arrow["\pi", from=3-4, to=4-4]
	\arrow["{\phi_{0i}^{\flat}}", from=4-2, to=6-4]
	\arrow["{r_0}", from=4-2, to=7-1]
	\arrow["{f_i^{\sharp}}"{description}, from=4-4, to=6-4]
	\arrow["\pi"{description}, from=4-4, to=7-5]
	\arrow["{r_i}", from=6-4, to=7-5]
	\arrow["{\phi_{0i}}", from=7-1, to=7-5]
\end{tikzcd}\]

	The leftmost part of the diagram commutes because $f_0^{\sharp}:  T \otimes I_0 \rightarrow \widetilde{T}$ was assumed to be an isomorphism, and it is also a section of the composite $\widetilde{T} \xrightarrow{r_0} T \xrightarrow{\sim}  T \otimes I_0$.
	

\item	To conclude, let us show that $\phi_{012}^{\flat}; \pi_i = \phi_{0i}^{\flat}$, as morphisms $\widetilde{T} \rightarrow \widetilde{S}_i$, for $i=1,2$. We claim that the following diagram commutes:

\[\begin{tikzcd}
	{\widetilde{T}} &&&&& {\widetilde{S}_1 \otimes \widetilde{S}_2} \\
	& T \\
	&& {T \otimes I_0} && {S_1 \otimes S_2 \otimes I_1 \otimes I_2} \\
	&&&& {S_i \otimes I_i \otimes O_i} \\
	&&&& {S_i \otimes I_i} \\
	\\
	\\
	{\widetilde{T}} &&&&& {\widetilde{S}_i}
	\arrow["{\phi_{012}^{\flat}}", from=1-1, to=1-6]
	\arrow["{r_0}", from=1-1, to=2-2]
	\arrow["{id_{\widetilde{T}}}"', from=1-1, to=8-1]
	\arrow["{\pi_i}", from=1-6, to=8-6]
	\arrow["\sim", from=2-2, to=3-3]
	\arrow["{s_1 \nabla s_2}", from=3-3, to=3-5]
	\arrow["{t_i}"{description}, from=3-3, to=4-5]
	\arrow["{s_i}"{description}, from=3-3, to=5-5]
	\arrow["{\sigma;(f_1^{\sharp}\otimes f_2^{\sharp})}"{description}, from=3-5, to=1-6]
	\arrow["\pi", from=3-5, to=4-5]
	\arrow["\pi", from=4-5, to=5-5]
	\arrow["{f_i^{\sharp}}", from=5-5, to=8-6]
	\arrow["{\phi_{0i}^{\flat}}"', from=8-1, to=8-6]
\end{tikzcd}\]

	Indeed, the top polygon commutes by definition of $\phi_{012}^{\flat}$. The bottom-left one commutes if precomposed with the map $f_0^{\sharp}:  T \otimes I_0 \rightarrow \widetilde{T}$, by assumption on $\phi_{0i}^{\flat}$; since $f_0^{\sharp}$ was assumed to be an isomorphism, this implies that the bottom-left polygon commutes.  The triangles involving $s_1 \nabla s_2$, $t_i$ and $s_i$ commute by marginalization of conditional products. Finally, the rightmost polygon commutes because of basic properties of the projections $\pi$.
\end{itemize}

This concludes the proof. \end{proof}

\begin{remark}
	The simplest example of $y$-morphism $({\begin{pmatrix} \widetilde{T} \\ T  \\ \end{pmatrix}} \leftrightarrows {\begin{pmatrix} I_0 \\ O_0  \\ \end{pmatrix}})$ that satisfies the assumptions is that where $I_0 = O_0 = \widetilde{T} = T = 1$, the monoidal unit of $\CC$, and all maps are the unique (deterministic) automorphism of $1$. In that case, the operation $\nabla$ can be used to define \emph{trajectories of tensor products of systems}.
\end{remark}